\documentclass[11pt]{amsart}
\usepackage{amsmath,amsthm,amssymb,amsfonts}
\usepackage{color}

\usepackage{hyperref}


\makeindex

\newcommand{\cleanup}[1]{\marginpar{ \tiny #1}}

\newcommand{\finalexponent}{\kappa_0}
\newcommand{\stab}{\mathrm{stab}}

\newcommand{\LHS}{H_w \cap K_w[m]}
\newcommand{\ad}{\mathrm{ad} \,}
\newcommand{\RHS }{ \mathfrak{h}_w \cap \mathfrak{g}_w[m]}

\newcommand{\nbhd}{\mathcal{N}}

\newcommand{\data}{\mathscr{D}}
\newcommand{\levelbound}{L}

\newcommand{\degbound}{N^2}

\newcommand{\supp}{\mathrm{supp}}

\newcommand{\h}{\mathfrak{h}}
\newcommand{\g}{\mathfrak{g}}
\newcommand{\SO}{\mathrm{SO}}

\newcounter{consta}
\renewcommand{\theconsta}{{\kappa_{\arabic{consta}}}}
\newcounter{constb}[section]

\newcounter{constc}[section]
\renewcommand{\theconstc}{{c_{\arabic{constc}}}}
\newcommand{\consta}{\refstepcounter{consta}\theconsta}

\newcommand{\constc}{\refstepcounter{constc}\theconstc}

\newcommand{\Sob}{\mathcal{S}}

\renewcommand{\r}{\mathfrak{r}}

\newcommand{\Ad}{\mathrm{Ad}}

\newcommand{\disc}{\mathrm{disc}}

\newcommand{\Lie}{\mathrm{Lie}}

\newcommand{\C}{\mathbb{C}}

\newcommand{\Siegel}{\Sieg}

\newcommand{\vol}{\operatorname{vol}}

\newcommand{\G}{\mathbf{G}}

\newcommand{\Sieg}{\mathfrak{S}}
\newcommand{\height}{\mathrm{ht}}

\newcommand{\m}{\underline{m}}

\newcommand{\PGL}{\mathrm{PGL}}
\newcommand{\PSL}{\mathrm{PSL}}

\newcommand{\SL}{\mathrm{SL}}

\newcommand{\N}{\mathbf{N}}
\newcommand{\R}{\mathbb{R}}
\renewcommand{\H}{\mathbf{H}}

\newcommand{\Z}{\mathbb{Z}}
\newcommand{\Q}{\mathbb{Q}}

\newcommand{\adele}{\mathbb{A}}

\newcommand{\newletter}{M}

\DeclareMathAlphabet{\mathpzc}{OT1}{pzc}{m}{it}
\DeclareFontFamily{OT1}{rsfs}{}
\DeclareFontShape{OT1}{rsfs}{n}{it}{<-> rsfs10}{}
\DeclareMathAlphabet{\mathscr}{OT1}{rsfs}{n}{it}

\numberwithin{equation}{section}

\swapnumbers
\newtheorem{thm}[subsection]{Theorem}

\newtheorem*{cor*}{Corollary}
\newtheorem{lemma}[subsection]{Lemma}
\newtheorem*{lem}{Lemma}
\newtheorem{propo}[subsection]{Proposition}

\newtheorem{thm*}{Theorem}[]
\newtheorem*{prop}{Proposition}

\newtheorem*{claim}{Claim}
\newcommand{\bbr}{\mathbb{R}}
\newcommand{\bbq}{\mathbb{Q}}

\newcommand{\Hcal}{\mathcal{H}}

\newcommand{\Pcal}{\mathcal{P}}
\newcommand{\Mcal}{\mathcal{M}}

\newcommand{\Ecal}{\mathcal{E}}

\newcommand{\Hfrak}{\mathfrak{H}}
\newcommand{\Gfrak}{\mathfrak{G}}

\newcommand{\hfrak}{\mathfrak{h}}
\newcommand{\rfrak}{\mathfrak{r}}
\newcommand{\gfrak}{\mathfrak{g}}

\newcommand{\order}{\mathfrak{o}}
\newcommand{\places}{\Sigma}
\newcommand{\plw}{{S}}
\newcommand{\pl}{p}
\newcommand{\red}{{\rm red}}
\newcommand{\be}{\begin{equation}}
\newcommand{\ee}{\end{equation}}

\newcommand{\tSL}{\theta_w(\SL_2)}

\newcommand{\Xc}{X_{{\rm cpt} }}
\newcommand{\tH}{\widetilde{\H}}

\newcommand{\thsa}{{\widetilde{H}_{\rm sa}}}
\newcommand{\sa}{{\rm sa}}
\newcommand{\tK}{{K}^*}
\newcommand{\Lbf}{\mathbf{L}}
\newcommand{\Gpp}{{\pi^+}}
\newcommand{\temp}{{M}}

\newcommand{\jmap}{{\rm j}}
\newcommand{\flowscale}{{m}}
\newcommand{\Av}{\operatorname{Av}_L}

\begin{document}

\title{Effective equidistribution and property $(\tau)$}

\author{M. Einsiedler}

\address{M.E.:  Departement Mathematik, ETH Z\"urich, R\"amistrasse 101, 8092, Z\"urich, Switzerland}
\email{manfred.einsiedler@math.ethz.ch}
\thanks{M.E. acknowledges support by the SNF (Grant 200021-127145 and 200021-152819).}

\author{G. Margulis}

\address{G.M.: Yale University,
Mathematics Dept.,
PO Box 208283, New Haven, CT 06520-8283, USA}
\email{gregorii.margulis@yale.edu}
\thanks{G.M. acknowledges support by the NSF (Grant 1265695).}

\author{A. Mohammadi}
\address{A.M.:  Mathematics Department,  UC San Diego, 9500 Gilman Dr, La Jolla, CA 92093, USA}
\email{ammohammadi@ucsd.edu}
\thanks{A.M. acknowledges support from the NSF and Alfred P.\ Sloan Research Fellowship.}

\author{A. Venkatesh}
\address{A.V.: Department of Mathematics, Stanford University, Stanford, CA, 94305, USA}
\email{akshay@math.stanford.edu}
\thanks{A.V. acknowledges support from the NSF and the Packard foundation.}

\begin{abstract}
 
We prove a quantitative equidistribution statement for
 adelic homogeneous subsets
 whose stabilizer is maximal and semisimple. 
 Fixing the ambient space, the statement is uniform in all parameters. 

 We explain how this implies certain equidistribution theorems which,  {\em even in a qualitative form}, 
 are not accessible to measure-classification theorems. 
As another application, we describe another proof of 
 property $(\tau)$ for arithmetic groups.
\end{abstract}

\maketitle

\tableofcontents
 
\section{Introduction} \label{tauproof}

\subsection{Homogeneous sets and measures}\label{sec:intro}

Number theoretical problems often relate to orbits of subgroups (periods) and so can be attacked
by dynamical methods. To be more specific let us recall the following terminology.

Let~$X=\Gamma\backslash G$ be a homogeneous space defined by a lattice~$\Gamma<G$
in a locally compact group~$G$. Note that any subgroup~$H<G$ acts naturally
by right multiplication on~$X$, sending~$h\in H$ to the map~$x\in X\mapsto xh^{-1}$. 
We will refer to~$H$ as the acting subgroup. 
A {\em homogeneous (probability) measure} on $X$ is, by definition, a probability measure $\mu$ that is supported
on a single closed orbit $Y=\Gamma g H_Y$ of its stabilizer $H_Y=\mathrm{Stab}(\mu)$. 
A {\em homogeneous set} is the support of some homogeneous probability measure.  In what follows, we shall deal only
with probability measures and shall consequently simply refer to them as homogeneous measures.

Ratner's celebrated measure classification theorem~\cite{Rat} and the so called linearization techniques 
(cf.\ \cite{DM} and \cite{Mozes-Shah}) imply in the case where~$G$
is a real Lie group that, given a sequence of homogeneous probability measures
$\{\mu_i\}$ with the property that $H_i=\mathrm{Stab}(\mu_i)$ contain ``enough'' unipotents,
any weak$^*$ limit of $\{\mu_i\}$ is also homogeneous, where often
the stabilizer of the weak$^*$ limit has bigger dimension than $H_i$ for 
every $i$. This has been extended also to quotients
of~$S$-algebraic groups (see~\cite{Rat-S-alg}, \cite{MT}, \cite[App.~A]{EV} and \cite[Sect.~6]{Gorodnik-Oh})
for any finite set~$S$ of places. We note that the latter allow similar corollaries (see \cite{Gorodnik-Oh}) 
for adelic quotients {\em if} the acting groups~$H_i$ contain unipotents
at one and the same place for all~$i$ -- let us refer to this as a {\em splitting condition}.
These theorems have found many applications in number theory (see e.g.~\cite{Eskin-Margulis-Mozes}, \cite{EV},
and~\cite{Gorodnik-Oh} to name a few examples), but are (in most cases) ineffective.

Our aim in this paper is  to present one instance
of an adelic result which is entirely quantitative in terms of the ``volume" of the orbits, 
and is in many cases not accessible, even in a non-quantitative form, by the measure classification theorem and linearization techniques
(as we will dispense with the splitting condition).
A special case of this result
will recover ``property $(\tau)$'' (but with weaker exponents) from the theory of automorphic forms.

\subsection{Construction of homogeneous measures}\label{construction-MASH}
In the following~$F$ will\index{F@$F$, the number field} 
always denote a number field,~$\adele$ will denote\index{a@$\adele, \adele_f$, adeles and finite adeles over~$F$}
the ring of adeles over~$F$, and~$\G$ will be a connected semisimple algebraic $F$-group.
We will consider the homogeneous space\index{X@$X=\G(F)\backslash\G(\adele)$, the ambient space} $X = \G(F) \backslash \G(\adele)$
defined by the\index{G@$G=\G(\adele)$, the ambient group} group~$G=\G(\adele)$ of~$\adele$-points of~$\G$.

We normalize the Haar measure~$\operatorname{vol}_G$ on~$G$ 
so\index{v@$\vol_G$, Haar measure of the ambient group and space}  that the induced measure
on~$X$ (again denoted by~$\operatorname{vol}_G$) is a probability measure.
Let us fix the following data $\data=(\mathbf{H},\iota,g)$\index{data@$\data=(\mathbf{H},\iota,g)$, triple defining a MASH} 
consisting of 
\begin{itemize}
\item[(1)] an $F$-algebraic group\index{H@$\mathbf{H}$, the algebraic group giving rise to the MASH~$Y$} $\mathbf{H}$
such that~$\H(F)\backslash \H(\adele)$ has finite volume, 
\item[(2)] an algebraic homomorphism $\iota: \mathbf{H} \rightarrow \mathbf{G}$ 
defined over $F$ with finite kernel, and
\item[(3)] an element $g \in G$. 
\end{itemize}

To this data, we may associate a homogeneous set 
\[
 Y_{\data} := \iota(\mathbf{H}(F) \backslash \mathbf{H}(\adele))g \subset X
\] 
and the {\em algebraic} homogeneous measure $\mu_{\data}$\index{mu@$\mu_\data, Y_\data$, 
algebraic homogeneous measure and homogeneous set} given by the push-forward, under the map $x \mapsto \iota(x) g,$ 
of the normalized Haar measure on $\mathbf{H}(F) \backslash \mathbf{H}(\adele)$. 
We refer to such a set $Y$ as an {\em algebraic} homogeneous set; we say it is simple, semisimple, simply connected, etc.\
according to whether the algebraic group~$\H$ is so,
and we say it is maximal if $\iota(\mathbf{H})\subset\mathbf{G}$ is a maximal\footnote{Here by maximality 
we mean maximal as an algebraic group over the algebraic closure of $F$.} proper subgroup. 

 { Our main theorem
 will discuss the equidistribution of maximal  semisimple simply connected homogeneous sets.
 The assumption that $\H$ is simply connected can be readily removed, as we explain in  \S \ref{s;not-sc}.}

 \subsection{The intrinsic volume of a homogeneous set}\label{intro-volume}

What does it mean for a homogeneous set to be ``large''?

If $H=\mathrm{Stab}(\mu)$ is fixed, then one may define
the volume of an $H$-orbit $xH$ using a fixed Haar measure on $H$.
However, as we will allow the acting group $H$ to vary we give another
reasonably intrinsic way of measuring this, as we now explain.

Let $Y=Y_\data$ be an algebraic homogeneous set with corresponding probability measure $\mu_\data$
and associated group
$H_\data=g^{-1}\iota(\mathbf{H}(\adele))g$.\index{H@$H_\data$, acting group associated to the data~$\data$}  
We shall always consider $H_\data$ as equipped 
with that measure, denoted by $m_\data,$ which projects
to $\mu_\data$ under the orbit map. 

Fix an open subset $\Omega_0 \subset \G(\adele)$\index{Omega@$\Omega_0=\prod_{v\in\places}\Omega_v$, open set in~$G$ to normalize the volume
of~$H$-orbits} that contains the identity and has compact closure.
Set 
\begin{equation} \label{volume}
\vol(Y) := m_\data ( H_\data \cap \Omega_0)^{-1},
\end{equation} 
this should be regarded as a measure of the ``volume'' of $Y $.
It depends on $\Omega_0$, but the notions arising from two different choices of $\Omega_0$
are comparable to each other, in the sense that their ratio is bounded above and below, see \S\ref{new-sec-on-volume}. 
Consequently, we do not explicate the choice of $\Omega_0$ in the notation.

The above notion of the volume of an adelic orbit is strongly related to 
the discriminant of the orbit, see Appendix~\ref{sec;discrim}. The theorem below could also
be phrased using this notion of arithmetic height or complexity instead of the volume.

\subsection{Notation for equidistribution in~$X$}\label{sec:equi0no}

If in addition~$\G$ is simply connected and~$\iota(\mathbf{H})$ is a maximal subgroup of~$\G$ we will show
in this paper that a homogeneous measure~$\mu_\data$ as above is almost equidistributed 
if it already has large volume. Dropping the assumption that~$\G$ is simply connected
we need the following notation: Let~$\G(\adele)^+$ denote\index{G@$\G(\adele)^+$, $\G(F_v)^+$, the image
of the simply connected cover}
the image of the simply connected cover (see also~\S\ref{notations}).
Using this we define the decomposition\index{L@$L^2_0(X,\vol_G)$, $L^2(X,\vol_G)^{\G(\adele)^+}$,
orthogonal decomposition of~$L^2$}
\be\label{e;orth-const}
L^2(X,\vol_G)=L^2_0(X,\vol_G)\oplus L^2(X,\vol_G)^{\G(\adele)^+},
\ee
where $L^2(X,\vol_G)^{\G(\adele)^+}$ denotes the space of $\G(\adele)^+$-invariant 
functions and $L^2_0(X,\vol_G)$ is the orthogonal complement of $L^2(X,\vol_G)^{\G(\adele)^+}$.
Note that if $\G$ is simply connected,
then $\G(\adele)=\G(\adele)^+$ and
$L^2(X,\vol_G)^{\G(\adele)^+}$ is the space of constant functions. 

The group $\G(\adele)^+$ is a closed, normal subgroup of $\G(\adele),$ see e.g.~\cite[p.~451]{PR}.
Therefore, the subspaces introduced in~\eqref{e;orth-const} are $\G(\adele)$-invariant. Let $\Gpp:L^2(X,\vol_G)\to L^2(X,\vol_G)^{\G(\adele)^+}$ 
denote the orthogonal projection. 
Let $C_c^\infty(X)$ denote the space of smooth compactly supported functions on $X$, see \S\ref{sec;sob-intro} for a discussion. Finally let us note that given $f\in C_c^\infty(X)$ 
\[
\Gpp f(x)=\int_{X} f\operatorname{d}\!\mu_{x\G(\adele)^+}.
\]
 It is worth noting that $\Gpp f$ is a finite-valued function for all $f\in C_c^\infty(X).$

\begin{thm}[Equidistribution of adelic periods]\label{adelic}
Let $Y_\data$ be a maximal algebraic semisimple homogeneous set
arising from~$\data=(\mathbf{H},\iota,g)$.
Furthermore, assume that $\H$ is simply connected.
Then
$$\left| \int_{Y_\data} f\operatorname{d}\!\mu_\data-\Gpp f(y)\right| \ll  \vol(Y_\data)^{-\finalexponent}
\Sob(f)\quad\text{ for all } f \in C^{\infty}_c(X),$$
where $y\in Y_\data$ is arbitrary, $\Sob( f)$ denotes a certain adelic Sobolev norm (see~\S\ref{sec;sob-intro} and Appendix A),
and $\finalexponent$ is a positive constant which depends only on $[F:\Q]$ and 
$\dim\G$.\index{kaa0@$\finalexponent$, exponent in Theorem~\ref{adelic}}
\end{thm}

Below we will abbreviate the assumption that~$Y_\data$ 
is a maximal algebraic semisimple homogeneous set in the theorem
by saying that~$Y_\data$ is a MASH set (resp.~$\mu_\data$ is a MASH measure).\index{MASH set, a maximal algebraic semisimple homogeneous set}
We stated the above theorem under the natural assumption that $\H$ is simply connected
(but note that~$\iota(\H)$ may not be simply connected). In \S\ref{s;not-sc} we discuss a formulation of the theorem without that assumption. 

Let us highlight two features of this theorem. Our method relies on a uniform version of Clozel's property~($\tau$) (see \cite{LC}, \cite[Thm.~1.11]{GMO}, and Section \ref{ss:history} for a summary of the history).
However,  it also allows us to give an  independent proof of Clozel's part of the proof of property $(\tau)$ except for groups of type $A_1$ -- i.e., if we only  
suppose property $(\tau)$ for groups of type $A_1$, we can deduce property $(\tau)$ in all other cases as well as our theorem. 
 We will discuss this in greater detail in \S\ref{sec:tau}. 
The theorem also allows $\H$ to vary without any splitting condition (as e.g.~in~\cite[Thm.~1.7]{Gorodnik-Oh}), an application of this to quadratic forms is given in~\S\ref{sec:quadratic}.

\subsection{An overview of the argument}
To overcome the absence of a splitting condition we make crucial use of Prasad's volume formula in~\cite{Pr}
to find a small place where the acting group has good properties (see~\S\ref{s;good-place} and~\S\ref{ss;good-place}
for a summary). This is needed to make the dynamics at this place useful.

The dynamical argument uses unipotent flows 
(but we note that one could also give an argument using the mixing property). 
Assuming that the volume is large, we find by a pigeon-hole principle
nearby points that have equidistributing orbits. Using polynomial divergence of the unipotent flow we obtain
almost invariance under a transverse direction. By maximality and spectral gap on the ambient space
we conclude the equidistribution, see~\S\ref{proof}. 

The first difficulty is to ensure that one really can choose a place which is ``sufficiently small,'' relative to the size of the orbit. 
 Using~\cite{Pr} we establish a 
logarithmic
bound for the first useful (``good'') prime in terms of the volume -- see ~\S\ref{s;good-place}. 
 We also need to use~\cite{BorPr}
if~$\iota(\H)$ is not simply connected (as in that case the stabilizer of~$\mu_\data$ is
larger than~$H_\data=g^{-1} \iota(\H(\adele))g$ which affects the notion of volume).

The second difficulty is that we also need to know that there are many points
for which the unipotent orbit effectively equidistributes with respect to the measure in question. 
This effectivity also relies on spectral gap, but as the measure~$\mu_\data$ 
(and so its~$L^2$-space)
varies we need uniformity for this spectral gap. 
This is a uniform version of Clozel's property~$(\tau)$ (see Section~\ref{ss:history}).

After completion of this project M.E., R.~R\"{u}hr, and P.~Wirth worked out a special case~\cite{ERW} that goes slightly beyond
the setting of this paper. However, due to the concrete setting many of the 
difficult ingredients of this paper were not needed in~\cite{ERW}, which may make it more accessible for some readers.

\subsection{Uniform non-escape of mass}\label{non-divergence-adelic}
We also note the following corollary of the above which does not
seem to follow from the standard non-divergence results alone\footnote{The non-divergence estimates
for unipotent flows enter our proof (see Lemma~\ref{l;non-div}) but removing all effects
from the above mentioned ``splitting condition", resp.\ its absence, 
seems to require the equidistribution theorem (Theorem~\ref{adelic}).}.

\begin{cor*}
	Let~$X=\G(F)\backslash \G(\adele)$. Then 
	for every~$\epsilon>0$ there exists some compact~$X_\epsilon\subset X$
	such that~$\mu(X_\epsilon)>1-\epsilon$ for every MASH measure~$\mu$ on~$X$.
\end{cor*}

{ We will prove the corollary in \S\ref{sec:proof-adel-nondiv}.}

\subsection{Acknowledgement}
M.E.~and~A.M.~thank Richard Pink for enlightening discussions, 
and also thank the Max Planck Institute in Bonn, where some of the final writing of the paper took place, for its hospitality.

A.M.~thanks Alireza Salehi Golsefidy for enlightening discussions regarding several aspects of this project.
A.M.~also thanks the Forschungsinstitut f\"ur Mathematik at the ETH Z\"urich for its support.

 We also thank Hee Oh for her comments on an earlier version of this paper.
We thank Brian Conrad for helpful discussions and for the proof of the lemma in \S\ref{l;good-sl2}.
Finally we are grateful for the comments and suggestions made by the referees.

\section{Notation and preliminary statements}\label{sec:prem-notation}

\subsection{Notation} \label{notations}\label{proofnotn}

Let us recall that~$F$ denotes\index{F@$F$, the number field}  a number field.
Throughout the paper $\places$ denotes the set of places on $F;$ similarly let 
$\places_f$ and $\places_\infty$ denote the set of finite and infinite (archimedean) places respectively. 
\index{s@$\Sigma=\Sigma_f\cup\places_\infty$, the sets of places, finite and archimedean places of~$F$}


For each $v\in\places$, we denote by $F_v$ the completion of $F$ at $v$.
For $v\in\places_f$, we denote by $\order_{v}$\index{F@$F_v,\widehat{F_v}$, local field and maximal unramified extension}\index{v@$v,w$, places of~$F$} 
the maximal compact subring of the completion $F_v$ and let $\varpi_{v}$ 
be {a} uniformizer of $\order_{v}.$\index{o@$\order_v\ni\varpi_v$, the maximal compact subring and its uniformizer}
We let $\adele=\prod_{v\in\places}'F_v$ be the ring of adeles over $F$ 
and define $\adele_f=\prod_{v\in\places_f}'F_v$, where~$\prod'$ denotes
the restricted direct product with respect to the compact open
subgroups~$\order_v<F_v$ for~$v\in \places_f$.\index{a@$\adele, \adele_f$, adeles and finite adeles over~$F$}

For any finite place~$v\in\places_f$ we let $k_v=\order_{v}/\varpi_{v}\order_{v}$ 
be the residue field,\index{kv@$k_v,\widehat{k_v}$, residue field of~$F_v$ and its algebraic closure}  
and we set $q_v = \# k_v$.\index{q@$q_v$, cardinality of the residue field~$k_v$}
Let $|x|_v$ denote the absolute value on $F_v$ normalized so that $|\varpi_v|_v=1/q_v$. 
Finally let $\widehat{F_v}$ denote the maximal unramified extension of $F_v,$
we let $\widehat{\order_v}$ denote the ring of integers in $\widehat{F_v}$ and $\widehat{k_v}$
denotes the residue field of $\widehat{\order_v}.$ We note that $\widehat{k_v}$ is the algebraic
closure of $k_v.$ 

Fix $\G$ and $\H$ as in the introduction, and
let $\mathfrak{g}$ (resp. $\mathfrak{h}$) denote the Lie algebra of $\G$ (resp. $\H$);
they are equipped with compatible $F$-structures. We define~$G=\G(\adele)$\index{G@$G=\G(\adele)$, the ambient group}
and~$X=\G(F)\backslash\G(\adele)$.\index{X@$X=\G(F)\backslash\G(\adele)$, the ambient space} 

{ In this paper rank of an algebraic group refers to its absolute rank. If we want to refer to the 
rank of an algebraic group over a not necessarily algebraically closed field $E$, we will use {\em relative rank} or $E$-{\em rank}.}

For any $v\in\places$ let $\G(F_v)^+$ be the image of $\tilde{\G}(F_v)$, 
and $\G(\adele)^+$ be the image of $\tilde{\G}(\adele)$, 
where $\tilde\G$ is the simply connected cover of $\G.$\index{G@$\G(\adele)^+$, $\G(F_v)^+$, the image
of the simply connected cover} 
If {each $F_v$-almost simple factor of} $\G$ is $F_v$-isotropic, then $\G(F_v)^+$ is the subgroup generated by all unipotent elements;
it is worth mentioning that our notation is different from the usual notation in the anisotropic case.

Let $\rho: \G \rightarrow \SL_N$ be an embedding defined 
over $F.$\index{r@$\rho:\G\rightarrow\SL_N$, a fixed embedding over~$F$}
For any $v\in\places_f$,
we let $K_v := \rho^{-1}(\SL_N(\order_v))$
and $K_f = \prod_{{v}\in\places_f} K_v$.\index{kv@$K_v,K_v[m],K_f$, compact open subgroups (defined using~$\rho$)}
Set also 
\begin{equation} \label{Kvm definition}
K_v[m] := \mathrm{ker}(K_v \rightarrow \SL_N(\order_v/\varpi_v^m \order_v))
\end{equation}
for $m \geq 1$. It is convenient to write $K_v[0] := K_v$. 

We also set up the corresponding notions at the level of the Lie algebra $\mathfrak{g}$
of~$\G$. \index{g@$\mathfrak{g},\mathfrak{g}_v$, Lie algebra and Lie algebra over
local field~$F_v$}
For any~$v\in\Sigma$ we let~$\mathfrak{g}_v$ be the Lie algebra of~$\G$ over~$F_v$.
For~$v\in\Sigma_f$ we 
write $\mathfrak{g}_v[0]$ for the preimage of the $\order_v$-integral $N \times N$ matrices
under the differential $D\rho: \mathfrak{g} \rightarrow \mathfrak{sl}_N$.
More generally, we write $\mathfrak{g}_v[m]$ for the preimage of the matrices
all of whose entries have valuation at least $m$.\index{g@$\mathfrak{g}_v[m]$, compact open subgroup
of level~$m$ of~$\mathfrak g_v$ for~$v\in\Sigma_f$}

Throughout, $\red_v:\SL_N(\order_v)\rightarrow\SL_N(k_{v})$ 
denotes the reduction mod $\varpi_v$ map; 
similarly we consider reduction mod $\varpi_v$ for the Lie algebras, 
see~\cite[Ch.~3]{PR} for a discussion of reduction maps.\index{r@$\red_v$, reduction maps} 

For $g \in \G(F_v)$, we write $\|g\|$ for the largest absolute value of the matrix entries of 
$\rho(g)$ and $\rho(g)^{-1}$.\index{g@$\|g\|$, norm of~$g\in\G(F_v)$}

We let $\vol_G$ denote the volume measure on $G$\index{v@$\vol_G$, Haar measure of the ambient group and space} 
which is normalized so that it assigns mass 1 to the quotient $X=\G(F) \backslash \G(\adele)$.
We will also use the same notation for the induced Haar measure on~$X$.

The notation $A \ll B$, meaning ``there exists a constant $\constc\label{sample-c}>0$ so that $A \leq \ref{sample-c} B$'',
will be used; the implicit constant $\ref{sample-c}$ is permitted to depend on
$F$, $\G$, and $\rho$, but (unless otherwise noted) not on\index{1@$A\ll B,A\asymp B$, $A\asymp B^\star$, 
inequalities with implicit constants} 
anything else. We write~$A\asymp B$ if~$A\ll B\ll A$.
We will use~$\ref{sample-c},\constc,\ldots$ to denote constants 
depending on~$F$,~$\G$, and~$\rho$ (and their numbering is reset at 
the end of a section).\index{c@$\ref{sample-c}$, a constant depending on $F, \mathbf{G}$ and $\rho$}
If a constant (implicit or explicit) depends on another parameter or only on a certain
part of~$(F,\G,\rho)$, we will make this clear by writing e.g.~$\ll_\epsilon$,~$\constc(N)$, etc.

We also adopt the $\star$-notation from~\cite{EMV}:\index{A@$A^\star, A^{-\star}$, power of~$A$
with an implicit exponent} 
We write $B=A^{\pm\star}$ if $B=\constc A^{\pm\consta\label{sample-kappa}},$ 
where $\ref{sample-kappa}>0$\index{kaa1@$\ref{sample-kappa}$, a sample constant depending on~$\dim_F\G$ and~$[F:\Q]$} 
depends only on $\dim_F\G$ and $[F:\Q].$
Similarly one defines
$B\ll A^\star,$ $B\gg A^\star$.
Finally we also write~$A\asymp B^\star$ if~$A^\star\ll B\ll A^\star$ (possibly with different exponents).\index{1@$A\ll B,A\asymp B$, $A\asymp B^\star$, 
inequalities with implicit constants}

We fix a MASH-set $Y_\data$, arising\index{MASH set, a maximal algebraic semisimple homogeneous set}
from\index{mu@$\mu_\data, Y_\data$, algebraic homogeneous measure and homogeneous set} 
the data $\data = (\mathbf{H},\iota,  g)$\index{data@$\data=(\mathbf{H},\iota,g)$, triple defining a MASH}
 and with corresponding measure $\mu_{\data}$
as in~\S\ref{construction-MASH}.\index{H@$\mathbf{H}$, the algebraic group giving rise to the MASH~$Y$}
By~\cite[p.~451]{PR} we have $\iota(\H(\adele))\subset\G(\adele)$ is a closed subgroup and by~\cite[Thm.\ 1.13]{Raghunathan} we also have that $\G(F)\iota(\H(\adele))g=\supp(\mu_\data)$ is closed.
We will also write~$g_\data=g$ for the element~$g=(g_v)_{v\in\Sigma}\in\G(\adele)$ determining the MASH set~$Y_\data$. 
Let $H_v = g_{v}^{-1} \iota(\mathbf{H}(F_v))g_{v}$; it is contained in $\G(F_v)$
and stabilizes $\mu_{\data}$.\index{h@$H_v$, component of acting group at the place~$v$}
The subgroup $H_v$ is a Zariski-dense subset of the $F_v$-algebraic group $g_v^{-1}\iota(\mathbf{H})g_v$. 
Of course, $H_v$ need not be the set of all $F_v$-points of {$g_v^{-1}\iota(\mathbf{H})g_v$.}

We recall that $\H$ is assumed to be simply connected in Theorem~\ref{adelic}.
Therefore, except for \S\ref{s;not-sc}, the standing assumption is that $\H$ is simply connected.

\begin{lemma}[Stabilizer lemma]\label{s:stabilizer-lem}
Let $\N$ be the normalizer of $\iota(\H)$ in $\G$. Then
the stabilizer $\stab(\mu_{\data})=\{h\in G: h$ preserves~$\mu_\data\}$ 
of $\mu_{\data}$ consists of $g^{-1} \iota( \H(\adele)) \N(F)  g$
and contains~$g^{-1}\iota(\H(\adele))g$ as an open subgroup. \end{lemma}
\proof
Without loss of generality we may and will assume $g = e$ is the identity element. 
Suppose that $h\in \N(F ).$ Then since $\H$ is simply connected
and the simply connected cover is unique up to isomorphism,
the automorphism $x \mapsto h^{-1}  x h$ of $\iota(\H)$ may be lifted 
to an $F$-automorphism of $\H$, and in particular preserves adelic points; so
    \[
 h^{-1}\iota(\H(\adele))h=\iota(\H(\adele)).
\]
Also note that the Haar measure on $\iota(\H(\adele))$ is not changed by conjugation by $h$ as $\H$ is semisimple.
Therefore, $\G(F)\iota(\H(\adele))h=\G(F)h^{-1}\iota(\H(\adele))h=\G(F)\iota(\H(\adele))$ and $h$ preserves $\mu_\data.$

Suppose
now that $h\in\stab(\mu_{\data}),$ then $h \in G(F) \iota(\H(\adele))$ because
it must preserve the support of $\mu_\data$.
Adjusting $h$ by an element in $\iota(\H(\adele))$, we may assume $h = \gamma \in \G(F)$
and $\gamma^{-1}\iota(\H(\adele))\gamma\subset \G(F)\iota(\H(\adele)).$ 
We note that the connected component of~$\H(\adele)$ of the identity with respect to the Hausdorff
topology is the subgroup~$\prod_{v\in\Sigma_\infty}\H(F_v)^\circ$ and 
the connected component of the countable 
union~$\G(F)\iota(\H(\adele))$ of cosets equals~$\prod_{v\in\places_\infty}\iota(\H(F_v)$. 
Therefore,~$\gamma^{-1}\iota(\H(F_v))\gamma\subset\iota(\H(F_v))$ for every~$v\in\places_\infty$. 
However, by taking Zariski closure this implies that 
$\gamma$ normalizes $\iota(\H)$, i.e. $\gamma \in \N(F)$.

For the final claim of the lemma, suppose $\gamma_i\in \N(F)$ and $h_i\in\H(\adele)$ are such that $\gamma_i\iota(h_i)\to e$ as $i\to\infty.$ We need to show that $\gamma_i\in\iota(\H(\adele))$ for all large enough $i.$ Without loss of generality we may and will assume that $\gamma_i\not\in\iota(\H(\adele))$ for all $i$
and derive a contradiction. If $\H(F)h_i\to\H(F)h$ for some $h \in \H(\adele)$ then there exists $\eta_i\in\H(F)$ so that $\eta_ih_i\to h.$ 
Applying $\iota$ we obtain 
\[
\gamma_i\iota(\eta_i^{-1})\iota(\eta_ih_i)\to\iota(h^{-1})\iota(h)
\]
which forces $\gamma_i\iota(\eta_i^{-1})=\iota(h^{-1})\in\G(F)\cap\iota(\H(\adele))$ for all large enough $i.$
This is a contradiction to our assumption even if the assumed convergence holds only along some subsequence. 
Using the compactness criterion~\cite[Thm.\ 1.12]{Raghunathan} on the 
finite volume homogeneous spaces~$\H(F)\backslash\H(\adele)$ we now obtain that there 
exists a sequence $e\neq\eta_i\in\H(F)$ with $h_i^{-1}\eta_ih_i\to e.$ Note that as 
the center of $\H$ is finite we see that $\iota(\eta_i)\neq e$ for all large enough $i.$ 
This contradicts $\G(F)\iota(h_i)\to\G(F)$ by
the same compactness criterion~\cite[Thm. 1.12]{Raghunathan} applied to~$X$.   
 \qed

\subsection{Volume of homogeneous sets}\label{new-sec-on-volume}
Let us discuss the definition of the volume of a homogeneous set in a general context. Let~$G$ be a locally 
compact group and let~$\Gamma<G$ be a discrete subgroup. 
Let~$\mu$ be a homogeneous probability measure on~$X=\Gamma\backslash G$ so that~$\mu$
is supported on a single closed orbit~$Y=xH_Y$ of the 
stabilizer~$H_Y=\stab(\mu)=\{g\in G: g$ preserves~$\mu\}$. 
Recall that~$Y=x H_Y$ is called a homogeneous subset of~$X$.

We normalize the Haar measure~$m_Y$
on the stabilizer group~$H_Y=\stab(\mu)$  
so that~$m_Y$ projects to~$\mu$.
I.e.\ if we choose $z \in \mathrm{supp}(\mu)$  
and any subset $S \subset H_Y$ for which the map $h \in S \mapsto zh \in Y$
is injective, then we require that $m_Y(S)=\mu(zS)$. Equivalently we identify~$Y$ with $\Gamma_Y\backslash H_Y$ (where~$\Gamma_Y=\operatorname{stab}_{H_Y}(z)$)
and normalize~$m_Y$ so that~$\mu$ is identified with the quotient measure of~$m_Y$ by the counting measure on~$\Gamma_Y.$

We\index{Omega@$\Omega_0=\prod_{v\in\places}\Omega_v$, open set in~$G$ to normalize the volume
of~$H$-orbits} fix some open neighborhood~$\Omega_0$ of the identity in~$G$ with compact closure and use it to normalize
a general definition of the volume of a homogeneous subset: If~$Y\subset X$ is a homogeneous set
and~$m_Y$ is the Haar measure on its stabilizer subgroup~$H_Y$ (normalized as above), then~$\operatorname{vol}_{\Omega_0}(Y)=m_Y(\Omega_0)^{-1}$.

We claim that~$\vol_{\Omega}(Y)\ll \vol_{\Omega_0}(Y)\ll\vol_{\Omega}(Y)$ 
if~$\Omega$ is another such neighborhood. Consequently we will drop the mention of~$\Omega_0$ in the notation of the volume.
To prove the claim it suffices to assume that~$\Omega\subset\Omega_0$, which immediately implies~$m_Y(\Omega)\leq m_Y(\Omega_0)$.
To prove the opposite, choose some open neighborhood~$O$ of the identity with~$O^{-1}O\subset\Omega$
and find some~$g_1,\ldots,g_n\in\overline{\Omega}_0$ with~$\Omega_0\subset\bigcup_i g_iO$. This gives~$m_Y(\Omega_0)\leq\sum_i m_Y(g_iO)$.
If~$m_Y(g_iO)>0$ for~$i\in\{1,\ldots,n\}$, then there exists some~$h_i=g_i\epsilon\in H_Y\cap g_i O$ which gives~$m_Y(g_iO)=m_Y(h_i^{-1}g_i O)=m_Y(\epsilon^{-1}O)\leq m_Y(\Omega)$. Consequently~$m_Y(\Omega_0)\leq n m_Y(\Omega)$ as required.

In the context of this paper we will work with algebraic homogeneous sets~$Y_\data$
and algebraic homogeneous measures~$\mu_\data$ as in~\S\ref{construction-MASH} and~\S\ref{notations}. 
By Lemma~\ref{s:stabilizer-lem} 
we have that~$H_\data=g^{-1}\iota(\mathbf{H}(\adele))g$  is an open subgroup of
the stabilizer~${H_{Y_\data}}$. Therefore, the Haar measure on~$H_\data$
is obtained from the Haar measure on~$H_{Y_\data}$ by restriction (and this
is compatible with the above normalization of the Haar measures).  Also,
the volume defined by
using the Haar measure on~$H_\data$ (as done in~\S\ref{intro-volume}) is bigger
than the volume defined using the Haar measure on the full stabilizer subgroup (as done here).  
In most of the paper (with the exception of~\S\ref{ss:Borel-Prasad} 
and~\S\ref{sec:pigeonhole}) we will work with the volume defined using
the Haar measure on~$H_\data$  (as in~\S\ref{intro-volume}). 

We will assume that~$\Omega_0=\prod_{v\in\places}\Omega_v$, where~$\Omega_v$ is
an\index{Omega@$\Omega_0=\prod_{v\in\places}\Omega_v$, open set in~$G$ to normalize the volume
of~$H$-orbits} open neighborhood
of the identity in~$\G(F_v)$ for all infinite places~$v\in\places_{\infty}$
and~$\Omega_v=K_v$ for all finite places~$v\in\places_f$.

We will make crucial use of the notion of volume in \S\ref{s;vol-hom}, 
where we will construct a good place, and again in \S\ref{sec:pigeonhole}.

\section{An application to quadratic forms}\label{sec:quadratic}

We now give an example of an equidistribution result that follows from our theorem  
but -- even in nonquantitative form -- does not appear to follow directly
from the (ineffective) measure classification theorems for the action of unipotent or semisimple groups.

Let $\mathcal{Q}=\PGL(n,\Z)\backslash \PGL(n,\R)/{{\rm PO}(n,\R)}$ 
be the space of positive definite quadratic forms on $\R^n$
up to the equivalence relation defined by scaling and equivalence over $\Z.$ 
We equip~$\mathcal{Q}$ with the push-forward of the normalised Haar measure
on~$\PGL(n,\Z)\backslash\PGL(n,\R)$. 

Let $Q$ be a positive definite integral quadratic form on $\Z^n$, 
and let $\operatorname{genus}(Q)$ (resp.\ $\operatorname{spin\ genus}(Q)$) 
be its {\em genus} (resp.\ spin genus).  

{ For the rest of this section, we assume that $n\geq 3$.}
 
\begin{thm}\label{thm;genus}
Suppose $Q_i$ varies through any sequence of pairwise inequivalent, integral, positive definite quadratic forms.
Then the genus (and also the spin genus) of $Q_i$, considered as a subset of $\mathcal{Q}$,
equidistributes as $i \rightarrow \infty$ (with speed determined by a power of~$|\operatorname{genus}(Q_i)|$). 
\end{thm}

Similar theorems have been proved elsewhere (see e.g.~\cite{EskinOh-Linik} where the splitting condition is made at the archimedean place).   What is novel here, besides the speed of convergence,  is the absence 
of any type of splitting condition on the $Q_i$. This is where the quantitative result of the present paper becomes useful.  
We also note that it seems plausible that one could
remove the splitting assumptions of \cite{EV} in the borderline cases
where $m-n \in \{3,4\}$ by means of the methods of this paper. However,
for this the maximality assumption in Theorem~\ref{adelic} would need to be removed.

\subsection{Setup for the proof}\label{setup-qf}
We set~$F=\Q$,~$\G=\PGL_n$, and define the quotient~$X=\PGL_n(\Q)\backslash\PGL_n(\adele)$.
Let us recall some facts about the genus and spin genus in order 
to relate the above theorem with Theorem~\ref{adelic}.
For every rational prime $\pl$ put $K_\pl=\PGL(n,\Z_\pl)$
and note that~$K_\pl$ is a maximal compact open subgroup of $\G(\Q_\pl).$ 
We also define $K=\prod_\pl K_\pl$. With this notation we have 
\be\label{e;PGL-class-no}
\G(\adele)=\G(\Q)\G(\R)K=\G(\Q) {\PGL(n,\R)}K.
\ee
It is worth mentioning that~\eqref{e;PGL-class-no} gives a natural identification
between $L^2(X,\vol_G)^K,$ the space of $K$-invariant functions, and 
\[
L^2(\PGL(n,\Z)\backslash\PGL(n,\R),\vol_{\PGL(n,\R)});
\] 
this identification
maps smooth functions to smooth functions. 

Given a positive definite integral quadratic form $Q$ in~$n$ variables, the isometry group $\H'=\SO(Q)$
is a $\Q$-group; it actually comes equipped with a model over $\Z.$ 
This group naturally embeds in $\G,$ and this embedding is defined over $\Z.$ 
We define $\H={\rm Spin}(Q)$ and let $\pi:\H\to\H'$
be the covering map.

Put $K_\pl'=\H'(\Q_\pl)\cap K_\pl,$ $K'=\prod_\pl K_\pl'$, and
$K'(\infty)=\H'(\R)K'$, the latter being a compact open subgroup of~$\H'(\adele)$. 
Note that $\operatorname{genus}(Q)$ is identified with the finite set $\H'(\Q)\backslash\H'(\adele_f)/K'$, 
which may also be rewritten as $\H'(\Q)\backslash\H'(\adele)/K'(\infty).$ 
Similarly the spin genus of $Q$ is given by
 $\H'(\Q)\backslash\H'(\Q)\pi(\H(\adele_f))K'/K',$
which may also be written as 
\[
\H'(\Q)\backslash\H'(\Q)\pi(\H(\adele))K'(\infty)/K'(\infty).
\]  

Let $g_Q\in \PGL(n,\R)$ be so that $g_Q^{-1}\H'(\R)g_Q=g_Q^{-1}\iota(\H(\R))g_Q=\SO(n,\R),$ the standard compact isometry group.  
We define the associated MASH set $Y:=Y_Q=\pi(\H(\Q)\backslash\H(\adele))(g_Q,e,\ldots)$.

\begin{lemma}\label{lem:genus-spingenus}
	The volume of the MASH set~$Y$, the spin genus of~$Q$, the genus of~$Q$, and the discriminant of~$Y$ (as defined in Appendix~\ref{sec;discrim}) are related to each other via\footnote{See~\S\ref{notations} for the~$\star$-notation.}
	\[
	\vol(Y)\asymp|\operatorname{spin\ genus}(Q)|^{\star}\asymp|\operatorname{genus}(Q)|^\star\asymp\operatorname{disc}(Y)^\star.
	\]
\end{lemma}

We postpone the proof to Appendix~\ref{finishingproof}. 

\subsection{Proof of Theorem~\ref{thm;genus}}
Let $Q$ be a positive definite integral quadratic form in~$n$ variables as above.
Let $f\in C_c^\infty(X)^K$ be a smooth, compactly supported and $K$-invariant function. 
Denote by $\pi^+:L^2(X,\vol_G)\to L^2(X,\vol_G)^{\PSL(n,\adele)}$ 
the projection onto the space of $\PSL(n,\adele)$-invariant functions;
this is a $\G(\adele)$-equivariant map. Therefore, $\pi^+(f)$ is $K$-invariant as $f$ is $K$-invariant.
Thus by~\eqref{e;PGL-class-no} we have $\pi^+(f)$ is $\G(\adele)$-invariant which implies
\be\label{e;proj=int}
\pi^+(f)=\int_Xf\operatorname{d}\!\vol_G
\ee
for all~$K$-invariant~$f\in C_c^\infty(X)$. 

Let $f\in C_c^\infty(X)^K.$ Applying Theorem~\ref{adelic}, 
with the homogeneous space $Y$ and in view of~\eqref{e;proj=int}, we get
\[
\left| \int_Y f\operatorname{d}\!\mu_\data-\int_X f\operatorname{d}\!\vol_G\right| \ll  \vol(Y)^{-\finalexponent}\Sob(f).
\]
Using~$\vol(Y)\asymp |\operatorname{spin\ genus}(Q)|^\star\gg |\operatorname{genus}(Q)|^{\star}$ (see Lemma~\ref{lem:genus-spingenus})
this implies Theorem~\ref{thm;genus}.

\section[A proof of property (tau)]{A proof of property $(\tau)$}\label{sec:tau}  
\sectionmark{A proof of property (tau)}
The following theorem was established in full generality through works of Selberg \cite{Sel}, Kazhdan \cite{Kazh}, Burger-Sarnak \cite{BS}, 
and the work of Clozel \cite[Thm. 3.1]{LC} completed the proof\footnote{Clozel states this theorem for a fixed $F$-group $\G,$ however, 
his proof also gives Theorem~\ref{tautheorem}. Our proof of Theorem~\ref{tautheorem} will include uniformity in the $F$-group $\G.$}.

\begin{thm}[Property $(\tau)$] \label{tautheorem}
Let $v$ be a place of $F$ and let~$\G_v$ be an~$F_v$-algebraic semisimple group which is isotropic over~$F_v$.
Let~$\G$ be an algebraic~$F$-group such that~$\G$ is isomorphic 
to~$\G_v$ over~$F_v$.
Then the representation $L^2_0(\G(F) \backslash \G(\adele))$
-- the orthogonal complement of $\G(\mathbb{A})^+$-invariant functions --
is isolated from the trivial representation as a representation of~$\G_v(F_v)$. Moreover,
this isolation (spectral gap) is independent of~$\G$. 
\end{thm}

A corollary of the above (dropping the crucial uniformity
in~$\G$) is that, if $\Gamma$ is any $S$-arithmetic lattice in the group $\G$, 
and $\{\Gamma_N\}_{N \geq 1}$ is the family of all congruence lattices,
then the $\Gamma$-action on $L^2(\Gamma / \Gamma_N)$ possesses a uniform spectral gap.

Our main result offers an alternative to Clozel's part of the proof. (Besides the groups of type $A_1$,
this is the most ``non-formal'' part, as it relies on a special instance of the functoriality principle of Langlands.)

\subsection{Short history of the problem}\label{ss:history}
Let us describe some of the history of Theorem \ref{tautheorem}.
Firstly, it is not difficult
to reduce to the case of an {\em absolutely almost simple},
 {\em simply connected} group $\G$. This being so, it follows by combining the following distinct
 results and principles:

\begin{enumerate}
\item Property $(T)$: If the $F_{v}$-rank of $\G(F_{v})$
is $\geq 2$, it follows from Kazhdan's ``property $(T)$,'' which furnishes the stronger statement
that {\em any} representation not containing the identity is isolated from it, see~\cite{Kazh} 
 and the work of Oh~\cite{O} for a more uniform version that is of importance to us.

\item Groups of type $A_1$:   If the rank of $\G$ over the algebraic closure $\bar{F}$
is equal to $1$, i.e. $\G \times_F \bar{F}$ is isogenous to $\SL_2$,
then $\G$ is necessarily the group of units in a quaternion algebra over $F$.
In that case, the result can be established by the methods of Kloosterman or by the work of 
Jacquet-Langlands and Selberg,  see~\cite{Sel, JL}.

\item Burger-Sarnak principle: Let $\rho: \mathbf{G}  \rightarrow \mathbf{G}'$
be a homomorphism with {finite central kernel}, $\mathbf{G}'$ absolutely almost simple and simply connected,
and suppose that property $(\tau)$ is known for groups that are isomorphic to $\mathbf{G}$ over~$\bar{F}$. 
Then property $(\tau)$ is known for $\mathbf{G}'$ {\em at any place where $\mathbf{G}(F_{v})$ is isotropic,}
see~\cite{BS}.

\item  \label{four} Groups of type~$A_n$: Property $(\tau)$ is true for groups of the form $\SL(1, D)$, where $D$ is a division algebra over
$F$ whose dimension is the square of a prime; or for groups of the form $\mathrm{SU}(D, \star)$, where
$D$ is a division algebra over a quadratic extension $E$ of $F$, and $\star$
is an ``involution of the second kind'' on $D$, i.e.\ inducing the Galois automorphism on $E,$
see Clozel's work~\cite{LC}.

\item For us the uniformity of the spectral gap across all types of groups and across all places
is crucial. This is obtained by combining the above results and was done by Gorodnik,
Maucourant, and Oh \cite[Thm.~1.11]{GMO}.
\end{enumerate}

The hardest of these results is arguably the fourth step. It is established in \cite{LC}
and uses a comparison of trace formulae.
In addition to these results, Clozel \cite[Thm.\ 1.1]{LC} proves that
any absolutely almost simple, simply connected group defined over $F$
admits a morphism from a group $\mathrm{Res}_{F'/F} \mathbf{G}$, where $\mathbf{G}$ is an algebraic $F'$-group isomorphic to one of the types described in \eqref{four}.

\subsection{Effective decay of matrix coefficients}\label{eff-decay-sec}
Let us also note that by the work of Cowling, Haagerup, and Howe \cite{CHH} and others
the conclusion in Theorem \ref{tautheorem} is equivalent to the existence
of a uniform decay rate for matrix coefficients on the orthogonal complement
of the~$\G(\adele)^+$-invariant functions. 
Once more, for groups with property $(T)$ this statement is true for any representation, see \cite{O, GMO}.
  Due to the assumption that~$\G$ is simply connected,
we are reduced to studying {functions in $f \in L^2(\G(F)\backslash \G(\adele))$ with $\int f=0$,}  see~\cite[Lemma 3.22]{GMO}. 

More precisely, Theorem~\ref{tautheorem} is equivalent to the existence of
some~$\consta\label{constHC}>0$\index{kaa2@$\ref{constHC}$, exponent in reformulation of property~$(\tau)$}
such that for all  $K_v$-finite functions {$f_1,f_2\in L^2(\G(F) \backslash \G(\adele))$  with $\int f_1 = \int f_2 = 0$,} the matrix coefficient  
can be estimated as follows
\be\label{HC-estimate}
 \bigl|\bigl\langle \pi_{g_v}f_1,f_2\bigr\rangle\bigr|\leq 
\dim \langle K_v\cdot f_1\rangle^{\frac12}\dim\langle K_v\cdot f_2\rangle^{\frac12}\ \|f_1\|_2\ \|f_2\|_2
\ \Xi_{\G_v}(g_v)^{\ref{constHC}},
\ee
where~$g_v\in\G_v(F_v)$,~$\pi_{g_v}$ is its associated
unitary operator,~$K_v$ is a good maximal compact open subgroup of~$\G_v(F_v),$  $\langle K_v\cdot f\rangle$ is the 
linear span of $K_v\cdot f,$ and~$\Xi_{\G_v}$
is a Harish-Chandra spherical function of~$\G_v(F_v)$. 

As noted in Theorem~\ref{tautheorem} the 
constant~$\ref{constHC}$ is independent of the
precise~$F$-structure of~$\G$. What we did not mention before (as we 
did not have the notation) is that~$\ref{constHC}$ is also
independent of the place~$v$. For groups with property~$(T)$ these statements are proven in
\cite{O}. 

We are able to give a direct proof of~\eqref{HC-estimate} (relying on~\cite{O})
which avoids the third and fourth\footnote{By only avoiding
the fourth point we may restrict ourselves to compact quotients and 
obtain in these cases a constant~$\ref{constHC}$ that only depends 
on~$\dim_F\G,$ the type of $\G_v$ over $F_v$ and the dimension of $F_v$ 
over $\Q_p$ where $p|v,$ but not on~$F$ or even the degree~$[F:\Q]$. 
If we wish to avoid the third and fourth the constant depends on~$[F:\Q]\dim_F\G$.} points of~\S\ref{ss:history} (but leads
to weaker exponents). Indeed,
using the second point, we are
left with the case where $\G$ is an absolutely almost simple, simply connected group over $F$ of absolute rank $\geq 2$.
In that case, one applies Theorem~\ref{adelic} to translates of the diagonal copy of
$X = \G(F) \backslash \G(\adele)$ inside $X \times X$ by elements from $\G(F_v)$
to establish a uniform decay rate for matrix coefficients and so Theorem~\ref{tautheorem}.

We will explain this step first in a special case and then in \S\ref{sec:tau-general} in general.

\subsection{A purely real instance of transportation of spectral gap}

Let $G_1, G_2$ be almost simple, connected Lie groups, and suppose $G_2$ has (T) but $G_1$ has not. 
Let $\Gamma$ be an irreducible lattice in $G=G_1\times G_2$,
e.g.\ this is possible for $G_1=\operatorname{SU}(2,1)(\R)$ and $G_2=\SL_3(\R)$.

We wish to bound the matrix coefficients of $G_1$ acting on $X=\Gamma\backslash G$. 
Let~$G_\Delta=\{(h,h):h\in G\}<G\times G$ and notice that the diagonal orbit $(\Gamma\times\Gamma) G_\Delta\subset X\times X$ 
is responsible for the inner product in the sense that the integral of $f_1\otimes \bar{f_2}$ over 
this orbit equals the inner product $\langle f_1,f_2\rangle$. In the same sense is the deformed orbit $(\Gamma\times\Gamma) G_\Delta (g,e)$
responsible for the matrix coefficients of $g$. The volume of this deformed orbit is roughly speaking a power of $\|g\|$, 
hence effective equidistribution of this orbit gives effective decay of matrix coefficients. 

We note that the main theorem\footnote{In that theorem the implicit constant in the
rate of equidistribution is allowed to depend on the acting group
and so implicity in this instance also on~$g$.} of \cite{EMV} does not apply to this situation
as the acting group giving the closed orbit has been conjugated and does not remain fixed. 
However, if $g=(g_1,e)$
then the almost simple factor of $G_\Delta$ corresponding to $G_2$ remains (as a subgroup of $G\times G$) 
fixed and this is the part with known effective decay (due to property (T)). In this case the 
method of \cite{EMV} (which is also applied in this paper in the adelic context)
can be used to show effective equidistribution and so decay of matrix coefficients for the $G_1$-action.
In all of this,  the rate of (i.e.\ the exponent for) the decay of matrix coefficients for $G_1$ only 
depends on the spectral gap for $G_2$ and the dimension of~$G$ (but not on $\Gamma$).

\subsection{The general case with absolute rank at least two}\label{sec:tau-general}
Let $F$ be a number field and let $\G$ be an absolutely almost simple, simply connected $F$-group whose absolute rank is at least 2.
Let $v$ be a place of the number field $F$ such that $\G(F_v)$ is non-compact.
For any $g\in\G(F_v)$ put $X_g=\{(xg,x): x\in X\},$ 
where we identify $g$ with an element of $\G(\adele)$.
Then $X_g$ is a MASH set, and in view of our definition of volume of a homogeneous set
there exist two positive constants~$\consta\label{kappp1},\consta\label{kappp2}$
(depending only on the root { system} of~$\G(F_v)$) such that  
\[
\|g\|^{\ref{kappp1}}\ll \vol(X_g)\ll \|g\|^{\ref{kappp2}}.
\]\index{kaa3@$\ref{kappp1},\ref{kappp2}$, two constants in the proof of property~$(\tau)$}

As mentioned before we want to apply Theorem~\ref{adelic} to $X_g\subset X\times X.$
However, we want the proof of that theorem to be independent 
of (3) and (4) of \S\ref{ss:history}. We note that in the proof of Theorem~\ref{adelic}
that spectral gap will be used for a ``good place"~$w$. 
In \S\ref{s;good-place} (see also {\S\ref{ss:existence-good}} and the summary in \S\ref{ss;good-place}) the following properties of a good place will be established.
\begin{itemize}
 \item[(i)] ${\rm char}(k_w)\gg 1$ {is large compared to $\dim\G$}, 
 \item[(ii)] both $\G$ and $\iota(\H)$ are quasi-split over $F_w,$ and split over $\widehat{F_w},$
 \item[(iii)] both $K_w$ and $K_w'$ are hyperspecial subgroups of $\G(F_w)$ and 
the subgroup $H_w$ (the component of the acting group at the place~$w$) respectively. 
\end{itemize}
Indeed almost all places satisfy these conditions. We also will show the effective 
estimate $q_w\ll\log(\vol(\text{homogeneous set}))^2;$ this needs special care when the $F$-group $\H$ changes. 
In our application to Theorem~\ref{tautheorem}, however,
the algebraic subgroup $\H=\{(h,h):h\in\G\}<\G\times\G$ is fixed 
and $X_g$ changes with the element $g\in \G(F_v).$
In this case we find a place $w\neq v$ (independent of~$g$)
which satisfies (i), (ii) and (iii) so that in addition $\G$ is $F_w$-split\footnote{In fact we may also ensure that~$F$ splits
over~$\Q_p$ by applying this argument for~$\operatorname{Res}_{F/\Q}\G$. In this case~$F_w=\Q_p$ for~$w|p$ and so~$\G(F_w)$
is a simple group over~$\Q_p$ from a finite list that is independent of~$F$ and even of~$[F:\Q]$.
Using this one can establish the earlier noted independence of~$\ref{constHC}$ from~$[F:\Q]$.}.
Note that by Chebotarev density theorem, {\cite[Thm.~6.7]{PR}} there are infinitely many such places.  
Then by our assumption that the {\it absolute rank of $\G$ is at least two}
we have the required spectral gap for $\G(F_w)$ since this group has property~(T).
Therefore we get: there exists a constant\index{kaa3@$\ref{oh-const}$, exponent for matrix coefficients
for groups with~$(T)$}
 $\consta\label{oh-const}>0$ (which depends only on the type of $\G$) so that for all  $K_w$-finite functions
$f_1,f_2\in L^2_0(X)$ we have
\be\label{eq:oh-mc}
|\langle \pi_{h_w}f_1,f_2\rangle|\ll \dim\langle K_w\cdot f_1\rangle^{1/2}\dim \langle K_w\cdot f_2\rangle^{1/2}\|f_1\|_2\|f_2\|_2\|h_w\|^{-\ref{oh-const}}
\ee
where the implicit constant depends on $\G(F_w),$ see~\cite{O}.

Fix such a place, then using~\eqref{eq:oh-mc} as an input in the proof of
Theorem~\ref{tautheorem} and taking 
$g$ large enough so that $q_w\leq \log(\|g\|)$ we get from {Theorem~\ref{adelic}} the conclusion of the theorem.
In particular, if $f=f_1\otimes\bar f_2$
with $f_i\in C_c^\infty(X)$ for $i=1,2,$ then 
\begin{align*}
\Big|\int_{X\times X} f\operatorname{d}\!\mu_{X_g}&-\int_{X\times X} f_1\otimes\bar f_2\operatorname{d}\!\vol_{G\times G}\Big|\\
&=\Big|\langle \pi_gf_1, f_2\rangle_X-\int_X f_1\operatorname{d}\!\vol_G\int_X\bar f_2\operatorname{d}\!\vol_G\Big|\ll \|g\|^{-\finalexponent}\Sob(f);
\end{align*}
where~$\finalexponent>0$ depends
on~$\dim\G$,~$[F:\Q]$ (if~$X$ is non-compact), and~$\ref{oh-const}$ as in~\eqref{eq:oh-mc}.
The implied multiplicative constant depends on~$X$ and so also on the~$F$-structure of~$\G$.
We note however, that this constant is irrelevant due to~\cite{CHH}, which
upgrades the above to a uniform effective bound on the decay of the matrix coefficients as in~\eqref{HC-estimate}
with $\ref{constHC}$ independent of~$\G$. This implies Theorem~\ref{tautheorem}.

\section{Construction of good places}\label{s;good-place}
See \S\ref{sec:prem-notation} for general notation.
In particular, $\data=(\H, \iota, g)$ consists of 
a simply connected semisimple $F$-group $\H,$ 
an $F$-homomorphism $\iota: \H \rightarrow \G$ and an element 
$g = (g_v) \in \G(\adele)$ determining a homogeneous set;
the stabilizer of this set contains the acting group $H_\data=g^{-1}\iota(\H(\adele))g$.
We will not assume within this section (or the related Appendix~\ref{sec;discrim}) 
that {$\iota(\H)$ is a maximal subgroup of $\G$}.

The purpose of this section is to show that we may always choose a place
$w$ with the property that $H_w=g_w^{-1}\iota(\H(F_w))g_w \subset \G(F_w)$ is not too ``distorted''.
The precise statement is the proposition in~\S\ref{l;splitting-place},
but  {\em if the reader is interested
in the case of Theorem~\ref{adelic} where $Y$ varies through a sequence
of sets where $w$ and $H_w$ are fixed (e.g.~the argument
in~\S\ref{sec:tau}) the reader may skip directly to \S\ref{sec:algebra-at-good}.}
This section relies heavily on the results established in~\cite{Pr} and~\cite{BorPr} which in turn relies on Bruhat-Tits theory.

\subsection{Bruhat-Tits theory.}
\label{ggo}

We recall a few facts from Bruhat-Tits theory, see~\cite{Ti} and references there for the proofs.
Let $\G$ be a connected semisimple group defined over $F.$ 
Let $v$ be a finite place, then 
\begin{enumerate}
\item For any point $x$ in the Bruhat-Tits building of $\G(F_v),$ 
there exists a smooth affine group scheme $\Gfrak_v^{(x)}$\index{g@$\Gfrak_v^{(x)}$, smooth affine group scheme for~$x$
in the Bruhat-Tits building} over 
$\order_v,$ unique up to isomorphism, such that: 
its generic fiber is $\G(F_v),$ and the compact open subgroup $\Gfrak_v^{(x)}(\order_v)$  
is the stabilizer of $x$ in $\G(F_v),$ see~\cite[3.4.1]{Ti}. \vspace{1mm} 
\item If $\G$ is split over $F_v$ and $x$ is a {\it special} point, then\index{special point in the Bruhat-Tits building}
the group scheme $\Gfrak_v^{(x)}$ is a Chevalley group scheme with
generic fiber $\G,$ see~\cite[3.4.2]{Ti}.\vspace{1mm} 
\item $\red_v:\Gfrak_v^{(x)}(\order_v)\rightarrow\underline{\Gfrak_v}^{(x)}(k_v),$ 
the reduction mod $\varpi_v$ map,\index{r@$\red_v$, reduction maps}   
is surjective, { which follows from the smoothness above}, see~\cite[3.4.4]{Ti}. \vspace{1mm} 
\item $\underline{\Gfrak_v}^{(x)}$ is connected and semisimple if and only if $x$ is a {\it hyperspecial} point.  Stabilizers of hyperspecial points in $\G(F_v)$ will be called hyperspecial subgroups, see~\cite[3.8.1]{Ti} and~\cite[2.5]{Pr}.\index{hyperspecial points and subgroups}
\end{enumerate}

If $\G$ is quasi-split over $F_v,$ and splits over $\widehat{F_v},$ 
then hyperspecial vertices exist, and they are compact open subgroups of maximal volume. 
Moreover a theorem of Steinberg implies that $\G$ is quasi-split over $\widehat{F_v}$ 
for all $v,$ see~\cite[1.10.4]{Ti}.

It is known that
for almost all $v$ the groups $K_v$ are hyperspecial, 
see~\cite[3.9.1]{Ti} (and~\S\ref{notations} for the definition of~$K_v$). 
We also recall that: for almost all $v$ the group $\G$ 
is quasi-split over $F_v,$ see ~\cite[Thm.~6.7]{PR}.

\subsection{Passage to absolutely almost simple case}
We will first find the place $w$ under the assumption that $\H$ is $F$-almost simple, 
the result for semisimple groups will be deduced from this case. 

In this section we will need to work with finite extensions of $F$ as well. To avoid confusion we will denote $\adele_E$\index{A@$\adele_E,$ the ring of adeles for the field $E=F$ or~$E=F'$}
 for the ring of adeles of a number field $E$, 
this notation is used only here in \S\ref{s;good-place} and in Appendix~\ref{sec;discrim}.
 
Suppose for the rest of this section, until specifically mentioned otherwise, 
that $\H$ is $F$-almost simple. Let $F'/F$ be\index{F@$F'$, finite field extension
with~$\H={\rm Res}_{F'/F}(\H')$} a finite extension so that 
$\H={\rm Res}_{F'/F}(\H'),$ where $\H'$ 
is an absolutely almost simple $F'$-group;  note that $[F':F] \leq \dim \H$. 
We use the notation~$v'\in\places_{F'}$ for the places of~$F'$\index{v@$v'$, a place of~$F'$,
where~$\H={\rm Res}_{F'/F}(\H')$}\index{sigma@$\Sigma_{F'}$, places of~$F'$, where~$\H={\rm Res}_{F'/F}(\H')$}.
For any $v\in\places_F,$ there is a natural isomorphism between 
$\H(F_v)$ and $\prod_{v'|v}\H'(F'_{v'});$ 
this induces an isomorphism between $\H(\adele_F)$ and $\H'(\adele_{F'}).$

\subsection{Adelic volumes and Tamagawa number}\label{s;volume-form}


 Fix an algebraic volume form $\omega'$ on $\H'$ defined over~$F'$. The form $\omega'$ determines a Haar measure on each vector space 
$\mathfrak{h}'_{v'} := \Lie(\H') \otimes F'_{v'}$
which also gives rise to a normalization of the Haar measure on $\H'(F'_{v'})$.  
Let us agree to refer to both these measures\index{omega@$\omega'$,
algebraic volume form giving normalized Haar measures} as $|\omega'_{v'}|.$
We denote by $|\omega'_{\adele}|$ the product measure on $\H'(\adele_{F'}),$
then
\be\label{e;tamagawa}
|\omega'_{\adele}|(\H'(F') \backslash \H'(\adele_{F'}))=D_{F'}^{\frac{1}{2}\dim\H'}\tau(\H'),
\ee
where $\tau(\H')$ is the {\em Tamagawa number} of $\H',$ and $D_{F'}$ is the discriminant of ${F'}$. 
In the case at hand $\H'$ is simply connected, thus, it is known that $\tau(\H')=1,$ see~\cite{Ko} 
and \cite[Sect.\ 3.3]{Pr} for historic remarks and references.

The volume formula~\eqref{e;tamagawa} is for us just a starting point. It relates the Haar measure 
on $Y$ to the algebraic volume form $\omega'$
(and the field~$F'$). However, the volume of our homogeneous set $Y$ as a subset 
of $X$ depends heavily on the amount of distortion (coming from the precise $F$-structure of $\H$ and $g$).  

\subsection{The quasisplit form}\label{volumes}
Following~\cite[Sect.~0.4]{Pr} we let $\Hcal'$\index{H@$\Hcal',$ the quasi split inner form of $\H'$} denote 
a simply connected algebraic group defined and quasi-split over $F'$ which is an inner form of $\H'.$  
Let $L$\index{L@$L,$ the field associated to $\Hcal'$} be the field associated\footnote{In most
cases~$L$ is the splitting field of~$\Hcal'$ except in the case where~$\Hcal'$ is a triality
form of~${}^6\mathsf{D}_4$ where it is a degree~$3$ subfield\label{pagewithfootnote}
of the degree~$6$ Galois splitting field with Galois group~$S_3$. Note that there are three such
subfields which are all Galois-conjugate.} to $\Hcal'$ as in~\cite[Sect.~0.2]{Pr},
it has degree~$[L:F']\leq 3$. We note that $\Hcal'$ should be thought of as the least distorted 
version of $\H',$ it and the field $L$ will feature in all upcoming volume considerations.

Let $\omega^0$ be a differential form on $\Hcal'$ corresponding to $\omega'$. 
This can be described as follows: Let $\varphi:\H'\rightarrow\Hcal'$ be an isomorphism 
defined over some Galois extension of $F'.$ 
We choose $\omega^0$ so that $\omega'=\varphi^*(\omega^0)$, it is defined over~$F'$.
It is shown in~\cite[Sect.\ 2.0--2.1]{Pr} that, up to a root of unity of order at most 3,   this is independent of the choice of $\varphi.$ 

As was done in~\cite{Pr} we now introduce local normalizing parameters $\lambda_{v'}$
which scale the volume form $\omega^0$ to a more canonical volume form on $\Hcal'(F'_{v'}).$

\subsection{Normalization of the Riemannian volume form} \label{RVF}

Let us start the definition of these parameters at the archimedean places. 

Let $\mathfrak{g}$ be any $d$-dimensional semisimple real Lie algebra. 
We may normalize an inner product on $\mathfrak{g}$ as follows: 
Let $\mathfrak{g}_{\C}$ be the complexification of $\mathfrak{g}$
and $\mathfrak{g}_0$ a maximal compact subalgebra.
The negative Killing form gives rise to an inner product $\langle \cdot, \cdot \rangle$ on $\mathfrak{g}_0$. 
This can be complexified to a Hermitian form on $\mathfrak{g}_{\C}$
and then restricted to a (real) inner product on $\mathfrak{g}$. 

As usual, the choice of an inner product on a real vector space
determines a nonzero $\nu \in \wedge^d \mathfrak{g}^*$ up to sign.
We refer to this as the Riemannian volume form on $\mathfrak{g}$,
and again write~$|\nu|$ for the associated Riemannian volume
on~$\mathfrak g$ or on a real Lie group with Lie algebra~$\mathfrak{g}$. 
Note that the Hermitian form depends on the choice of the maximal compact
subalgebra, but the Riemannian volume is independent of this choice.

For any archimedean place, let $\lambda_{v'}>0$\index{L@$\lambda_{v'},$ normalizing factor for $\omega^0_{v'}$}
be such that $\lambda_{v'} |\omega^0_{v'}|$ 
coincides with the Riemannian volume on $\mathcal{H}'(F'_{v'})$ (using the above normalization).

\subsection{Normalization of the Haar measure at the finite places}\label{ss:integral-norm}

For any finite place $v'$ of $F',$ we choose an $\mathfrak{o}'_{v'}$-structure  on $\Hcal'$, i.e.\ a smooth affine group scheme over $\order'_{v'}$ with generic fiber $\Hcal'.$
To define $\lambda_{v'}$ at the finite places 
we have to choose the $\order'_{v'}$-structure more explicitly, as in \cite[Sect.\ 1.2]{Pr}.

We let $\{\Pcal_{v'} \subset \Hcal'(F'_{v'})\}$ denote
a coherent collection of parahoric subgroups  of ``maximal volume'', see~\cite[Sect.\ 1.2]{Pr} for an explicit description.  Let us recall that by a coherent collection we mean that $\prod_{v'\in \places_{F',f}}\Pcal_{v'}$ is a compact open subgroup of $\Hcal'(\adele_{F',f}).$ 

Note that any $F'$-embedding of $\Hcal'$ into ${\rm GL}_{N'}$ 
gives rise to a coherent family of compact open subgroups of $\Hcal'(\adele_{F'})$
which at almost all places satisfies the above requirements on $\Pcal_{v'},$ see \S\ref{ggo}.
At the other places we may choose $\Pcal_{v'}$ as above and then use (1) in 
\S\ref{ggo} to define the $\order'_{v'}$-structure on $\Hcal'(F'_{v'}).$
Let us also remark that maximality of the volume implies that the corresponding parahoric is either hyperspecial (if $\Hcal'$ splits over an unramified extension)
or special with maximum volume (otherwise).

This allows us, in particular, to speak of ``reduction modulo $\varpi'_{v'}$". 
If $v'$ is a finite place of $F'$, we let $\overline{\Mcal}_{v'}$ denote the reductive quotient of $\red_{v'}(\Pcal_{v'})$; this is a reductive group over the residue field. 

For any nonarchimedean place, let $\ell_{v'}\in F'_{v'}$ be so that $\ell_{v'}\omega^0_{v'}$ is 
a form of maximal degree, defined over $\order'_{v'},$ whose reduction mod $\varpi'_{v'}$ is non-zero, and 
let $\lambda_{v'}=|\ell_{v'}|_{v'}$\index{L@$\lambda_{v'},$ normalizing factor for $\omega^0_{v'}$}. 

\subsection{Product formula}\label{ss:p-formula}
Let us use the abbreviation~$D_{L/F'}=D_LD_{F'}^{-[L:F']}$
for the norm of the relative discriminant of~$L/F'$, see~\cite[Thm.~A]{Pr}.
It is shown in~\cite[Thm. 1.6]{Pr} that  \be\label{e;prod-lambda}
\prod_{v'\in\places_{F'}}\lambda_{v'} = 
D_{L/F'}^{\frac{1}{2}\mathfrak{s}(\Hcal')}\cdot A,
\ee
where $A>0$ depends only on~$\H$ over~$\bar{F}$ and~$[F':\Q]$, $\mathfrak{s}(\Hcal')=0$ when $\Hcal'$ splits
over $F'$ in which case $L=F',$ and $\mathfrak{s}(\Hcal')\geq5$ otherwise;
these constants depend only on the root system of $\Hcal'$.

It should be noted that the parameters $\lambda_{v'}$ were defined using $\Hcal'$ and~$\omega^0$ but will be used to renormalize $\omega'_{v'}$ on $\H'.$ 

\subsection{Local volume contributions}\label{ss;k-w-star}

Recall we fixed an open\index{Omega@$\Omega_0=\prod_{v\in\places}\Omega_v$, open set in~$G$ to normalize the volume
of~$H$-orbits} subset $\Omega_0=\prod_{v\in\places_\infty}\Omega_v\times\prod_{v\in\places_f}K_v\subset \G(\adele)$ and defined $\vol(Y)$
of an algebraic semisimple homogeneous set~$Y$ using this subset, see~\eqref{volume} and \S\ref{new-sec-on-volume}.

For every $v\in\places_\infty$ we may assume that~$\Omega_v$ is constructed as follows. 
Fix a bounded open subset\index{x@$\Xi_v$, open subset of~$\mathfrak{g}_v$ \cleanup{I think we define $\mathfrak{g}$ but not $\mathfrak{g}_v$ - it's almost obvious but worth defining}
defining~$\Omega_v$} $\Xi_v \subset \mathfrak{g}_v$ which is symmetric around the origin
such that $\exp$ is diffeomorphic on it and so that
 \begin{equation} \label{omegav2} 
  \mbox{every eigenvalue { $\sigma$} of $\mathrm{ad}(u)$
   for $u \in \Xi_v$ satisfies $|\sigma|_v < \frac1{10},$ }
\end{equation} 
where we regard the {norm $|\cdot|_v$} as being extended to an algebraic closure of $F_v$. 
With this we define $\Omega_v= \exp(\Xi_v)$. We will also require that properties
similar to~\eqref{omegav2} hold for finitely many finite-dimensional representations
which will be introduced in the proof below, see the discussion leading 
to~\eqref{eq:chevalley-claim}.

To compare the Haar measure on~$\H(F_v)$ with the Haar measure on the Lie algebra~$\mathfrak h_v$\index{h@$\hfrak_v,$ Lie algebra of $\H(F_v)$ for $v\in\places_\infty$ in \S\ref{ss;k-w-star}}
in the following proof we also recall that the derivative of the exponential map~$\exp:\mathfrak h_v\to\H(F_v)$
at $u \in \mathfrak{h}_v$ is given by\footnote{We think of
the derivative as a map from~$\mathfrak h_v$ to itself by using left-mutiplication by the inverse of~$\exp(u)$
to identify the tangent plane at the point~$\exp(u)$ with the tangent plane at the identity. As the latter
is measure preserving, this identification does not affect the estimates for the Jacobian of the exponential map.}
\begin{equation} \label{dderiv} 
\frac{1 - \exp(-\mathrm{ad} u)}{\mathrm{ad} u}=1-\frac{\mathrm{ad}u}2+\frac{(\mathrm{ad}u)^2}{3!}-+\cdots.
\end{equation}

For $v\in\places_{F,f},$ set $K_v^* =  \iota^{-1}(g_v K_v g_v^{-1})$,
and put $K_v^* = \iota^{-1} (g_v \Omega_v g_v^{-1})$ for $v\in\places_{F,\infty}$. 
Note that, for each finite place $v,$ the group $K_v^*$  is an open compact subgroup of $\H(F_v).$ 
For any place $v\in\places_{F,f}$ we can write 
\[
\mbox{$K_v^*\subseteq\prod_{v'|v} K_{v'}^*\;$
where $K_{v'}^*$  is the projection of $K^*_v$ into $\H'(F'_{v'}).$}
\]
Let $J_{v'}$ be the measure of $K_{v'}^*$ under  $\lambda_{v'} |\omega'_{v'}|$. 

We define $p_{v'}={\rm char}(k_{v'})$ and note that $q_{v'}=\#k_{v'}=p_{v'}^l$ for some $l\leq [F':\Q]$.

 \begin{prop}\label{p;non-hyp-est}
	The local terms~$J_{v'}$ as above satisfy the following properties.
 \begin{enumerate}
 \item For $v'$ a finite place of $F'$, $J_{v'} \leq 1$. 
 \item Let $v'$ be a finite place of $F'$ such $L/F'$ is unramified at $v'$. 
  Suppose that  $\H'$ is not quasi-split over $F'_{v'}$ or $K_{v'}^*$ is not hyperspecial, then  $J_{v'}\leq 3/4$.
{ If in addition~$q_{v'}>13$, then~$J_{v'} \leq \max\{\tfrac{1}{p_{v'}},\tfrac{q_{v'}+1}{q_{v'}^2}\}\leq 1/2$.}
 \item
 For an archimedean place $v$ of $F$, $\Big(\prod_{v'|v}\lambda_{v'} |\omega'_{v'}| \Big)(K_v^*) $ is bounded above by a constant depending only on $\G$ and $\Omega$.
 \end{enumerate}
 \end{prop} 

\begin{proof}  
(Case of $v'$ finite:) Let $P_{v'}$ be a minimal parahoric subgroup containing $K_{v'}^*$. 
Let $\overline{{\rm M}}_{v'}$ be the reductive quotient of the corresponding $k_{v'}$-group $\red_{v'}(\mathfrak P_{v'})$ where $\mathfrak P_{v'}$ is the smooth affine {$\order_{v'}$-group scheme whose $\order_{v'}$-points are $P_{v'}$; existence of such is guranteed by Bruhat-Tits theory, see (1) of \S\ref{ggo}.} 
  
This gives
  \[ 
  J_{v'} = \lambda_{v'} |\omega'_{v'}|(K_{v'}^*)  \leq  \lambda_{v'}  |\omega'_{v'}|(P_{v'}) =
  \frac{\# \overline{{\rm M}}_{v'}(k_{v'})}{q_{v'}^{(\dim \overline{{\rm M}}_{v'}+\dim\overline{\Mcal}_{v'})/2}}
  \]
 where the last equality is~\cite[Prop. 2.10]{Pr}. The same proposition also shows
that the right-hand side is at most $1$ as claimed in (1).

  Using~\cite[Prop.\ 2.10]{Pr} one more time
 we have: if $\H'$ is not quasi-split over $F'_{v'}$, or $\H'$ splits over the maximal 
unramified extension~$\widehat{F'_{v'}}$ but $P_{v'}$ is not hyperspecial, then
\[
\frac{\# \overline{{\rm M}}_{v'}(k_{v'})}{q_{v'}^{(\dim \overline{{\rm M}}_{v'}+\dim\overline{\Mcal}_{v'})/2}}\leq\frac {q_{v'}+1}{q_{v'}^2}.
\]
Therefore, we assume now (as we may), that $\H'$ is quasi-split over $F'_{v'}.$  

As the quasi-split inner form is unique we obtain that~$\Hcal'$ and~$\H'$ are isomorphic over~$F'_{v'}.$

Note that if~$v'$ does not ramify in~$L$, then~$\Hcal'$ splits over the maximal unramified extension~$\widehat{F'_{v'}}$.
Indeed, by the footnote on page~\pageref{pagewithfootnote} in most cases~$L$ is the splitting field of~$\Hcal'$ which gives the remark
immediately. In the case of the triality
form of~${}^6\mathsf{D}_4$ the splitting field of~$\Hcal'$ is a degree~$6$ Galois extension~$E/F'$
with Galois group~$S_3$ which is generated by~$L\subset E$ and its Galois images. As~$v'$ does not ramify in~$L$,
this also implies that~$v'$ does not ramify in~$E$. As we may assume $\H'$ and $\Hcal'$ are isomorphic over $F'_{v'}$, the group $\H'$ also splits over $\widehat{F'_{v'}}.$


In view of this the only case which needs extra argument is when $\H'$
is $F'_{v'}$-quasisplit, split over $\widehat{F'_{v'}},$ the only parahoric subgroup containing $K_{v'}^*$ is a hyperspecial parahoric subgroup $P_{v'},$ and $K_{v'}^*\subsetneq P_{v'}.$ Note that (1) and the fact that $K_{v'}^*\subsetneq P_{v'}$ in particular imply that\footnote{Let us mention that this bound is sufficient for finding a ``good place" in \S\ref{s;vol-hom}. However, the stronger estimate in (2) will be needed 
in~\S\ref{ss:Borel-Prasad} if~$\iota(\H)$ is not simply connected and in Appendix~\ref{sec;discrim}.} $J_{v'}\leq 1/2.$ It remains to show that the stronger bound holds in this case as well.  
 
We will use the notation and statements recalled in \S\ref{ggo}. Let $\mathfrak P_{v'}$ be the smooth group scheme associated to $P_{v'}$ by Bruhat-Tits theory.
Since $P_{v'}$ is hyperspecial we have $\red_{v'}\mathfrak P_{v'}$ is an almost simple
group. The natural map $P_v'\to\red_{v'}\mathfrak P_{v'}(k_{v'})$ is surjective. We also recall that since $\H'$ is simply connected $\red_{v'}\mathfrak P_{v'}$ is connected, see~\cite[3.5.3]{Ti}. Let $P_{v'}^{(1)}$ denote the first congruence subgroup of $P_{v'},$ i.e.\ the kernel of the natural projection. 

First note that if $P_{v'}^{(1)}\not\subset K_{v'}^*,$ then the finite set $P_{v'}^{(1)}/P_{v'}^{(1)}\cap K_{v'}^*$ injects into $P_{v'}/K_{v'}^*.$ But $P_{v'}^{(1)}$ is a pro-$p_{v'}$ group and hence any subgroup of it has an index which is a power of $p_{v'}.$ Therefore, we get the claim under this assumption.
In view of this observation we assume $P_{v'}^{(1)}\subset K_{v'}^*.$
Therefore, since $K_{v'}^*\subsetneq P_{v'}$ we have $K_{v'}^*/P_{v'}^{(1)}$ is a proper subgroup of $P_{v'}/P_{v'}^{(1)}=\red_{v'}\mathfrak P_{v'}(k_{v'}).$ The latter is a connected almost simple group of Lie type and the smallest index of its subgroups is well understood. In particular, by~\cite[Prop.\ 5.2.1]{KleidLieb} the question reduces to the case of simple groups of Lie type. 
Then by the main Theorem in~\cite{LiebSax}, for the exceptional groups, and~\cite{Coop} for the classical groups, see also~\cite[Thm.\ 5.2.2]{KleidLieb} for a discussion, we have $[P_{v'}^*:K_{v'}^*]\geq q_{v'}$ so long as $q_{v'}\geq 13.$
The conclusion in part (2) follows from these bounds.
\medskip

(Case of $v$ infinite:)\footnote{See also~\cite[Sect.\ 3.5]{Pr}.} 
Note that up to conjugation by $\G(F_v)$ there are only finitely many
homomorphisms from $\H(F_v)$ to $\G(F_v)$ with finite (central) kernel.  We fix once and for all representatives 
for these maps. 
We will refer to these representatives as {\it standard homomorphisms} (and the list depends only on~$\G(F_v)$).  
We fix a  compact  
form~$\mathfrak h_{v,0}$ for the group~$\H(F_v)$. Taking the negative of the
restriction of the Killing form to the compact form we extend it to a Hermitian form on~$\mathfrak h^\C_v$ and restrict
it to a Euclidean structure on~$\mathfrak h_v$, which we will denote by~$\mathfrak q$.
For each standard homomorphism we fix a {\it standard Euclidean structure} 
on $\mathfrak g_v$ as follows: 
Let~$\jmap_0$ be the derivative of a standard homomorphism. Let  
$\mathfrak{g}_{v,0}\subset\mathfrak g_v^{\C}$ be a compact form of $\mathfrak g_v$ 
so that $\jmap_0( \mathfrak{h}_{v,0})\subset \mathfrak{g}_{v,0}.$
As above for the Euclidean structure $\mathfrak q$ on $\mathfrak h_v$, we use the compact forms $\mathfrak{g}_{v,0}$ to induce standard Euclidean structures associated to~$\jmap_0$ which we will denote
by~$\mathfrak p_0$.  

We also let~$\rho_0:\G\to\SL_D$ be a representation given by Chevalley's theorem 
applied to the semisimple algebraic group~$\jmap_0(\H)$ considered
over~$F_v$, and let~$w_{\jmap_0}\in F_v^D$ be so that~$\jmap_0(\H)=\operatorname{Stab}_{\G}(w_{\jmap_0})$. 
As there are only finitely many standard homomorphisms we may require that the analogue of~\eqref{omegav2}
holds also for these representations. In particular, we obtain that for~$u\in\Xi_v$ there is a one-to-one
correspondence between the eigenvalues and eigenvectors of~$D\rho_0 u$ and
the eigenvalues and eigenvectors of~$\rho_0(\exp u)$. Hence for any~$u\in\Xi_v$ and~$w\in F_v^D$
we have
\be\label{eq:chevalley-claim}
 \rho_0(\exp u)w=w\mbox{ implies that }\rho_0(\exp (tu))w=w\mbox{ for all }t\in F_v.
\ee

Let~$D\iota:\mathfrak h_v\to\mathfrak g_v$ denote the derivative 
of the homomorphism~$\iota:\H(F_v)\to\G(F_v)$. Then the map $\Ad(g_v^{-1}) \circ D\iota$
induces an inclusion of real Lie algebras $\jmap: \mathfrak{h}_v \rightarrow \mathfrak{g}_v$, and
a corresponding inclusion of complexifications 
$\jmap_{\C}: \mathfrak{h}_v^{\C} \rightarrow \mathfrak{g}_v^{\C}.$
Let $g_0\in \G(F_v)$ be so that $\jmap_0=\Ad(g_0)\circ \jmap$ is the derivative of one of the standard homomorphisms. 
Then $\mathfrak{g}_{v,\jmap}=\Ad(g_0^{-1})\mathfrak{g}_{v,0}\subset\mathfrak g_v^{\C}$ is a compact form of $\mathfrak g_v$ 
so that $\jmap( \mathfrak{h}_{v,0})\subset \mathfrak{g}_{v,\jmap}.$
The compact form $\mathfrak g_{v,\jmap}$ induces a Euclidean structure on $\mathfrak g_v$, which
we refer to as $\mathfrak p_\jmap$ and satisfies~$\|\Ad(g_0) u \|_{\mathfrak p_0}=\|u\|_{\mathfrak p_\jmap}$
for all~$u\in\mathfrak g_v$.

Recall the definition of $\Xi_v$ from \eqref{omegav2}.
We now will analyze the preimage~$K_v^*$ of $g_v\Omega_vg_v^{-1}$ in $\H(F_v)$ under the map $\iota$,
and show that it equals $\exp(\jmap^{-1}(\Xi_v))$. 
Clearly the latter is contained in~$K_v^*$ and we only have to concern ourself with the 
opposite implication. So let~$h\in K_v^*$ satisfy~$g_v^{-1}\iota(h)g_v=\exp u\in \Omega_v$ for some~$u\in\Xi_v$. We need to show that $u\in \jmap(\hfrak_v).$ Note that $g_v^{-1}\iota(\H)(F_v)g_v$ is the stabilizer of $w=\rho_{\jmap_0}(g_0^{-1})w_{\jmap_0}$ in $\G(F_v).$  
So $\exp(u)$ fixes $w$ and
by the property of~$\Xi_v$ in~\eqref{eq:chevalley-claim} we obtain $\exp(F_vu)\subset g_v^{-1}\iota(\H)(F_v)g_v.$ Thus we have 
$u\in\jmap(\hfrak_v)$ as we wanted.

{
Let us write $E(u)$ for the Jacobian of the exponential map and use
the abbreviation~$\mu_{\mathfrak q}=\prod_{v'|v}\lambda_{v'} |\omega'_{v'}|$
for the normalized Riemannian volume on~$\mathfrak h_v$. 
We define 
\be\label{eq:def-Jv-Arch}
J_v:= \Big({\textstyle\prod_{v'|v}}\lambda_{v'} |\omega'_{v'}| \Big)(K_v^*)   
= \int_{u \in \jmap^{-1}(\Xi_v)}E(u)\operatorname{d}\!\mu_{\mathfrak q}(u),
\ee
 where we used also the definitions above.}
 In view of \eqref{omegav2} (which pulls back to an analogous claim for~$u\in\jmap^{-1}(\Xi_v)$
 and its adjoint on~$\mathfrak h_v$) and~\eqref{dderiv},
 $E(u)$ is bounded above and below for all $u \in \jmap^{-1}(\Xi_v).$ Therefore
 \[
 J_v\asymp \int_{u \in \jmap^{-1}(\Xi_v)} \operatorname{d}\!\mu_{\mathfrak q}(u)=\mu_{\mathfrak q}(\jmap^{-1}(\Xi_v)).
 \]
Now note that for the derivative~$\jmap_0$ of a standard homomorphism we have
\[ 
\|u\|_{\mathfrak q}\asymp \|\jmap_0(u)\|_{\mathfrak p_0}=\|\jmap(u)\|_{\mathfrak p_\jmap},
\]
which also gives
\[
 J_v\asymp \mu_{\mathfrak q}(\jmap^{-1}(\Xi_v))\asymp \mu_{\jmap}(\Xi_v\cap \jmap(\mathfrak{h}_v)),
\]
where~$\mu_{\jmap}$ is the~$\ell$-dimensional Riemannian volume 
(induced by $\mathfrak p_\jmap$) on the subspace~$\jmap(\mathfrak{h}_v)\subset\mathfrak g_v$ 
with~$\ell=\dim\mathfrak{h}_v$. 

Let $u_1,\ldots,u_\ell\in \jmap(\mathfrak{h}_v)$ be an orthonormal basis, with respect to 
the standard Euclidean structure $\mathfrak p_0$, 
for $\jmap(\mathfrak{h}_v)$.
Then there exists a constant $\constc\label{shapeconstc}$ (which depends only on~$\Xi_v$ and so on~$\G$) such that
\[
\textstyle\Big\{\sum_{|\lambda_r|\leq 1/\ref{shapeconstc}}\lambda_ru_r\in \jmap(\mathfrak h_v)\Big\}\subset\Xi_v \cap \jmap(\mathfrak{h}_v)\subset\Big\{\sum_{|\lambda_r|\leq \ref{shapeconstc}}\lambda_ru_r\in \jmap(\mathfrak h_v)\Big\}
\]
which gives
\[
J_v\asymp\mu_\jmap(\Xi_v \cap \jmap(\mathfrak{h}_v))\asymp 
\|u_1\wedge\cdots\wedge u_\ell\|_{\mathfrak p_\jmap}=
\frac{\|u_1\wedge\cdots\wedge u_\ell\|_{\mathfrak p_\jmap}}{\|u_1\wedge\cdots\wedge u_\ell\|_{\mathfrak p_0}}.
\]
However, the last expression is independent of the choice of the basis of~$\jmap(\mathfrak{h}_v)$.
Let us now choose it so that~$u_i=\Ad(g_0^{-1})(u_{0,i})$
for~$i=1,\ldots,\ell$ and a fixed orthonormal basis~$u_{0,1},\ldots,u_{0,\ell}$
of~$\jmap_0(\mathfrak h_v)$ w.r.t.~$\mathfrak p_0$. This gives
\[
J_v\asymp\frac1{\|\wedge^\ell\Ad(g_0^{-1})(u_{0,1}\wedge\cdots\wedge u_{0,\ell})\|_{\mathfrak p_0}}
\]
and part (3) of the proposition will follow if we show that
\[
 \|\wedge^\ell\Ad(g_0^{-1}) (u_{0,1}\wedge\cdots\wedge u_{0,\ell})\|_{\mathfrak p_0}
\] 
is bounded away\footnote{This could also be seen using the more general fact that 
$\wedge^\ell\Ad (\G(F_v))(u_{0,1}\wedge\cdots\wedge u_{0,\ell})$
is a closed subset of $\wedge^\ell\mathfrak g_v$ which does not contain $0.$} from $0$ ({independently} of $\jmap$).  

To see this claim recall that the Killing form $B:=B_{\G(F_v)}$ is a $\G(F_v)$ invariant 
nondegenerate bilinear form on $\mathfrak g_v$ whose restriction to 
$\jmap_0(\mathfrak h_v)$ is nondegenerate.  {Let $Q_B$ be the quadratic
form on ~$\wedge^\ell\mathfrak g_v$ induced by $B$.  Then $|Q_B(\cdot)|$ is bounded from above by a multiple of $\|\cdot\|^2$.
Our claim follows from the fact that the value of $Q_B$ at the vector~$\wedge^\ell\Ad(g_0^{-1}) (u_{0,1}\wedge\cdots\wedge u_{0,\ell})$
is nonzero and independent of $g_0$.}
\end{proof}

 

\subsection{Finite index in volume normalization} 

Let $\mathbf{C}$ be the central kernel of $\H \rightarrow \iota(\H)$. We may identify
$g^{-1} \iota(\H(\adele)) g$ with the quotient of $\H(\adele)$ by the {\em compact} group $\mathbf{C}(\adele)$ -- it is a product of infinitely many 
finite groups.    

 The associated homogeneous space $Y=\G(F)\iota(\H(\adele)) g$ is identified with
$$ 
\iota^{-1} \Delta\backslash \H(\adele),
$$ 
where 
$$
\Delta =  \iota(\H)(F)\cap\iota(\H(\adele)).
$$  
Note that $\Delta$ is a discrete subgroup of $\iota(\H(\adele))$, which is a closed subgroup of $\G(\adele)$.
We need to compare the Haar measure on $\H(F)\backslash\H(\adele)$ (studied in this section) 
with the Haar measure on $\iota^{-1} \Delta\backslash \H(\adele)$ (used to define the volume of~$Y$).

Now $\H(F) \mathbf{C}(\adele) \subset \iota^{-1} \Delta$.
The quotient $\iota^{-1} \Delta/\H(F) \mathbf{C}(\adele)\cong\Delta/\iota(\H(F))$ is isomorphic to a subgroup $\mathcal{S}'$ of the kernel 
\[
\mathcal{S} := \ker\Bigl(H^1(F, \mathbf{C}) \rightarrow \prod_{v} H^1(F_v, \mathbf{C})\Bigr).
\]
This can be seen from the {exact sequence of pointed sets
\[ 
\H(F) \stackrel{\iota}{\rightarrow} \iota(\H)(F) \stackrel{\delta}{\rightarrow} H^1(F, \mathbf{C}),
\] 
arising from Galois cohomology,}
whereby $\Delta$ is identified with the preimage  under $\delta$ of $\mathcal{S}$. The group $\mathcal{S}$
is finite by~\cite[Thm.\ 6.15]{PR} and so is $\mathcal{S}'.$

We endow $\H(\adele)$ with the measure 
for which $H(F) \backslash \H(\adele)$ has volume $1$,  
use on $\iota(\H(\adele))$  the quotient measure by the Haar probability measure on $\mathbf{C}(\adele)$,
and use counting measure on $\Delta.$
With these choices the homogeneous space $\Delta\backslash\iota(\H(\adele)) \cong \iota^{-1} \Delta\backslash\H(\adele)$ 
has total mass
$$ 
\frac{1}{\#\mathcal{S}'} \mbox{mass of} \left( \H(F) \mathbf{C}(\adele)\backslash \H(\adele) \right)
$$
where the Haar measure on $\H(F) \mathbf{C}(\adele)$ is such that  
each coset of $\mathbf{C}(\adele)$ has measure $1$.
Together this gives that the mass of~$\Delta\backslash\iota(\H(\adele))$ 
equals $\frac{\#\mathbf{C}(F)}{\#\mathcal{S}'}$. 


Finally, the size of $\mathbf{C}(F)$ is certainly bounded above and below in terms of $\dim(\H)$,
by the classification of semisimple groups.
 As for $\mathcal{S}\supset\mathcal S'$,  it is finite\footnote{It need not itself be trivial, 
 because of Wang's counterexample
related to the Grunwald--Wang theorem.} by~\cite[Thm.\ 6.15]{PR}. 
Indeed we can give an explicit upper bound for it in terms of $\dim\H$, see the proof of~\cite[Lemma 6.11]{PR}. 
We outline the argument. {The absolute Galois group of $F$ acts on $\mathbf{C}(\bar{F})$, by ``applying Galois automorphisms
to each coordinate'';} we may choose  a Galois extension $E/F$
such that the Galois group of $E$ acts trivially on $\mathbf{C}(\bar{F})$. Then $[E:F]$ can be chosen
to be bounded in terms of $\dim(\H)$. 
By the inflation-restriction sequence in group cohomology, the  kernel of $H^1(F, \mathbf{C}) \rightarrow H^1(E, \mathbf{C})$
is isomorphic to a quotient of $H^1(\mathrm{Gal}(E/F), \mathbf{C}(\bar{F}))$, 
whose size can be bounded in terms of $\dim(\H)$.   On the other hand, the image
of $\mathcal{S}$ consists of classes in $H^1(E, \mathbf{C})$ -- i.e., homomorphisms
from the Galois group of $E$ to the abelian group $\mathbf{C}(\bar{F})$ -- 
which are trivial when restricted to the Galois group of each completion of $E$. 
Any such homomorphism is trivial, by the Chebotarev density theorem. 

Let us summarize the above discussions.

\begin{lem}
 Normalize the Haar measure on~$\H(\adele)$
 so that the induced measure on~$\H(F)\backslash\H(\adele)$ is a probability measure.
 Then the induced measure on the homogeneous set~$\G(F)\iota(\H(\adele))g$
 equals~$\frac{\#\mathbf{C}(F)}{\#\mathcal S'}\in[\frac1M,M]$, where~$M\geq 1$
 only depends on~$\dim\H$.
\end{lem}

\subsection{The volume of a homogeneous set}\label{s;vol-hom}
In view of our definition of volume 
and taking into account the choice of $\Omega_0$, 
the equation~\eqref{e;tamagawa} implies that  
\be\label{e;vol-est}
\vol(Y)  =  \frac{\#\mathbf{C}(F)}{\#\mathcal S'} D_{F'}^{\frac{1}{2}\dim\H'} \prod_{v\in\places_F} 
\left(|\omega_v|(K_v^*)\right)^{-1},
\ee
where $|\omega_v|:=\prod_{v'|v}|\omega'_{v'}|.$  

Since $K_v^*\subseteq\prod_{v'|v}K_{v'}^*$
\begin{multline}\label{eq;volume-upperbd}
\vol(Y) =\frac{\#\mathbf{C}(F)}{\#\mathcal S'}A D_{L/F'}^{\frac{1}{2}\mathfrak{s}(\Hcal')}D_{F'}^{\frac{1}{2}\dim\H'} \prod_{v\in\places_F}\Big(|\omega_v|(K_v^*)\textstyle\prod_{v'|v} 
\lambda_{v'}\Big)^{-1}\\
\gg\;D_{L/F'}^{\frac{1}{2}\mathfrak{s}(\Hcal')}D_{F'}^{\frac{1}{2}\dim\H'}
\!\prod_{v'\in\places_{F',f}}\!\!
\left(\lambda_{v'}|\omega'_{v'}|(K_{v'}^*)\right)^{-1}\!\prod_{v\in\places_{F,f}}\!\!\Big[\textstyle\prod_{v'|v}K_{v'}^*:K_v^*\Big],
\end{multline}
where we used~\eqref{e;vol-est} and~\eqref{e;prod-lambda} in the first line and
part (3) of the proposition in \S\ref{ss;k-w-star} in the second line. Let us note
the rather trivial consequence\footnote{This would also follow trivially
from the definition if only we would know that the orbit intersects a fixed compact subset.}~$\vol(Y)\gg 1$ of~\eqref{eq;volume-upperbd}. 
Below we will assume implicitly~$\vol(Y)\geq 2$ (which we may achieve by replacing~$\Omega_v$ by a smaller neighborhood 
at one infinite place in a way that depends only on~$\G$).

Let $\Sigma_{\rm ur}^{\flat}$\index{S@$\Sigma_{\rm ur}^{\flat}$ the set of unramified distorted places} be 
the set of finite places~$v$ such that $L/F$ is unramified at $v$
but at least one of the following holds  
\begin{itemize}
\item $K^*_v\subsetneq\prod_{v'\mid v}K^*_{v'},$ or 
\item there is some $v'|v$ such that 
$\H'$ is not quasi-split over $F'_{v'},$ or 
\item there is some $v'|v$ such that $K_{v'}^*$ is not hyperspecial.
\end{itemize}
Then, in view of the proposition in \S\ref{p;non-hyp-est} we find\index{kaa6@$\ref{exp-in-volume-gives-good}$,
exponent of discriminant in lower bound of volume}
 some~$\consta\label{exp-in-volume-gives-good}>0$ {such that}
\be\label{eq:sigma-flat}
{\vol(Y) 
\gg D_{L/F'}^{\frac{1}{2}\mathfrak{s}(\Hcal')}D_{F'}^{\frac{1}{2}\dim\H'} 
 2^{\#\Sigma_{\rm ur}^{\flat}}\gg D_{L}^{\ref{exp-in-volume-gives-good}}\; 2^{\#\places_{\rm ur}^\flat}};
\ee
where $\mathfrak{s}(\Hcal')\geq0$ as in~\S\ref{ss:p-formula}. We note that~\eqref{eq:sigma-flat} and the prime number theorem imply the existence of a good place in the case at hand. 

\subsection{Existence of a good place in general}\label{ss:existence-good}
Recall that the discussion in this section, so far, assumed $\H$ is $F$-almost simple.
For the details of the proof of the existence of a good place 
we return to the general case.
Thus, let $\H=\H_1\cdots\H_k$ be a direct product decomposition of $\H$ 
into $F$-almost simple factors. Let $F_j'/F$ be a finite extension so that
$\H_j={\rm Res}_{F_j'/F}(\H_j')$ where $\H_j'$ is an absolutely almost simple $F_j'$-group for all $1\leq j\leq k.$
As above $[F_j':F]$ is bounded by $\dim\H.$ 
Let $\mathcal H'_j$ and $L_j/F'_j$ be the corresponding algebraic group and number field defined as in \S\ref{volumes}.  

For any place $v\in \places_F$ let $K_v^*$ be as above. We have $K_v^*\subset \prod_{j=1}^k\prod_{v'|v}K^*_{j,v'}$ where 
$K^*_{j,v'}$
is the projection of $K^*_v$ into $\H_j'(F'_{j,v'})$ and, in particular, it is a compact open subgroup of $\H_j'(F'_{j,v'})$ when $v$ is finite.


\begin{prop}[Existence of a good place]\label{l;splitting-place}
There exists a place $w$ of $F$ such that 
\begin{itemize}
	\item[(i)] $\G$ is quasi split over~$F_w$ and split over~$\widehat{F_w}$, and~$K_w$
	is a hyperspecial subgroup of~$\G(F_w)$,
\item[(ii)] $L_j/F$ is unramified at $w$ for every $1\leq j\leq k,$
\item[(iii)] $\H_{j,w'}'$ is quasi split over $F'_{j,w'}$ (and split
over~$\widehat{F'_{j,w'}}$) for every $1\leq j\leq k$ and every $w'|w,$
\item[(iv)] $K_w^*=\prod_{j=1}^k\prod_{w'|w}K_{j,w'}^*,$ and $K_{j,w'}^*$ is hyperspecial for all $1\leq j\leq k$ and all $w'|w$, and finally
\item[(v)] $q_w \ll (\log(\vol Y))^{2}.$
\end{itemize}
\end{prop}

\proof
First note that similar to~\eqref{e;vol-est}--\eqref{eq:sigma-flat}
we have 
\[
\vol(Y) \gg 2^{I}\prod_{j=1}^k\left( D_{L_j}^{\ref{exp-in-volume-gives-good}}
\textstyle{\prod_{v'\in\places_{F_j'}}} 
\left(\lambda_{v'}|\omega'_{v'}|(K_{j,v'}^*)\right)^{-1}\right)
\]
where $I$ is the number of finite 
places where the first assertion in (iv) does not hold.
Note that at the archimedean places replacing $K_v^*$ by 
$\prod_{j,v'|v} K_{j,v'}^*$ leads to a lower bound of the volume
to which we may again apply part~(3) of the proposition in~\S\ref{p;non-hyp-est}.

As was done prior to~\eqref{eq:sigma-flat}, let $\Sigma_{\rm ur}^{\flat}$ be the set of finite places $v$ of $F$ where $L_j/F$ is unramified at $v$ for all $1\leq j\leq k$ but (iii) or (iv) does not hold. Then 
\[
\vol(Y) 
\geq \constc\label{constc-in-goodplace} 2^{\#\Sigma_{\rm ur}^{\flat}} \prod_{j=1}^k D_{L_j}^{\ref{exp-in-volume-gives-good}}.
\]

This implies the proposition in view of the prime number theorem. More concretely, suppose~${ T}=q_w$ is the 
smallest norm of the prime ideal of a good place (satisfying~(i)--(iv)) and recall that by Landau's prime ideal theorem
the number of prime ideals in~$F$ with norm below~$T$ is asymptotic to~$\frac{T}{\log T}$. 
Recall that~(i) only fails at finitely many places~$w\in\Sigma_f$ so we restrict
ourselves to places~$w$ with~$q_w\geq\constc\label{constc-rho-N}$.
Hence if~$T\geq\constc\label{constc-pnth}=\ref{constc-pnth}(F,\ref{constc-rho-N})$ we 
may assume that there are more than~$\frac{T}{2\log T}$ places with norm between~$\ref{constc-rho-N}$ and~$T$ where~(ii), (iii), or~(iv) fails. 
Combining this with the above estimate 
gives
\[
\sqrt{q_w}=\sqrt{T}\ll\frac{\ref{exp-in-volume-gives-good}}{[F:\Q]}\frac{T}{2\log T}\leq\log_2\vol(Y)-\log_2\ref{constc-in-goodplace},
\]
which implies~(v). 
\qed

\subsection{Comparison of two notions of volume}\label{ss:Borel-Prasad}
Let $\N$ be the normalizer of $\iota(\H)$ in $\G.$
By Lemma~\ref{s:stabilizer-lem} we have 
\[
\stab(\mu_{\data})=g^{-1} \iota( \H(\adele)) \N(F)  g.
\]

It will be essential for our argument in~\S\ref{sec:pigeonhole} to control
the ``interplay" between the 
volume defined using~$H_\data=g^{-1}\iota(\H(\adele))g$ (as done so far) and the volume defined using 
$g^{-1}\iota( \H(\adele_F))\N(F)g$ (which contains $H_\data$ as an open subgroup). In fact it will 
not be too difficult to reduce from $\gamma\in\N(F)$ to the case where
$\gamma\in\iota(\H(F_v))$ at finitely many places. Note that $\gamma\in\iota(\H(F_v))$ at one place implies that
$\gamma\in\iota(\H)(F).$

Let us make this more precise, recall that 
$
\vol(Y)=m_\data(H_\data\cap \Omega_0)^{-1}
$
where $\Omega_0=\left(\prod_{v\in\places_{F,\infty}}\Omega_v\right)K_f$ and 
$K_f=\prod_{v\in \places_{F,f}} K_v.$ 
We write $\jmap(\cdot)=g^{-1}\iota(\cdot)g$ and use this map also for one or several local factors.
{Note, in particular, that $H_v=\jmap(\H(F_v))$.}

Let $w$ be a good place given by 
the proposition in~\S\ref{ss:existence-good}.
We define $S=\places_\infty\cup\{w\},$ 
$F_S=\prod_{v\in S} F_v$,
 and 
 \[
 \Psi_S=\jmap(\H(F_S))\times\prod_{v\in\places\setminus S}K_v.
 \]

Let $\widetilde{H_\data}=\jmap( \H(\adele_F))N_S$ where
\[
N_S=g^{-1}\bigl\{\gamma\in\iota(\H)(F):\gamma\in\iota(\H(F_v)) \mbox{ for all }v\in S\bigr\}g,
\]
Note that $H_\data\subset\widetilde{H_\data}\subset\stab(\mu_{\data})$. We will see in \S\ref{sec:pigeonhole} that we need to compare $\vol(Y)$ with 
\[
\widetilde{\vol}(Y)=\widetilde{m_\data}(\widetilde{H_\data}\cap\Omega_0)^{-1}
\]
where $\widetilde{m_\data}$ is the unique Haar measure
induced on $\stab(\mu_\data)$ from $m_\data.$ 

Define $\Lambda:=\Psi_S\cap\jmap(\H(F))$ and $\widetilde\Lambda:=\Psi_S\cap N_S$.

\begin{lem}[Volume and index]
	The index of $\Lambda$ in $\widetilde\Lambda$ 
	controls the ratio of the above notions of volume, i.e.\ we have
	\be\label{eq:vol-tvol}
	\vol(Y)\ll[\widetilde{\Lambda}:\Lambda]\;\widetilde{\vol}(Y),
	\ee
	where the implicit constant depends on $\G(F_v)$ for $v\in\places_\infty.$
\end{lem}

\begin{proof}
	Set
	\[
	B:=\bigl\{\jmap(\H(\adele_F))\gamma:\gamma\in N_S\mbox{ and }\bigl(\jmap(\H(\adele_F))\gamma\bigr)\cap\Omega_0\neq\emptyset\bigr\}.
	\]
We will first prove that $\# B\leq [\widetilde\Lambda:\Lambda]$.
	
The properties of the good place $w$, { in particular,} guarantee that using the strong approximation theorem for $\H$ we have 
\[
\H(\adele_F)=\H(F_S)\Bigl(\prod_{v\in\places\setminus S}K_v^*\Bigr)\H(F).
\]
Let now $\jmap(\H(\adele_F))\gamma\in B,$ then there exists some 
$g_\gamma\in \jmap(\H(F_S)\prod_{v\in\places\setminus S}K_v^*)$ 
and some $\delta\in\jmap(\H(F))$ so that  
$
g_\gamma\delta \gamma\in\Omega_0.
$
Hence for all $v\in\places\setminus S$ we have $(\delta\gamma)_v\in K_v.$
This says 
\[
\delta \gamma\in N_S\cap\Bigl(\jmap(\H(F_S))\times \prod_{v\in\places\setminus S}K_v\Bigr)=\widetilde{\Lambda}.
\]
Suppose now $\delta'\in\jmap(\H(F))$ is so that $\delta'\gamma\in\widetilde\Lambda$
then 
\[
\delta'\delta ^{-1}\in\widetilde\Lambda\cap\jmap(\H(F))=\Lambda\unlhd\widetilde\Lambda.
\]
Hence we get a map from $B\to \widetilde\Lambda/\Lambda.$ 

This map is injective. Indeed let $\gamma,\gamma'\in N_S$ be as in the definition of~$B$, 
suppose that $\delta_\gamma,\delta_{\gamma'}\in\jmap(\H(F))$ are as above, and
$\delta_{\gamma}\gamma$ and $\delta_{\gamma'}\gamma'$
map to the same coset in $\widetilde\Lambda/\Lambda.$ 
Then
$\Lambda\delta_{\gamma}\gamma=\Lambda\delta_{\gamma'}\gamma',$ and in particular,   
$\jmap(\H(\adele_F))\gamma=\jmap(\H(\adele_F))\gamma'.$ In other words
we have shown that $\# B\leq [\widetilde\Lambda:\Lambda]$.

With this we now have the estimate
\[
 \widetilde{m_\data}({\widetilde{H_\data}}\cap\Omega_0)=\sum_{\jmap(\H(\adele_F))\gamma\in B}{m_\data}((h_\gamma\gamma)^{-1}\Omega_0)
\leq [\widetilde\Lambda:\Lambda]m_\data(\Omega_0^{-1}\Omega_0),
\]
where we use for every~$\jmap(\H(\adele_F))\gamma\in B$ some~$h_\gamma\in\jmap(\H(\adele))$ with~$h_\gamma\gamma\in\Omega_0$.
The claim now follows from the independence, up to a multiplicative scalar, of the notion of volume 
from the neighborhood~$\Omega_0$, see \S\ref{new-sec-on-volume}.
\end{proof}

The landmark paper~\cite{BorPr} by Borel and Prasad deals with questions similar to bounding
the above index, $[\widetilde{\Lambda}:\Lambda]$.
The setup in~\cite{BorPr} is that $\Lambda$ is defined using a coherent family of {\em parahoric} subgroups at every place.
However, our group $\Lambda$ is defined using $\{K_v^*\}$, and $K_v^*$ may only be a parahoric subgroup 
for almost all $v.$ 
We will use the strong approximation theorem to address this issue and then use~\cite{BorPr} to estimate the above index.



We will again need some reductions due to the fact 
that our group $\H$ is not necessarily absolutely  
almost simple.
Recall that $\H=\H_1\cdots\H_k$ is a product of $F$-almost simple groups
where $\H_i={\rm Res}_{F_i'/F}(\H_i')$ with $\H_i'$ an absolutely almost simple $F_i'$-group.

{Let $v\in\places\setminus S$; that is: $v$ is a finite place and $v\neq w$.}
The Bruhat-Tits building $\mathcal B_v$ of $\H(F_v)$ 
is the product of the corresponding buildings $\mathcal B_{i,v}$ for $1\leq i\leq k.$ 
The group $H_v$ is naturally identified with 
$\jmap(\prod_{i,v'|v}\H_i'(F'_{i,v'}))$ and acts on $\mathcal B_v$;
this action is identified with the action of $\prod_{i,v'|v}\H_i'(F'_{i,v'})$
on the product of the corresponding buildings $\mathcal B_{i,v'}$.
Our group $K_{i,v'}^*$ (which by definition is the group obtained by projecting $K_v^*$
into $\H'_{i}(F'_{i,v'})$) is a compact open subgroup of $\H'_i(F'_{i,v'})$
for all places $v'$ of $F_i'$ over~$v$. 
Hence by~\cite[\S3.2]{Ti} 
the fixed point set ${\rm Fix}_{i,v'}$ of $K_{i,v'}^*$ in $\mathcal B_{i,v'}$ 
is a compact and non-empty {subset}.

Let $\overline\H$ denote the adjoint form of $\H$ and let $\varphi:\H\to\overline{\H}$
be the universal covering map. 
The adjoint form $\overline\H$
is identified with $\prod_i{\rm Res}_{F_i'/F}(\overline{\H_i'})$ where $\overline{\H_i'}$
is the adjoint form of $\H_i'.$
Recall that $\widetilde\Lambda\subset g^{-1}\iota(\H)(F)g,$
and let $\varphi'_{v}:\jmap(\H)\to\overline\H$ be so that $\varphi=\varphi'_{v}\circ\jmap.$
Then $\varphi'_{v}(\widetilde\Lambda)\subset\overline{\H} (F)$.  
In particular, $\widetilde\Lambda$
naturally acts on $\mathcal B_{i,v'}$ for all $i$ and 
all places $v'$ of~$F_i'$ above $v.$ 

\begin{lem}
The fixed point set $\widetilde{{\rm Fix}_{i,v'}}$ of 
$\widetilde\Lambda$ in $\mathcal B_{i,v'}$ is a non-empty compact subset which satisfies
$\widetilde{{\rm Fix}_{i,v'}}\subset{\rm Fix}_{i,v'}$.
\end{lem}

\begin{proof}
Let $\Lambda_v$ (resp.\ $\widetilde{\Lambda}_v$) 
be the closure (in the Hausdorff topology) of the projection of $\Lambda$ (resp.\ $\widetilde\Lambda$)
in $K_v$. By the strong approximation theorem, we have 
\[
\Lambda_v=H_v\cap K_v=\jmap(K_v^*).
\]

Moreover, taking projections, we may identify both~$\Lambda$ and $\widetilde\Lambda$
as lattices in $\jmap(\H(F_S)).$ Therefore, we have $[\widetilde\Lambda_v:\Lambda_v]\leq [\widetilde\Lambda:\Lambda]<\infty$. 

Hence, using~\cite[\S3.2]{Ti}, the fixed point set $\widetilde{{\rm Fix}_{i,v'}}$ of 
$\widetilde\Lambda$ in $\mathcal B_{i,v'}$ is a non-empty compact subset which satisfies
$\widetilde{{\rm Fix}_{i,v'}}\subset{\rm Fix}_{i,v'}$ as claimed.
\end{proof}



Let us fix, for every~$v\in \Sigma\setminus S$, one point in $\mathcal B_v$ which is fixed by~$\widetilde\Lambda$.
This determines a subset $\Phi_{i,v'}$ of the affine root system $\Delta_{i,v'}.$ 
The collection $\{\Phi_{i,v'}\}$ gives us a coherent collection of parahoric subgroups $P_{i,v'}\subset\H'_i(F'_{i,v'}).$ 
{For every $v\in \Sigma\setminus S$, let $\widetilde P_{v}$ denote the stabilizer of $\prod_{v'\mid v}\Phi_{i,v'}$ in $\jmap(\H)(F_v)$.}
We define two subgroups 
\begin{align*}
 \Lambda'&= \jmap\Bigl(\prod_i\H_i'(F_i')\cap\bigl(H'_S\times\prod_{i,v'\nmid w}P_{i,v'}\bigr)\Bigr),\\
  \widetilde\Lambda'&=N_{{\jmap(H'_S)\times\prod_{v\neq w}\widetilde P_{v}}}(\Lambda')
\end{align*} 
where $H'_S=\prod_i\prod_{v'|v, v\in S}\H'_i(F'_{i,v})$.

Note that $\Lambda\subset \Lambda'$ and $\widetilde\Lambda\subset \widetilde\Lambda'$ by the construction of the parahoric subgroups.\footnote{To verify the second inclusion, for example, we first verify that $\widetilde{\Lambda}$ 
	belongs to $j(H'_S) \times \prod_{v \not\in S} \widetilde P_{v}$: it projects to $j(H'_S)$ at places in $S$ because
$\widetilde{\Lambda} \subset \Psi_S$, and it projects to the $\widetilde{P}_v$ by the way they were chosen. We then verify $\widetilde{\Lambda}$ normalizes $\Lambda'$. Because of the inclusion  $\widetilde{\Lambda} \subset N_S$
we can regard $\widetilde{\Lambda}$ as acting on   $\mathbf{H}(F) = \prod_{i} \mathbf{H}_i'(F_i')$. 
It preserves the subset of this defined by intersecting with $ (H_S' \times \prod P_{i,v'})$ because each $\tilde{P}_v$ normalizes $\prod_{i, v'|v} P_{i,v'}$. }
Moreover, $\Lambda'$ is a finite index subgroup of $\widetilde\Lambda'$, see~\cite[Prop.~1.4]{BorPr}.

Recall the definition of the fields $L_i/F'_i$ from~\S\ref{ss:existence-good}. As we have done before we define a subset $\places^\flat\subset\places_{F,f}$ as follows. Let $\places_{\rm ur}^{\flat}$ be the set of finite places $v$ of $F$ where $L_j/F$ is unramified at $v$ for all $1\leq i\leq k$ but 
at least one of the following fails  
\begin{enumerate}
\item $\H_{i,v'}'$ is quasi split over $F'_{i,v'}$ (and split
over~$\widehat{F'_{i,v'}}$) for every $1\leq i\leq k$ and every $v'|v,$
\item $K_{i,v'}^*$ is hyperspecial for all $1\leq i\leq k$ and all $v'|v$
and $K_v^*=\prod_{i,v'|v}K_{i,v'}^*.$
\end{enumerate}
Define $\places^\flat_{\rm rm}$ to be the set of places $v\in\places_{F,f}$
so that $L_i/F$ is ramified at $v$ for some $1\leq i\leq k.$ 
Put $\places^{\flat}=\places_{\rm ur}^{\flat}\cup \places_{\rm rm}^{\flat}$;
note that $\places^{\flat}\cap S=\emptyset$.

Let us note that if $K_{i,v'}^*$ is hyperspecial for all $1\leq i\leq k$ and all $v'|v$
but $K_v^*\neq\prod_{i,v'|v}K_{i,v'}^*,$ then 
\be\label{new-eq-BG}
 \Bigl[\prod_{i,v'|v}K_{i,v'}^*:K_v^*\Bigr]\geq p_v.
\ee
Indeed the reduction mod $v'$ of the group scheme corresponding to $K_{i,v'}^*$
is an almost simple group and $K_{i,v'}^*$ maps onto the $k_{i,v'}$ points of this group.
Let~$R$ be the semisimple group obtained from~$\prod_{i,v'|v}K_{i,v'}^*$
by taking it modulo the first congruence subgroup.
By construction of~$K_{i,v'}^*$ the image of $K_v^*$ modulo the first congruence subgroup of $\prod_{i,v'|v}K_{i,v'}^*$,
let us call it $R',$ projects onto each factor of $R$. If $R'$ does not equal $R,$ then~\eqref{new-eq-BG} follows. 
If these two equal each other, then an argument
as in the proof of part~(2) of the proposition in \S\ref{ss;k-w-star} implies~\eqref{new-eq-BG}. 

This observation together with
parts~(1) and~(2) of the proposition in \S\ref{ss;k-w-star} implies that for all $v\in\places_{\rm ur}^\flat$
we have
\be\label{eq:local-ind-improve}
\bigl(\prod_{i,v'|v}\lambda_{i,v'}|\omega'_{i,v'}|\bigr)(K_v^*)\leq \frac{p_v+1}{p_v^2}
\ee
if~$q_v>13$.

\begin{lem}[Bound on index]
The index of $\Lambda$ in $\widetilde\Lambda$ satisfies the bound
\be\label{eq:index-cont}
[\widetilde\Lambda:\Lambda]\leq N^{\ref{field-degree}+\ref{index-exp}(\#\places^\flat)}\prod_i 2h_{L_i}^{a}(D_{L_i/F'_i})^{b}
\ee
where 
\begin{itemize}\index{kaa7@$\ref{field-degree}, \ref{index-exp}$,
	two constants in the comparison of volumes}
\item $\consta\label{field-degree}=\sum_i[L_i:\Q]$ and 
$\consta\label{index-exp}=2\sum_i[F_i':F],$
\item $h_{L_i}$ is the class number of ${L_i}$
\item $a=2$ if $\H'_i$ is an inner form of a split group of type 
$D_r$ with $r$ even resp.\ $a=1$ otherwise, and finally
\item $b=1$ if $\H'_i$ is an outer form of type $D_r$ with 
$r$ even resp.\ $b=0$ otherwise.
\end{itemize}
\end{lem}

\begin{proof}
We first consider the map $\widetilde\Lambda/\Lambda\to\widetilde\Lambda'/\Lambda'$.
This is an injective map. Indeed, if $\gamma\in\Lambda'\cap\widetilde\Lambda$, then
$\gamma\in\widetilde\Lambda\subset\Psi_S$ and $\gamma\in\Lambda'\subset\jmap(\H(F))$. Hence
\[
\gamma\in\Psi_S\cap\jmap(\H(F))=\Lambda.
\]
We, thus, get that
$
[\widetilde\Lambda:\Lambda]\leq[\widetilde\Lambda':\Lambda'].
$

%

Bounding $[\widetilde\Lambda':\Lambda']$ is rather non-trivial. This is done in~\cite[\S2 and \S5]{BorPr}, and we have
\[
[\widetilde\Lambda':\Lambda']\leq \prod_i 2h_{L_i}^{a}N^{[L_i:\Q]+2[F_i':F]\#\places^\flat}(D_{L_i/F'_i})^{b},
\] 
with $a,b$ and $h_{L_i}$ as in the statement of the lemma. 
\end{proof}

The following is crucial in the application of the volume
for the pigeon hole argument in \S\ref{sec:pigeonhole}. 

\begin{prop}[Equivalence of volume definitions]
	The above two notions of volume are related in the sense that there exists some $\consta\label{vol-tvol-cons}>0$
	so that\index{kaa7@$\ref{vol-tvol-cons},$ exponent in the comparison of algebraic and stabilizer volume}
\be\label{eq:vol-tvol-final}
\vol(Y)^{\ref{vol-tvol-cons}}\leq\widetilde{\vol}(Y)\leq\vol(Y),
\ee
if $\vol(Y)$ is sufficiently large depending only on the dimensions $\dim\G$ and $[F:\Q].$
\end{prop}

\proof
For any number field $E$ we have
\[
h_{E}\leq 10^2\bigl(\tfrac{\pi}{12}\bigr)^{[E:\Q]}D_E, 
\]
see e.g.\ equation (7) in the proof of~\cite[Prop.\ 6.1]{BorPr}.
Also recall that 
\[
 D_{L_i/F'_i}=D_{L_i}/D_{F'_i}^{[L_i:F'_i]}\geq 1.
\]

These imply the following\footnote{See~\cite[Prop.\ 6.1]{BorPr}, and also~\cite[Prop.\ 3.3]{Belo}, for more general statements.} estimates (for the field related quantities coming in part from \eqref{eq:index-cont} and~\eqref{eq;volume-upperbd}).
If $\Hcal'_i$ is an outer form then $a=1$ and $b\leq 1.$ Hence in this case we have
\[
h_{L_i}^{-a}D_{F'_i}^{\frac12\dim\H'_i}D_{L_i/F'_i}^{\frac12\mathfrak s(\Hcal'_i)-b}\gg D_{F_i'}^{\frac12\dim\H'_i-c}D_{L_i/F'_i}^{\frac12\mathfrak s(\Hcal'_i)-2},
\]
where~$c=[L_i:F_i']$ equals $3$ if~$\Hcal_i'$ is a triality form of~$D_4$ resp.\ $2$ otherwise.

Suppose $\H'_i$ is an inner form of type other than $A_1.$
Then $L_i=F_i',$ $a\leq 2$ and $b=0.$ Together we get
\[
h_{L_i}^{-a}D_{F'_i}^{\frac12\dim\H'_i}D_{L_i/F'_i}^{\frac12\mathfrak s(\Hcal'_i)-b}\gg D_{F_i'}^{\frac12\dim\H'_i-2}D_{L_i/F'_i}^{\frac12\mathfrak s(\Hcal'_i)}=D_{F_i'}^{\frac12\dim\H'_i-2}.
\]

Finally let $\H'_i$ be an inner form of type $A_1,$ then $L_i=F'_i$, $a=1$, $b=0$, and we have 
\[
h_{L_i}^{-a}D_{F'_i}^{\frac12\dim\H'_i}D_{L_i/F'_i}^{\frac12\mathfrak s(\Hcal'_i)-b}\gg D_{F_i'}^{\frac12\dim\H'_i-1}=D_{F'_i}^{1/2}.
\]

These estimates together with $\mathfrak s(\Hcal'_i)\geq 5$
when $\Hcal_i'$ is an outer form and $L_i=F'_i$ when $\Hcal'_i$ is an inner form give
\be\label{eq:discrim-bound}
h_{L_i}^{-a}D_{F'_i}^{\frac12\dim\H'_i}D_{L_i/F'_i}^{\frac12\mathfrak s(\Hcal'_i)-b}
  \gg \bigl(D_{F_i'}D_{L_i/F'_i}\bigr)^{1/2}.
\ee

We now prove~\eqref{eq:vol-tvol-final} and note that~$\widetilde\vol(Y)\leq\vol(Y)$
follows directly from the definition. 

For the opposite inequality we argue as follows.
\begin{align}
\notag\widetilde\vol(Y)&\geq\frac{1}{[\widetilde\Lambda:\Lambda]}\vol(Y)&&\text{by~\eqref{eq:vol-tvol}}\\
\label{eq:expanded}&\geq \Bigl(N^{\ref{field-degree}+\ref{index-exp}(\#\places^\flat)}\prod_i 2h_{L_i}^{a}(D_{L_i/F'_i})^{b}\Bigr)^{-1}\vol(Y)
&&\text{by~\eqref{eq:index-cont}}.
\end{align}
Now by~\eqref{eq:discrim-bound} we have
\be\label{eq:disc-bd-used}
\Bigl(\prod_i 2h_{L_i}^{a}(D_{L_i/F'_i})^{b}\Bigr)^{-1}\gg \prod_i D_{F'_i}^{-\frac12\dim\H'_i+\frac12}D_{L_i/F'_i}^{-\frac12\mathfrak s(\Hcal'_i)+\frac12}.
\ee
Moreover, by~\eqref{eq:local-ind-improve} we have
\be\label{eq:local-index-improved-used}
N^{-\ref{field-degree}-\ref{index-exp}(\#\places^\flat)}\gg\Bigl(\prod_{\places^\flat}\bigl(\prod_{i,v'|v}\lambda_{i,v'}|\omega'_{i,v'}|\bigr)(K_v^*)\Bigr)^{-\frac12}
\ee
Note that for the
few bad places with~$q_v<13$ the power of~$N$ simply becomes
an implicit multiplicative constant.

In view of~\eqref{eq:disc-bd-used} and~\eqref{eq:local-index-improved-used}, the lower bound in~\eqref{eq:vol-tvol-final} 
follows from~\eqref{eq:expanded} and the asymptotic in~\eqref{eq;volume-upperbd}.
\qed


\section{Algebraic properties at a good place}\label{sec:algebra-at-good}

As in~\S\ref{construction-MASH} and~\S\ref{proofnotn} we 
let~$Y=Y_\data$ be the MASH set for the data~$\data=(\H,\iota,g_{\data})$ and let~$\G$ be the ambient algebraic group; in particular 
{ we are assuming that $\H$ is simply connected and that} $\iota(\H)$
is maximal in $\G.$ In this section we collect algebraic properties of the MASH set~$Y$ and its associated
groups at a good place~$w$. These properties may be summarized as saying that the acting group is not distorted at $w$ and will be needed in the dynamical argument
of the next section.

\subsection{Good places}\label{ss;good-place}
We say a place $w\in\places_f$ is {\em good} (for $Y$)  when\index{w@$w$, starting with~\S\ref{ss;good-place} a good place for~$Y=Y_\data$}  
\begin{itemize}
\item $w$ satisfies (i)--(iv) in the proposition concerning the existence of good places in \S\ref{l;splitting-place},
\item in particular $\G$ and $\iota(\H)$ are quasi-split over $F_w$ and split over $\widehat{F_w},$ the maximal unramified extension, 
and\footnote{For the last claim increase
in the proof of~\S\ref{ss:existence-good} the value of~$\ref{constc-rho-N}$ accordingly.}
\item ${\rm char}(k_w)\gg_{N,F} 1,$ where $\rho(\G)\subset\SL_N$ as before. 
\end{itemize}

We note that the last property of a good place as above allows us e.g.\ to avoid difficulties arising from 
the theory of finite dimensional representations of algebraic groups over fields with ``small" characteristic.   

By the proposition in \S\ref{ss:existence-good} we have: there is a good place $w$ satisfying\footnote{The good 
place for the proof of Theorem~\ref{tautheorem} is found as in \S\ref{sec:tau-general}: There are infinitely many places
where~$\G$ splits, and all properties of a good place for the maximal subgroup~$\H=\{(h,h):h\in\G\}<\G\times\G$
are satisfied for almost all places.}
$$q_w \ll (\log(\vol Y))^{2}.$$

Let $g_{\data,w}\in \G(F_w)$ denote the component of~$g_\data$ at~$w$.\index{gw@$g_{\data,w}$, 
component of~$g_\data$ at good place~$w$} 
For simplicity in notation we write $\jmap_w:\H\to\G$ for the homomorphism
defined by $\jmap_w(\cdot)=g_{\data,w}^{-1}\iota(\cdot)g_{\data,w}$ at the good place~$w$\index{jw@$\jmap_w(\cdot)=g_{\data,w}^{-1}\iota(\cdot)g_{\data,w}$,
homomorphism at good place~$w$}.
We define the group $H_w^*=\H(F_w)$ and\index{Hw@$H_w^*=\H(F_w)$ is the group of~$F_w$-points
of~$\H$ at the good place~$w$} 
recall from \S\ref{ss;k-w-star} the notation\index{Kw@$K_w^*=\iota^{-1}(g_wK_wg_w^{-1})<H_w$, 
hyperspecial subgroup at the good place~$w$} $K_w^*=\jmap_w^{-1}(K_w).$
 It is worth mentioning again that $\jmap_w(H_w^*)$
does not necessarily equal $\jmap_w(\H)(F_w)$ or the group of $F_w$-points of any algebraic group.

\subsection{Compatibility of hyperspecial subgroups}\label{ss:compatibility}

By the properties of the good place $\G$ and $\H$ are quasi-split 
over $F_w$ and split over $\widehat{F_w}.$  
{ Furthermore,} $K_w$ and $K_w^*$ are  hyperspecial subgroups of $\G(F_w)$ and $H_w^*=\H(F_w)$ respectively. 

Let ${\rm vert}$ and ${\rm vert}^*$ denote the vertices corresponding to $K_w$
and $K_w^*$ in the respective buildings.
As was recalled in \S\ref{ggo}, Bruhat-Tits theory associates smooth group schemes $\Gfrak_w$ and $\Hfrak_w$ to ${\rm vert}$ and ${\rm vert}^*$ in 
$\G(F_w)$ and $\H(F_w)$ respectively, so that $K_w=\Gfrak_w(\order_w)$ and 
$K_w^*=\Hfrak_w(\order_w).$ 
Since~$\jmap_w$ is a homomorphism,~$\jmap_w(\H)(\widehat{F_w})$ acts on the building of~$\H(\widehat{F_w})$. 

Let $p_w$\index{pw@$p_w$, the residue characteristic at place $w$} be the prime number so that $w\mid p_w,$ i.e.\ $p_w={\rm char}(k_w).$

%

\begin{lem}
For $p_w\gg1$ the following hold.
The stabilizer of~${\rm vert}^*$ in~$\jmap_w(\H)(\widehat{F_w})$ equals~$\jmap_w(\Hfrak_w(\widehat{\order_w}))$, i.e.\ the image of the stabilizer of~${\rm vert}^*$
 in~$\H(\widehat{F}_w)$ under the map~$\jmap_w$. 
 Moreover, the homomorphism $\jmap_w$ extends to a closed immersion from $\Hfrak_w$
 to $\Gfrak_w$ which we continue to denote by $\jmap_w$. 
\end{lem}


\begin{proof}
Let $h\in \jmap_w(\H)(\widehat{F_w})$ be in the stabilizer of ${\rm vert}^*$ in $\jmap_w(\H)(\widehat{F_w}).$
{ In the following paragraph we will use similar arguments to that of \cite[\S 2]{BorPr} and we refer to that for unexplained notions.} Then the induced action of $h$ on the affine root system fixes the vertex {corresponding to} ${\rm vert}^*,$ this implies $h$ acts trivially on the affine root system, see~\cite[1.8]{IwMat}.   
It follows from ~\cite[1.8]{IwMat} and ~\cite[Prop.\ 2.7]{BorPr} that
$h\in\jmap_w(\H(\widehat{F_w}))$ (i.e.\ represents the 
trivial cohomology class {with reference to \cite[\S 2.5(1)]{BorPr}}), at least
for  $p_w$ large enough. 
Now by~\cite[3.4.3]{Ti} the smooth scheme structure corresponding to the stabilizer of ${\rm vert}^*$
in $\H(\widehat{F_w})$ is deduced from $\Hfrak_w$ by base change from $\order_w$ to $\widehat{\order_w}.$
Therefore,     
$\jmap_w(\Hfrak_w(\widehat{\order_w}))$ equals the stabilizer of ${\rm vert}^*$
in $\jmap_w(\H)(\widehat{F_w})$ which is the first claim of the lemma.



We now claim
\begin{equation} \label{inclusion-tbp} \jmap_w(\Hfrak_w(\widehat{\order_w}))\subset\Gfrak_w(\widehat{\order_w}).\end{equation} 
Assuming the claim, let us finish the proof. 
{ By the criterion described in  \cite[1.7.3, 1.7.6]{BrTi2}},  the homomorphism $\jmap_w$ extends 
to an $\order_w$-morphism  $\tilde{\jmap}_w:\Hfrak_w\to \mathfrak{G}_w$ which by~\cite[Cor.\ 1.3]{Pr-JKY}
is a closed immersion. 
 
Let us now turn to the proof of~\eqref{inclusion-tbp}. It suffices to prove 
\[
\rho\circ\jmap_w(\Hfrak_w(\widehat{\order_w}))\subset \SL_N(\widehat{\order_w}).
\]
Put $\rho_w:=\rho\circ\jmap_w.$ Then
$\rho_w(\Hfrak_w(\widehat{\order_w}))$ is a bounded subgroup 
of $\SL_N(\widehat{F_w}),$ hence it is contained in a { maximal parahoric subgroup $P$ of $\SL_N(\widehat{F_w})$ --
we may even suppose that $P$ is a hyperspecial parahoric subgroup.}
Let us assume $P\neq\SL_N(\widehat{\order_w})$ as there is nothing to prove otherwise. 

 {
 Inside the building of $\SL_N$
	over $\widehat{F_w}$, let $v_0$ be the vertex corresponding to $P$
	and $v$ the vertex corresponding to $\SL_N(\widehat{\order_w})$.
	Choose a geodesic, inside this building, connecting the vertex $v_0$
	with the vertex $v$.  Consider the collection $\mathcal{C}$
	of all  facets whose interior meets this geodesic path.  Any element of $P \cap 
		\SL_N(\widehat{\order_w})$ fixes all facets in $\mathcal{C}$
	-- recall that, in the building for $\SL_N$, fixing a facet setwise
	implies fixing it pointwise. 
	Therefore, $\rho_w(\Hfrak_w(\order_w))$ fixes all the facets in $\mathcal{C}$.

			In this language, we  must show that  $\rho_w(\Hfrak_w(\widehat{\order_w}))$
			fixes the vertex $v$.   The union of facets in $\mathcal{C}$
			is connected, and so its $1$-skeleton is connected; thus we may choose a path 
			$v_0, v_1, \dots, v_{\ell} = v$ starting from the vertex $v_0$ 
			and ending at $v$, where any two adjacent vertices belong to a common chamber. 
		Let $P_i$ be the stabilizer of $v_i$. We have
		seen above that $\rho_w(\Hfrak_w(\order_w)) \subset P_i$ 
and we will argue inductively that $\rho_w(\Hfrak_w(\widehat{\order_w}))\subset P_i$ for all $0\leq i\leq\ell.$ }


For each $0\leq i<\ell $ there is an element $g_i\in {\rm PGL}_N(\widehat{F_w})$
so that $P_i=g_i\SL_N(\widehat{\order_w})g_i^{-1}.$ 
Denote by $\SL_{N,g_i}$
the corresponding scheme structure, that is $P_i=\SL_{N,g_i}(\widehat{\order_w}).$
Assume $\rho_w(\Hfrak_w(\widehat{\order_w}))\subset \SL_{N,g_i}(\widehat{\order_w}).$
By~\cite[1.7.3, 1.7.6]{BrTi2} the homomorphism $\rho_w$ extends 
to {a $\widehat{\order_w}$}-morphism $\tilde{\rho}_w:\Hfrak_w\to\SL_{N,g_i}$ which by~\cite[Cor.\ 1.3]{Pr-JKY}
is a closed immersion.
Let $\red_w(\tilde{\rho}_w):\underline{\Hfrak_w}(\widehat{k_w})\to\SL_{N, g_i}(\widehat{k_w})$
be the corresponding homomorphism on special fibers. 
Then the finite group $\red_w(\tilde{\rho}_w)\bigl(\underline{\Hfrak_w}(k_w)\bigr)$ is contained in $\red_w(P_i\cap P_{i+1})$
which is a proper parabolic subgroup of $\SL_{N, g_i}$; and we must show that the same is true with $k_w$ replaced by $\widehat{k_w}$.

Since each proper parabolic subgroup of $\SL_N$ can be expressed as the intersection of certain subspace stabilizers, our assertion
reduces to the following: Regarding  $\underline{\Hfrak_w} $ as acting on an $N$-dimensional representation via $\red_w(\tilde{\rho}_w)$, 
and if  $p_w \gg_{N} 1$, 
the following holds: 
\[ 
\text{If  $\underline{\Hfrak_w}(k_w)$ fixes a  subspace $W \subset \widehat{k_w}^N,$   then $\underline{\Hfrak_w}(\widehat{k_w})$ also fixes $W$.} 
 \]
 Passing to exterior powers,  and using the semisimplicity we reduce to the same statement 
with the subspace $W$ replaced by a vector $v$.  But $\underline{\Hfrak_w}$ is generated by unipotent one-parameter subgroups, i.e.
by closed immersions $u: \mathbb{G}_a \rightarrow \underline{\Hfrak_w}$.   
 Because the map $\underline{\Hfrak_w} \rightarrow \SL_{N,g_i}/\widehat{k_w}$ is a closed immersion, 
we can regard $u$ as a closed immersion $\mathbb{G}_a \rightarrow \SL_{N,g_i}/\widehat{k_w}$ also, and 
from that we see that the coordinates of $u(t) v$ are polynomials { in} $t$ whose degree is bounded in terms of $N$.
Since  these polynomials vanish  identically
 for $t \in k_w$,  we see that, for $p_w \gg 1$, they vanish identically on $\widehat{k_w}$ too. 
 \end{proof}

In view of the above lemma, and abusing the notation, $\jmap_w(\Hfrak_w)$
is a smooth subgroup scheme of $\Gfrak_w.$  
Taking reduction mod $w$ on $\Gfrak_w,$ which induces 
the reduction map on $\jmap_w(\Hfrak_w),$ we have 
$\underline{\jmap_w(\Hfrak_w)}\subset\underline{\Gfrak_w}$ 
for the corresponding algebraic groups over $\widehat{k_w}$ (the residue field of $\widehat{F_w}$, i.e. the algebraic closure of 
$k_w$).

\begin{lemma}[Inheritance of Maximality]\label{c;red-max}
Let~$\iota(\H)<\G$, the place~$w\in\Sigma_f$ and~$g_{\data,w}\in G_w$ be as above. Then
$\underline{\jmap_w(\Hfrak_w)}\;$ is a maximal connected algebraic subgroup of $\underline{\Gfrak_w}$
provided that $p_w$ is large enough.
\end{lemma} 

\begin{proof}
First note that the subgroups $\Gfrak_w(\widehat{\order_w})$
and $\jmap_w(\Hfrak_w)(\widehat{\order_w})$ are hyperspecial subgroups of $\G(\widehat{F_w})$ and
$\jmap_w(\H)(\widehat{F_w}),$ see~\S\ref{ggo} (1) as well as~\cite[2.6.1 and 3.4.3]{Ti}. In particular, $\underline{{\jmap_w(\Hfrak_w)}}$ 
and $\underline{\Gfrak_w}$ are connected by~\S\ref{ggo} (4).

Let us also recall our assumption that both $\G$ and $\H$ 
split over $\widehat{F_w}.$
Therefore, by~\S\ref{ggo} (2) we have: ${\jmap_w(\Hfrak_w)}$ and 
$\Gfrak_w$ are $\widehat{\order_w}$-Chevalley group
schemes with generic fibers $\jmap_w(\H)(\widehat{F_w})$ and 
$\G(\widehat{F_w})$ respectively. 

We now show that this and maximality of $\iota(\H)$ in $\G$ implies that the subgroup $\underline{{\jmap_w(\Hfrak_w)}}$ is a maximal subgroup of $\underline{\Gfrak_w}.$

We first claim that $\underline{{\jmap_w(\Hfrak_w)}}$
is not contained in any proper parabolic subgroup of $\underline{\Gfrak_w}.$ 

To see this, let $\mathsf H$ (resp.\ $\mathsf G$) denote the  split Chevalley group over 
$\mathbb Z$ which has the same type as $\H$ (resp.\ $\G$).   These are affine schemes. 
For an arbitrary ring $R$, we denote by $\mathsf H_R$ the base change of $\mathsf{H}$ to $R$, and similarly for $\mathsf{G}$.

We want to make an argument involving the ``scheme of homomorphisms from $\mathsf H$ to $\mathsf G$.''
Such a scheme, whose $R$-points are canonically in bijection with 
homomorphisms of $R$-group schemes $\mathsf H_R \rightarrow \mathsf G_R$
is constructed in~\cite{SGA3}, but it is too big for us, because it has many components corresponding
to Frobenius twists of morphisms. Thus we use a  home-made variant:

Let $\mathcal{O}_H$ and $\mathcal{O}_G$ be the ring of global sections of the structure sheafs
of $\mathsf H$ and $\mathsf G$, respectively.   Fix generators $f_1, \dots, f_r$ for $\mathcal{O}_G$ as a $\Z$-algebra. 
There are finitely many conjugacy classes of homomorphisms $\mathsf{H}_{\overline{\Q}} \rightarrow \mathsf{G}_{\overline{\Q}}$. Therefore, 
there exists a finite dimensional $\mathsf{H}(\overline{\Q})$-stable sub vector space $\widetilde{M} \subset \mathcal{O}_H \otimes \overline{\Q}$ with the following property:
For any homomorphism $\rho:\mathsf H_{\overline{\Q}} \rightarrow \mathsf G_{\overline{\Q}}$, 
the pullback $\rho^* f_i$ belongs to $\widetilde{M}$.
Write $M = \widetilde{M} \cap \mathcal{O}_H$.

Let $\mathcal{S}$ be the affine scheme defined thus: 
an $R$-point of $\mathcal{S}$ is a  homomorphism of Hopf algebras 
$\rho^*: \mathcal{O}_G \otimes R \rightarrow \mathcal{O}_H \otimes R$
such that  $\rho(f_i) \subset M  \otimes R$. Said differently, $\mathcal{S}(R)$ parameterizes  homomorphisms of $R$-group schemes $\rho: \mathsf{H}_R  \rightarrow
\mathsf{G}_R$ with the finiteness property just noted, 
i.e. \begin{equation} \label{finiteness-old} \rho^* f_i \in  M \otimes R, \ \ 1 \leq i \leq r.\end{equation}
It is easy to see by writing out equations that  this functor is indeed represented by  a scheme of finite type over $\Z$.

 If $R$ is an integral domain whose quotient field $E$ has characteristic zero,
 then $\mathcal{S}(R)$ actually classifies arbitrary homomorphisms $\mathsf{H}_R \rightarrow \mathsf{G}_R$ 
 (i.e., there is no need to impose the condition \eqref{finiteness-old}). 
This is because an arbitrary homomorphism $\mathsf{H}_E \rightarrow \mathsf{G}_E$ 
has the property \eqref{finiteness-old}, since we can pass from $\overline{\Q}$ to $E$ by means
of the Lefschetz principle.

%
%

Next, let $\mathcal{M}$ be the projective smooth $\Z$-scheme of parabolic subgroups of $\mathsf{G}$ (see
\cite[Theorem 3.3, Expos{\'e} XXVI]{SGA3}) 
and let $\mathcal Y\subset \mathcal S\times\mathcal M$ be a scheme of finite type over $\mathbb Z$
defined as follows. 
\[
\mathcal Y:=\{(\rho,\mathsf P): \rho(\mathsf H)\subset\mathsf P\}.
\]
where the condition ``$\rho(\mathsf{H}) \subset \mathsf{P}$'' means, more formally,
that the pull-back of the ideal sheaf of $\mathsf{P}$ under $\rho^*$ is identically zero.

In view of the main theorem in~\cite{Greenberg} we have, for all $p_w\gg1$ 
the reduction map $\mathcal Y(\widehat{\order_w})\to\mathcal Y(\widehat{k_w})$ is surjective. This together with our assumption that $\underline{{\jmap_w(\Hfrak_w)}}(\widehat{k_w})$ is contained in a proper 
parabolic subgroup of $\underline{\Gfrak_w}(\widehat{k_w})$
implies that there exists some $f:\Hfrak_w\to \Gfrak_w$ and some parabolic $\mathbf P$
of $\G$ so that $f(\H)\subset \mathbf P,$ moreover, the reductions of $f$ and $\jmap_w$ coincide.
 A contradiction will follow if we verify that $f(\H)$ is conjugate to $j_w(\H)$ (which is conjugate to $\iota(\H)$ by definition). 

Both $f$ and $j_w$ define $\widehat{\mathfrak{o}_w}$-points of $\mathcal{S}$
and their reductions to $\widehat{k_w}$-points coincide.   
We will deduce from this that $f(\H)$ and $j_w(\H)$ must actually be conjugate, as follows: 
 
 By an infinitesimal computation (which we omit) 
 the geometric generic fiber $\mathcal{S} _{\overline{\Q}}$ (i.e., the base-change of $\mathcal{S}$ to $\overline{\Q}$)  is smooth,
 and moreover the orbit map of $\mathsf{G}_{\overline{\Q}}$ is surjective on each tangent space. 
Therefore, each 
 connected component  of $\mathcal{S}_{\overline{\Q}}$ is a single orbit of $\mathsf{G}_{\overline{\Q}}$. 
 
 Let $\mathcal{S}_1, \dots, \mathcal{S}_r$ be these geometric connected components.
We can choose a finite extension $E \supset \Q$ so that every $\mathcal{S}_i$ 
is defined over $E$ and also has a $E$-point, call it $x_i$. 
For simplicity, we suppose that $E=\Q$, the general case being similar, but notationally more complicated. 

By ``spreading out,'' there is an integer $A$
and a decomposition into disjoint closed subschemes:
 $$ \mathcal{S} \times_{\Z} \mathrm{Spec} \  \Z[\tfrac{1}{A}] = \coprod \mathcal{S}_i$$
i.e.\ $\mathcal{S}_i$ is a closed subscheme, the different $\mathcal{S}_i$ are disjoint, 
 and the union of $\mathcal{S}_i$ is the left-hand side.  Note that each $\mathcal{S}_i$ is both open and closed inside the left-hand side. 
  
 In particular, if $p_w > A$, any $\widehat{\mathfrak{o}_w}$-point of $\mathcal{S}$ 
 will necessarily factor through some $\mathcal{S}_i$. Therefore, if two $\widehat{\mathfrak{o}_w}$-points
 of $\mathcal{S}$ have the same reduction, they must factor through the same $\mathcal{S}_i$. 
 In particular, the associated $\widehat{F_w}$-points of $\mathcal{S}$ belong to the same $\mathcal{S}_i$,
 and therefore to the same geometric $\mathbf{G}$-orbit (i.e., the same orbit over $\overline{\Q}$).   This implies that  
that  $f(\H)$ and $j_w(\H)$ were conjugate inside $\mathbf{G}(\overline{\Q})$, giving the desired contradiction.   

Let now $\underline{\mathfrak{S}}$ be a maximal, proper, connected subgroup of $\underline{\Gfrak_w}$ 
so that $\underline{{\jmap_w(\Hfrak_w)}}\subset\underline{\mathfrak{S}}\subset\underline{\Gfrak_w}.$ 
Then by~\cite[Cor.~3.3]{BoTi-Unip} either $\underline{\mathfrak{S}}$ is a parabolic subgroup or it is reductive. 
In view of the above discussion $\underline{\mathfrak{S}}$ must be reductive, and by the above claim in fact semisimple.
Hence there is an isomorphism  $f:{\mathsf S} \times_{\mathbb{Z}} \widehat{k_w} \to\underline{\mathfrak{S}}$
where ${\mathsf S}$ denotes the Chevalley group scheme over $\mathbb Z$ 
of the same type as $\underline{\mathfrak{S}}.$

We are now in a similar situation to the prior argument, i.e. we will lift the offending subgroup $\underline{\mathfrak{S}}$ to characteristic zero using
\cite{Greenberg}. Let $\mathsf{H}, \mathsf{S}, \mathsf{G}$ be split Chevalley groups over $\Z$
of the same type as $\Hfrak_w, \underline{\mathfrak{S}}, \Gfrak_w$. 
Consider the $\Z$-scheme  parameterizing pairs  of homomorphisms 
$$(\rho_1: \mathsf{H} \rightarrow \mathsf{G}, \rho_2: \mathrm{Lie}( \mathsf{S}) \rightarrow 
\mathrm{Lie}(\mathsf{G})) \mbox{ with }  \mathrm{image}(d \rho_1) \subset \mathrm{image}(\rho_2),$$
where we impose the same finiteness conditions of $\rho_1$ as in the prior argument. 
  
The pair $(\jmap_w, {f})$, together with identifications of $\underline{\Hfrak_w}$ and $\underline{\mathfrak{S}}$ with $\mathsf{H}$ and $\mathsf{S}$, 
gives rise to a $\widehat{k_w}$-point of this scheme
with the maps $d\rho_1$ and  $\rho_2$ injective;
for large enough $p_w$ this lifts, again by  ~\cite{Greenberg}, to a $\widehat{\mathfrak{o}_w}$-point $(\widetilde{\rho}_1, \widetilde{\rho}_2)$, and still {with}
$d\widetilde{\rho_1}$ and $\widetilde{\rho_2}$ injective.

But, then, $\widetilde{\rho}_1(\mathsf{H}_{\widehat{F_w}})$ cannot be maximal, 
e.g.\ by examining its derivative. As in the previous argument 
we deduce that $\jmap_w(\mathbf{H})$ is not maximal, 
and this contradiction finishes the proof.   
\end{proof}

\subsection{A Lie algebra complement}\label{Lie-alg-comp}
We regard $\mathfrak{g}$ as a sub-Lie-algebra of $\mathfrak{sl}_N$. 
Let $B$ be the Killing form of $\mathfrak{sl}_N$ whose restriction to $\gfrak$ we will still denote by $B$. 
The properties of the good place $w$ give us, in particular, 
that the restriction of $B$ on the Lie algebra $\mathfrak{h}_w=\Lie(H_w)$\index{hw@$\hfrak_w,$ Lie algebra of the acting group $H_w$ at good place $w$} has the following property.

\begin{lem}
Assuming $p_w$ is larger than an absolute constant depending only on the dimension the following holds.
If we choose an $\mathfrak{o}_w$-basis $\{e_1,\ldots,e_{\dim\H}\}$ 
for $\mathfrak{h}_w \cap \mathfrak{sl}_N(\mathfrak{o}_w)$, 
then $\det(B(e_i, e_j))_{ij}$ is a unit in $\mathfrak{o}_w^{\times}.$
That is:
\begin{multline}\label{Crit}
\mbox{$B$ restricted to $\mathfrak{h}_w \cap \mathfrak{sl}_N(\mathfrak{o}_w)$ is a}\\
\mbox{{non-degenerate} bilinear form over $\mathfrak{o}_w$.} 
\end{multline}
\end{lem}
\begin{proof}
Let the notation be as in the previous section, in particular, 
abusing the notation we denote the derivative of $\jmap_w$ with $\jmap_w$
as well.  
Let $\Hfrak_w$ denote the smooth $\order_w$-group scheme whose
generic fiber is $\H(F_w)$ and $\Hfrak_w(\order_w)=K^*_w$ given by Bruhat-Tits theory. 
The Lie algebra $\Lie(\jmap_w(\Hfrak_w))$ of $\jmap_w(\Hfrak_w)$ 
is an $\order_w$-algebra.
The Lie algebra $\hfrak_w$ is isomorphic to $\Lie(\jmap_w(\Hfrak_w))\otimes_{\order_w}F_w.$ 
Fix an $\order_w$-basis $\{e_i\}$ for $\Lie(\jmap_w(\Hfrak_w)),$ 
this gives a basis for $\hfrak_w.$ 
Now since $\H$ splits over $\widehat{F_w}$ and $K^*_w$ is a hyperspecial subgroup of $H^*_w$, we get: $\Hfrak_w(\widehat{\order_w})$
is a hyperspecial subgroup of $\H(\widehat{F_w}),$ see~\cite[2.6.1 and 3.4.1]{Ti}.
Fix a Chevalley $\widehat{\order_w}$-basis $\{\widehat{e_i}\}$ for 
$\Lie(\jmap_w(\Hfrak_w))\otimes_{\order_w}\widehat{\order_w}$ 
which is a Chevalley basis for $\hfrak_w\otimes_{F_w}\widehat{F_w},$ see~\cite[\S3.4.2 and 3.4.3]{Ti}.

Recall that $\underline{\jmap_w(\Hfrak_w)}$ denotes the reduction 
mod $\varpi_w$
of $\jmap_w(\Hfrak_w).$ This is a semisimple $k_w$-subgroup, 
therefore, in view of our assumption on characteristic\footnote{The requirement here is that 
${\rm char}(k_w)$ is big enough so that the following holds. 
The restriction of $B$ to each simple factor of $\Lie(\jmap_w(H_w))$ is a multiple of the 
Killing form on that factor, and this multiple is bounded in terms on $N.$ 
We take ${\rm char}(k_w)$ to be bigger than all the primes appearing in these factors.}
of $k_w$ we get 
\[
 \mbox{$\det B(\underline{\widehat{e_i}},\underline{\widehat{e_j}})\neq0,$}
\]
hence, $\det B(\widehat{e_i},\widehat{e_j})\in \widehat{\order_w}^\times.$ This implies that 
$\det B(e_i,e_j)\in \order_w^\times,$ as 
$\{e_i\}$ is another $\widehat{\order_w}$-basis for $\Lie(\jmap_w(\Hfrak_w))\otimes_{\order_w}\widehat{\order_w}.$    
\end{proof}

It follows from~\eqref{Crit}, 
and our assumption on ${\rm char}(k_w)$ that 
there exists an $\order_w$-module $\rfrak^{\mathfrak{sl}_N}_w[0]$ which is the orthogonal complement
of $\hfrak_w[0]$ in $\mathfrak{sl}_N(\order_w)$ with respect to $B,$ 
see e.g.~\cite{Ba}.\index{rw@$\rfrak_w,\rfrak_w[m]$, invariant complement
at good place~$w$}
Let $\rfrak_w[0]=\rfrak^{\mathfrak{sl}_N}_w[0]\cap\gfrak$
and write~$\rfrak_w$ for its~$F_w$-span. 
Then we have 
\begin{equation} \label{gwm splitting}
\mathfrak{g}_w[m] = (\mathfrak{h}_w \cap \mathfrak{g}_w [m] )\oplus ( \mathfrak{r}_w \cap \mathfrak{g}_w[m]),
\;\;\text{ for all }\;\; m\geq0.
\end{equation}
(see discussion after \eqref{Kvm definition} for notation). 


\subsection{The implicit function theorem at the good place} \label{properties}
Recall that $p_w={\rm char}(k_w).$
The first congruence subgroup of $\SL_N(\order_w)$ is a pro-$p_w$ group, see e.g.~\cite[Lemma 3.8]{PR}.
Moreover, a direct calculation shows that 
if the $w$-adic valuation of $p_w$ is at most $p_w-2$,  then the first congruence subgroup
of $\SL_N(\order_w)$ is torsion free.  {The condition on the valuation comes from estimating the radius of convergence
of the exponential map on $N \times N$ matrices with entries in $F_w$; we just use the estimate that the $p$-adic valuation of $n!$
is bounded by $n/(p-1)$.}
 This condition is 
satisfied in particular if $w$ is unramified over $p_w$ or if $p_w\geq [F:\Q]+2.$  

In the sequel we assume $w$ is so that $p_w\geq \max\{N^3,[F:\Q]+2\}.$
In view of the above discussion, for such $w$ we have 
$\exp:\g_w[m]\rightarrow K_w[m]$ is a diffeomorphism for any $m\geq 1,$
see e.g.~\cite[Ch.~9]{DSMS} for a discussion. 

Let us also put $H'_w=g_{\data,w}^{-1}\iota(\H)(F_w)g_{\data,w}=\jmap_w(\H)(F_w).$

\begin{lem}\label{l;pro-p-nbhd}
For any $m \geq 1$ we have
\begin{itemize}
 \item[(i)] $ H_w \cap K_w[m] =  \exp( \mathfrak{h}_w \cap \mathfrak{g}_w[m]).$
 \item[(ii)] Moreover,
every element of $K_w[m]$ can be expressed as $\exp(z)h$
where $h \in H_w \cap K_w[m]$, and $z \in  \mathfrak{r}_w \cap \mathfrak{g}_w[m]$.
\end{itemize}

\end{lem}

\begin{proof}
We shall use the following characterization of the Lie algebra of $H_w$:
$u$ belongs to $\h_w$ if and only if $\exp(tu) \in H_w$
for all sufficiently small $t,$ see e.g.\ \cite{Bour} or \cite[Lemma 1.6]{LapFin}.

For $z \in \RHS$, $\exp(t z)$ defines a $p$-adic analytic function of $t$
for $t \in \order_w$.  If $f$ is a polynomial function vanishing on 
$H'_w$, we see that $f(\exp(tz))$ vanishes for $t$ in a sufficiently 
small neighborhood of zero, and so also for $t \in \order_w$. 
Therefore, $\exp(tz)\in K_w[m]\cap H_w'.$
Recall that $K_w[m]$ is a pro-$p_w$ group, hence, 
$K_w[m]\cap H_w'$ is also a pro-$p_w$ group. 
This, in view of our assumption that $p> N,$
implies that $K_w[m]\cap H'_w\subset H_w,$
indeed $[H'_w:H_w]\leq F_w^\times/(F_w^\times)^N$ which
is bounded by $N^2,$ see e.g.~\cite[Ch.~8]{PR}.

Conversely, take $h \in \LHS,$ there 
is some $z \in \g_w[m]$ with $\exp(z) = h$.
Then $h^\ell = \exp(\ell z) \in \LHS,$ for $\ell=1,\dots$. The map
$t \mapsto \exp(tz)$ is $p_w$-adic analytic, so $\exp(\ell z) \in \LHS$ for all $t$ in a $p_w$-adic neighborhood of zero. It follows that $z$ in fact belongs to the Lie algebra of $H_w$.

The second assertion is a consequence of the first and the implicit function theorem, thanks to the fact that $\exp$ is a  diffeomorphism on $\g_w[m]$ { (see the discussion before the lemma).}
\end{proof}

\subsection{Adjustment lemma}\label{lem:adjustment}

As usual we induce a measure in $H_w$ using a measure on its Lie algebra. Then,  
\[
\mbox{$\exp:\RHS \rightarrow \LHS,\quad m\geq1$}
\] 
is a {\em measure-preserving} map. 
To see this, it is enough to compute the Jacobian of this map;
after identifying the tangent spaces at different points in $H_w$ via left  translation 
the derivative maybe thought of as a map $\mathfrak{h}_w \rightarrow \mathfrak{h}_w$.  
We again apply~\eqref{dderiv} for~$u\in\mathfrak{h}_w$. 
If now $u \in \RHS$, then $\ad u$ preserves the lattice $\mathfrak{h}_w \cap \mathfrak{g}_w[0]$ 
and induces an endomorphism of it that is congruent to $0$ 
modulo the uniformizer. It follows that~\eqref{dderiv} is congruent to the identity modulo $\varpi_w$, 
and in particular the Jacobian is a unit which implies the claim. 

The following is useful in acquiring two measure theoretically generic points to
be algebraically ``in transverse position" relative to each other, see the lemma regarding nearby generic points in \S\ref{GP}.

\begin{lem} [Adjustment lemma]
Let $m \geq 1$ be an integer, and $g \in K_w[m]$.
Given
subsets $A_1, A_2 \in K_w[1] \cap H_w$ of relative measure $>1/2$, there exists $\alpha_i \in A_i$ so that
$\alpha_1^{-1} g \alpha_2 = \exp(z)$ for some $z \in \r_w, \|z\| \leq q_w^{-m}$.
\end{lem}

\proof
Write, using the previous lemma, $g = \exp(z) h$ 
where $z \in \mathfrak{r}_w, \|z\| \leq q_w^{-m},  h \in H_w \cap K_w[m]$.
If $\alpha \in K_w[1] \cap H_w \subset \SL_N(\order_w)$ we have:
$$\alpha^{-1}  g = \exp(\Ad(\alpha^{-1}) z) (\alpha^{-1} h).$$

The map $f:\alpha \mapsto \alpha^{-1}h$ is measure-preserving. 
In view of our assumption on the relative measures of $A_1$ and $A_2,$ we may choose 
$\alpha \in A_1$ with $f(\alpha)^{-1}  \in A_2;$ the conclusion follows.
\qed

\subsection{The principal $\SL_2$}\label{l;good-sl2}
In the dynamical argument we will use spectral gap properties and 
dynamics of a unipotent flow. The following lemma will provide us with an {\it undistorted} copy of $\SL_2$. Here undistorted refers to the property that the ``standard" maximal compact subgroup of $\SL_2$
is mapped into $K_w$ which will be needed to relate our notion of Sobolev norm with the representation theory of $\SL_2.$  

As before we let $\Hfrak_w$ be a smooth $\order_w$-group scheme whose generic fiber is $\H(F_w)$ and so that $\Hfrak_w(\order_w)=K^*_w.$ 
 \begin{lem}
 There exists a {homomorphism} of $\order_w$-group schemes 
\[
\theta: \SL_2 \longrightarrow \Hfrak_w,
\] 
such that the projection of 
$\theta_w(\SL_2(F_w))$ into each $F_w$-almost simple factor of $H_w$ is nontrivial where $\theta_w=\jmap_w\circ\theta$.
\end{lem}

The following proof is due to Brian Conrad. We are grateful for his permission to include it here.

\proof
By our assumption $\Hfrak_w$ is semisimple.
Letting $R = \mathfrak{o}_w$ for ease of notation, pick a Borel $R$-subgroup $\mathcal{B}$ in $\mathfrak{H}_w$ and a maximal $R$-torus $\mathcal{T}$ in $\mathcal{B}$ (which exist by Hensel's Lemma and Lang's Theorem \cite[\S6.2]{PR}).  
Let $R \to R'$ be a finite unramified extension that splits $\mathcal{T}$, so $(\mathfrak{H}_w, \mathcal{T}, \mathcal{B})_{R'}$ is $R'$-split.

By the Existence and Isomorphism Theorems for reductive groups over rings  \cite[Thm. 6.1.16]{Conrad} this $R'$-split triple descends to a $\mathbb{Z}_{(p_w)}$-split triple $(\mathsf H,\mathsf T,\mathsf B)$. By~\cite[Thm.~7.1.9(3)]{Conrad},  $(\mathfrak{H}_w, \mathcal{T}, \mathcal{B})$ is obtained from $(\mathsf H,\mathsf T,\mathsf B)$ by twisting through an $R'/R$-descent datum valued in the finite group of pinned $R'$-automorphisms of $(\mathsf H,\mathsf T,\mathsf B)$ (all of which are defined over $\mathbb{Z}_{(p_w)}$). A specific $\mathbb{Z}_{(p_w)}$-homomorphism $\theta: {\rm{SL}}_2 \to \mathsf H$ is constructed in \cite[Prop. 2]{Serre} that carries the diagonal torus into $\mathsf T$ and carries the strictly upper triangular subgroup into $\mathsf B$.  Though it is assumed in  \cite{Serre} that the target group is of adjoint type, semisimplicity is all that is actually used in the construction.

We claim that for any local extension of discrete valuation rings $\mathbb{Z}_{(p_w)} \to A$ (such as $\mathbb{Z}_{(p)} \to R'$), the map $\theta_A$ is invariant under the finite group $\Gamma$ of pinned automorphisms of $\mathsf H$. It suffices to check this invariance over the fraction field $E$ of $A$, and then even on ${\rm{Lie}}(\mathsf H_E)$ since ${\rm{char}}(E)=0$. This in turn follows from the explicit description of ${\rm{Lie}}(\theta_{\mathbb{Q}})$ in Serre's paper because pinned automorphisms permute simple positive root lines respecting the chosen bases for each.

Thus, $\theta_{R'}$ is compatible with any $\Gamma$-valued $R'/R$-descent datum (such as the one obtained above), so $\theta_{R'}$ descends to an $R$-homomorphism ${\rm{SL}}_2 \to \mathfrak{H}_w$. This has the desired property relative to almost simple factors of the generic fiber over $F_w$ because (by design)   the composition of $\theta_{\mathbb{Q}}$ with projection to every simple factor of the split isogenous quotient $\mathsf H^{\rm{ad}}_{\mathbb{Q}}$ is non-trivial. \qed

\vspace{2mm}

We will refer to $\theta_w(\SL_2)$ as the {\em principal} $\SL_2$ in the sequel.
We define the one-parameter unipotent subgroup~$u:F_w\to\theta_w(\SL_2(F_w))$
by
\[
 u(t)=\theta_w\left(\begin{pmatrix} 1&t\\0&1\end{pmatrix}\right),
\]
and define the diagonalizable element 
\[
 a=\theta_w\left(\begin{pmatrix} p_w^{-1}&0\\0&p_w\end{pmatrix}\right)\in\theta_w(\operatorname{SL}_2(F_w)).
\]

\subsection{Divergence of unipotent flows.}  \label{div}

In the dynamical argument of the next section we will study the unipotent orbits
of two nearby typical points. As is well known the fundamental property of unipotent
flows is their polynomial divergence. We now make a few algebraic preparations
regarding this behavior at the good place~$w$.

Recall that $\r_w$ is invariant under the adjoint action of $\tSL.$
Let $\r_w^{\rm trv}$ denote the sum of all 
{\it trivial} $\theta_w(\SL_2)$ components of $\r_w$, and put $\r_w^{\rm nt}$
to be the sum of all {\it nontrivial} $\theta_w(\SL_2)$ components of $\r_w$.
In particular $\r_w=\r_w^{\rm trv}+\r_w^{\rm nt}$ as a $\tSL$ representation.

Even though we will not use this fact, 
let us remark that $\r_w^{\rm trv}$ does not contain
any $H_w$-invariant subspace. To see this let $V$ be such a subspace. 
{ Let $V[1]=V\cap\gfrak_w[1]$, in particular, $V[1]$ is a compact open (additive) subgroup of $V$ and 
the exponential map is defined on $V[1]$. 
Then the Zariski closure of the group generated by $\exp(V[1])$} is a {\it proper} subgroup of $\G$
which is normalized by $H_w$ and centralized by $\theta_w(\SL_2).$ 
In particular, it is normalized by $\iota(\H),$
the Zariski closure of $H_w.$
This however contradicts the maximality of $\iota(\H)$ in view of 
the fact that $\G$ is semisimple. 

Let\index{rw@$\r_w^{\rm hwt},\r_w^{\rm hwt}[m]$, sum of highest weight subspaces} $\r_w^{\rm hwt}$ be the sum of all of the highest weight spaces 
with respect to the diagonal torus of $\theta_{w}(\SL_2)$ 
in $\r_w^{\rm nt}.$ 
Note that $\r_w^{\rm hwt}$ is the space of $u(F_w)$-fixed vectors in $\r_w^{\rm nt}$. 
Let $\r_w^{\rm mov}$ be the sum of all of the remaining weight spaces in $\r_w^{\rm nt}$
where ${\rm mov}$ stands for moving.

Using this decomposition we write $\r_w=\r_w^{\rm hwt}+ \r_w^{\rm mov}+\r_w^{\rm trv} ;$
therefore, given $z_0\in\r_w$ we have $z_0=z_0^{\rm hwt}+z_0^{\rm mov}+z_0^{\rm trv}$.
In view of the construction of $\theta_w(\SL_2)$ and ${\rm char}(k_w)\gg1$
we also have 
\be\label{ee-good-again}
\mbox{$\r_w[m]=\r_w^{\rm hwt}[m]+ \r_w^{\rm mov}[m]+\r_w^{\rm trv}[m]\;\;$ for all $m\geq0.$}
\ee
Note that elements in $\mathfrak{r}_w^{\rm hwt}$ are 
{\em nilpotent}\footnote{This can be seen, e.g.\ because they can be 
contracted to zero by the action of the torus inside $\theta_{w}(\SL_2)$.}.

In the following we understand $\G$ as
a subvariety of the ${N^2}$-dimensional affine space via  $\rho:\G \to {\rm Mat}_N$.
We call a polynomial $p: F_w \rightarrow \mathbf{G}(F_w)$ {\em admissible}
if it has the following properties, 
\begin{enumerate}
\item The image of $p$ is centralized by $u(F_w)$ and contracted by~$a^{-1}$, i.e.\ for every $t$ we have $a^{-N} p(t) a^N \rightarrow e$
as $N \rightarrow \infty$; in particular the image of $p$ consists of unipotent elements.
\item $\deg(p) \leq {N^3}$. 
\item $p(0) = e$, the identity element;
\item All coefficients of $p$ belong to $\order_w$.  
\item { $p(F_w)\subset\exp(\r)$. }
\item {There exists some $t_0\in \order_w$
such that~$p(t_0)$ is not small.
More precisely, we have 
$p(t_0) =  \exp(\varpi_w^rz)$, where { $0  <r \leq \degbound$}, and $z$
is a nilpotent element of $\mathfrak{r}_w [0]\setminus \mathfrak{r}_w[1]$.}
\end{enumerate}
{ Note that $p(F_w)\subset \exp(\r_w^{\rm hwt})$ for all admissible polynomials.}

{The following construction and its dynamical significance is one of the driving
tools in unipotent dynamics,  we refer the reader e.g.\ to~\cite{Rn-Acta} and \cite[Lemma 4.7]{E-SL2} in the real case. 
}

\begin{lem}[Admissible polynomials] \label{divergence}
Let $z_0 \in \mathfrak{r}_w{[1]}$ with $z_0^{\rm mov}\neq 0.$ 
There exists $T  \in F_w$ with $|T|\gg \|z_0^{\rm mov}\|^{-\star}{q_w^{-1}}$, and an admissible  polynomial function
$p$ so that:
$$
\exp(\mathrm{Ad}({u(t)}) z_0) = p(t/T) g_t,
$$
where {$g_t\in \mathbf{G}(F_w)$ satisfies} $d(g_t, 1) \leq \|z_0\|^\star {q_w}$ whenever $|t| \leq |T|$.
\end{lem}

\proof
By the above we may write $z_0 = z_0^{\rm hwt} + z_0^{\rm mov}+z_0^{\rm trv}$ with 
$z_0^{\bullet} \in \r_w^{\bullet}{[1]}.$
By~\eqref{ee-good-again} we have $\|z_0^{\bullet}\|\leq \|z_0\|$ for $\bullet={\rm hwt},{\rm mov},{\rm trv}.$

Let us now decompose $\mathrm{Ad}(u(t)) z_0 = p^{\rm hwt}(t) + p^{\rm mov}(t)+z_0^{\rm trv}$ 
according to the above splitting of $\mathfrak{r}_w$.
{ Since} $z_0\notin\r_w^{\rm hwt}+\r_w^{\rm trv},$ the polynomial $p^{\rm hwt}$ is nonconstant, 
has degree $\leq N^2,$ and $p^{\rm hwt}(0)=z_0^{\rm hwt}$. Let $p_0(t)=p^{\rm hwt}(t)-z_0^{\rm hwt},$
and choose $T \in F_w$ of maximal {norm} so that the polynomial $p_0(T{s})$
has coefficients { of norm less than one}.  Then $p({s}) := \exp(p_0(T {s}))$ defines a polynomial { of degree at most $N^3$. 
In fact $p_0(T{s})$ is nilpotent for every~${s}$ and $\exp(\cdot)$ evaluated on nilpotent elements is a polynomial of degree at most $N$ with values in ${\rm Mat}_N$. Moreover, it}
still has integral coefficients so long as ${\rm char}(k_w) > N$.
This polynomial satisfies conditions (1)--{(5)} of admissibility by definition.

 Note that each coefficient of $p_0(t)$ is bounded { from} above by a constant multiple of $\|z_0^{\rm mov}\|$
 and so the $r$th coefficient of $p_0(tT)$ is bounded { from} above by a constant multiple of $\|z_0^{\rm mov} \| {|T|}^r$. 
 { By our choice of $T$ we have that for some $r\in\{1,\ldots,N^2\}$ we have $\|z_0^{\rm mov} \| |T|^rq_w^r\gg 1$.} 
Therefore { we obtain} that $|T| \gg \|z_0^{\rm mov}\|^{-\star}{ q_w^{-1}}$.

Suppose that condition {(6)} fails, i.e.~we have~{$p_0(T{s})\in\gfrak_w[N^2+1]$}
for all~${s}\in\order_w$.
As ${\rm char}(k_w)\gg_N1$ we then may choose~$N^2$ points in~$\order_w$ with distance~$1$.
Using Lagrange interpolation for the polynomial~$p_0(T{s})$ and these points
we see that the coefficients of~$p_0(T{s})$ 
all belong to~{$\order_w[N^2+1]$}. However, this contradicts { our choice of~$T$
and proves {(6)}}.

{Finally we define the function $t\mapsto g_t$ by the formula 
\[
p(t/T) g_t=\exp(\mathrm{Ad}(u(t))z_0)=\exp\bigl(p^{\rm hwt}(t) + p^{\rm mov}(t)+z_0^{\rm trv}\bigr)
\]
for all $t\in F_w$ with $|t|\leq |T|$. We note that the polynomial $p^{\rm mov}(t)$ 
corresponds to the weight spaces that are not of highest weight. We will now use the description 
of the $\operatorname{SL}_2$-representation $\mathfrak{r}_w$ in terms of a basis consisting of weight vectors obtained
from a list of highest weight vectors in $\mathfrak{r}^{\rm hwt}_w$. In fact,
our choice of the place $w$ implies that this basis can be chosen integrally and also over $\mathfrak{o}_w$. 
Using this basis we see that  
the coefficients of $p^{\rm mov}(t)$ appear also as coefficients in $p^{\rm hwt}(t)$
up to some constant factors of norm one -- recall that $p_w\gg N$. Moreover, terms in $p^{\rm mov}(t)$ 
always have smaller degree than the corresponding terms with the same coefficient (up to a norm one factor) in $p^{\rm hwt}(t)$.
Together with our choice of $T$
this implies that 
\[
  \|p^{\rm mov}(t)\|\leq |T|^{-1} \ll\|z_0^{\rm mov}\|^{\star}q_w
\]
for all $t\in F_w$ with $|t|\leq |T|$.  { In fact, this holds initially for each of the monomials 
in the various weight spaces appearing in $p^{\rm mov}$, but then by the ultrametric triangle inequality and integrality
of the weight decomposition also for their combination $p^{\rm mov}$}. 
Since $p(t/T)=\exp(p^{\rm hwt}(t)-z_0^{\rm hwt})$ and $\|z_0^{\rm hwt}\|,\|z_0^{\rm trv}\|\leq\|z_0\|$,  
the estimate concerning $d(g_t,e)$ for all $t\in F_w$ with $|t|\leq |T|$
now follows since the map $\exp:\gfrak_w[1]\to K_w[1]$ 
is $1$-Lipschitz.}
\qed

\subsection{Efficient generation of the Lie algebra}\label{eff-gen}

In the dynamical argument of the next section the admissible polynomial constructed above will give us elements
of the ambient group that our measure will be almost invariant under. We now study how
effectively this new element together with the maximal group~$H_w$ generate
some open neighborhood of the identity in~$\G(F_w)$.

\begin{lem}\label{l;conj-gen} 
There exist constants $\ell$ and $L\geq 1,$ depending on $N,$ such that, for any $z\in\r_w^{\rm hwt}[0]\setminus\r_w^{\rm hwt}[1]$, the following holds:
Every $g\in K_w[L]$ can be written as 
\[
\mbox{$g=g_1 g_2 \dots g_\ell,\;\;$ 
where $\;\;g_i \in K_w\cap (H_w\cup \exp(z)H_w\exp(-z))$.}
\]
\end{lem}

Note that $z$ as in the above lemma is a nilpotent element (because it belongs to the highest weight space),
and its exponential $\exp(z)$ belongs to  $K_w\setminus H_w.$ It turns out that the latter statement
continues to hold even reduced modulo $w$, and this is what is crucial for the proof.


It was mentioned in \S\ref{l;pro-p-nbhd}
that our choice of $w$ implies that 
for all $m\geq 1$ the group $K_w[m]$ is a torsion free pro-$p_w$ group;
we also recall that $K_w[m]\subset \G(F_w)^+.$ 

\proof
Let $\Gfrak_w$ (resp. $\mathfrak{H}_w$) be smooth $\order_w$-group schemes 
with generic fiber $G_w$ (resp. $\H(F_w)$)
so that $K_w=\Gfrak_w(\order_w)$ (resp. $K_w^*=\Hfrak_w(\order_w)$).
Recall the notation: for any $\order_w$-group scheme $\mathfrak M$ 
we let $\underline{\mathfrak M}$ denote the reduction mod $\varpi_w.$ 
As was shown in the lemma of \S\ref{ss:compatibility}
$\jmap_w(\Hfrak_w(\widehat{\order_w}))$ and $\Gfrak_w(\widehat{\order_w})$ are hyperspecial subgroups of 
$\jmap_w(\H)(\widehat{F_w})$ and $\G(\widehat{F_w})$ respectively. Furthermore,
they are $\widehat{\order_w}$ Chevalley group schemes with generic fibers
$\jmap_w(\H)(\widehat{F_w})$ and $\G(\widehat{F_w}),$ see e.g.\ the discussion following~\eqref{Crit}.

Since the group $\underline{\Hfrak_w}$ is quasi-split over $k_w$
we may choose one-dimensional unipotent subgroups $\underline{\mathfrak{U}}_{i}$ for $1\leq i\leq\dim\H$,
with the property that the product map
$\prod_{i=1}^{\dim \H} \underline{\mathfrak{U}}_{i} \rightarrow \underline{\Hfrak_w}$
is dominant.  In fact, these can be taken to be the reduction mod $\varpi_w$ of smooth closed $\order_w$-subgroup schemes 
$\mathfrak{U}_i$ of $\Hfrak_w,$ see~\cite[\S3.5]{Ti}.

To see why, fix a collection  $\underline{\mathfrak{U}}_{\alpha}$ of one-dimensional unipotent groups 
which generate $\underline{\Hfrak_w}$ as an algebraic group.   We prove inductively on $r$
that we may choose     $\alpha_1, \dots, \alpha_r$ such that
the Zariski closure $Z_r$ of $\prod_{i=1}^{r} \underline{\mathfrak{U}}_{\alpha_i}$ is $r$-dimensional.
Suppose this has been done for a given $r$. Then for any $\beta$ the closure of $Z_r \cdot \underline{\mathfrak{U}}_{\beta}$ 
is an irreducible algebraic set; if it is $r$-dimensional, it must therefore coincide with $Z_r$. 
If $r < \dim \H$ this cannot be true for all choices of $\beta$, by the generation hypothesis, and we deduce
that we can increase $r$ by taking $\alpha_{r+1} = \beta$.

By Lemma~\ref{c;red-max},
$\underline{\jmap_w(\Hfrak_w)}\;$ is a maximal connected algebraic subgroup of $\underline{\Gfrak_w}$.
Now let $g$ be the reduction modulo $\varpi_w$ of $\exp(z)$.  
We claim that  \begin{equation} \label{gnot} g \notin \underline{\jmap_w(\Hfrak_w)},\end{equation}
from where it follows that 
$g \underline{\jmap_w(\Hfrak_w)} g^{-1}$
together with $\underline{\jmap_w(\Hfrak_w)}$ generates $\underline{\Gfrak_w}$.

Indeed, if \eqref{gnot} failed, the elements $g^t$ for $t \in \mathbb{Z}$
belong to $ \underline{\jmap_w(\Hfrak_w)}$. Let $\bar{z}$ be the reduction of $z$ to the Lie algebra of $\underline{\Gfrak_w}$. 
Now
$ t \in \mathbb{A}^1 \mapsto \exp(t\bar{z}) \in \SL_N$
defines a one-parameter subgroup of $\SL_N$ over $k_w$. Consider the associated homomorphism
\begin{equation} 
\label{A1Lie}  \mathbb{A}^1 \rightarrow   \mathrm{End}(\wedge^{\dim \H}  \mathrm{Lie}(\SL_N)).
\end{equation}
The  degree  of this map is bounded only in terms of $N, \dim(\G)$.   
The value of  \eqref{A1Lie} at each $t \in \Z$ preserves the line in $\wedge^{\dim \H} \mathrm{Lie}(\SL_N)$
associated to the Lie algebra of $ \underline{\jmap_w(\Hfrak_w)}$. In a suitable basis,
this assertion amounts to the vanishing of various matrix coefficients of \eqref{A1Lie}. 
But if a matrix coefficient
of the map \eqref{A1Lie}  vanishes for all $t \in \mathbb{Z}$, it vanishes identically --  possibly after increasing the implicit bound
 for $\mathrm{char}(k_w)$ in 
\S \ref{ss;good-place} if necessary. Therefore the one-parameter subgroup $t \mapsto \exp(t \bar{z})$ of $\SL_N$ normalizes
the Lie algebra of $\underline{\jmap_w(\Hfrak_w)}$. Therefore
$$ [\bar{z}, \Lie  \ \underline{ \jmap_w(\Hfrak_w)}] \subset \Lie  \ \underline{ \jmap_w(\Hfrak_w)}.$$ 
But this contradicts the assumption on $z$ --  e.g. we can find an element
 $\mathrm{H}$ in the Lie algebra of  $\underline{\jmap_w({\Hfrak_w})}$, arising from the $\SL_2$  in \S\ref{l;good-sl2}, 
 such that $[\bar{z},{\rm H}] $ is a nonzero multiple of $\bar{z}$, and $\bar{z}$ is not in  $ \Lie \  \underline{ \jmap_w(\Hfrak_w)}$ by \eqref{gwm splitting}.

For simplicity in the notation, put $\mathfrak{U}_i'=\exp(z)\jmap_w(\mathfrak{U}_i)\exp(-z).$
Arguing just as above, we see that we may choose  $\mathfrak{X_i}$ 
 (for $i=1,\ldots,\dim\G$), each equal to either $\jmap_w(\mathfrak{U}_i)$ or $\mathfrak{U}_i'$ for some $i,$ with  $\mathfrak{X}_i=\jmap_w(\mathfrak{U}_i)$ for $i=1,\ldots,\dim\H$
and such that if we define $\varphi$ via
\[
\mbox{$\varphi : \mathfrak{X}:=\prod_{i=1}^{\dim\G}{\mathfrak{X}_i}\xrightarrow{(\iota_i)}
\prod_{i=1}^{\dim\G}{\Gfrak_w}\xrightarrow{{\rm mult}}{\Gfrak_w}\;,$}
\] 
then the map $\underline{\varphi}$ is dominant.  

The above definition implies that 
$\underline{\varphi}$ is a polynomial map on the $\dim\G$ dimensional 
affine space with ${\rm deg}( \underline{\varphi})\leq N^4.$ 

Therefore, in view of our assumption on the characteristic of $k_w,$ one gets that $\underline{\varphi}$
is a separable map. We recall the argument: 
note that $\underline{\varphi}$ is a map from the $\dim\G$ dimensional affine space 
${\underline{\mathfrak{X}}}$ into the affine variety ${\underline{\mathfrak{G}_w}}$. 
Let $E=\widehat{k_w}({\underline{\mathfrak{X}}})$
and let $E'$ be the quotient field of $\widehat{k_w}[\underline{\varphi}^*(\underline{\Gfrak_w})]$
in $E.$ 
In view of the above construction, $E$ is an algebraic extension of $E'.$
By Bezout's Theorem 
the degree of this finite extension is bounded by a constant depending on 
$\deg(\underline{\varphi}),$ see e.g.~\cite{Ul-Wil} or~\cite[App.~B]{E-Gh-pos}.   
The claim follows in view of ${\rm char}(k_w)\gg_N 1.$   

In particular, $\Phi=\det({\rm D}(\underline{\varphi}))$ is a nonzero polynomial.
This implies that we can find a finite extension
$k_w'$ of $k_w$ and a point $\underline{a'}\in\underline{\mathfrak{X}}(k'_w)$ so that 
$\Phi(\underline{a'})=\det({\rm D}_{\underline{a'}}(\underline{\varphi}))\neq 0.$ 
Note that in fact under our assumption on ${\rm char}(k_w)$ and since 
$\deg(\underline{\varphi})\leq N^4$ there is some point $\underline{a}\in\underline{\mathfrak{X}}(k_w)$ so that $\Phi(\underline{a})\neq 0.$ 
This can be seen by an inductive argument on the number of variables in the polynomial 
$\Phi.$ For polynomials in one variable the bound one needs is ${\rm char}(k_w)>\deg (\Phi);$
now we write $\Phi(a_1,\ldots,a_{\dim\G})=\sum_j \Phi^j(a_2,\ldots,a_{\dim\G})a_1^j$
and get the claim from the inductive hypothesis.

All together we get: there is a point $a\in\mathfrak{X}(\order_w)$ so that
$\det({\rm D}_a(\varphi))$ is a unit in $\order_w^\times$.   
The implicit function theorem thus implies that there is some ${b}\in K_w$ so that
$\varphi\bigl(\mathfrak{X}(\order_w)\bigr)$ 
contains ${b}K_w[L],$ where $L$ is an absolute constant.  
Therefore 
\[
\Bigl(\varphi\bigl(\mathfrak{X}(\order_w)\bigr)\Bigr)^{-1}
\Bigl(\varphi\bigl(\mathfrak{X}(\order_w)\bigr)\Bigr)
\] 
contains $K_w[L]$ as we wanted to show.
\qed

\medskip

The following is an immediate corollary of the above discussion; 
this statement will be used in the sequel. 

\begin{propo}[Efficient generation]\label{effgen}  
Let $p: F_w \rightarrow \mathbf{G}(F_w)$ be an admissible polynomial map as defined in \S\ref{div}. 
Then there exist constants $\levelbound\geq1$ and $\ell,$ depending on $N,$ so that: 
each $g \in K_w[\levelbound]$ may be written as
a product $g=g_1 g_2 \dots g_\ell$, where 
\[
 g_i\in \bigl\{h\in H_w:\|h\|\leq q_w^{\levelbound} \bigr\} \cup 
 \bigl\{p(t)^{p_w^{s}}: t\in\order_w\mbox{ and }{0\leq s\leq 2N}\bigr\}^{\pm1}.
\]
\end{propo}

\proof
Let~$t_0\in\order_w$ be as in property {(6)} of admissibility so that~$p(t_0)=\exp(\varpi_w^rz)$, 
where {$0 < r \leq \degbound$}, and {$z\in\r_w^{\rm hwt}[0]\setminus\r_w^{\rm hwt}[1]$.

Let~$a\in\theta_w(\operatorname{SL}_2(F_w))$
be the element corresponding to the diagonal element with\footnote{ We apologize for the notational
clash between $p_w$, which is the residue characteristic  of $F_w$, and the polynomials $p, p_0$.} eigenvalues~$p_w^{-1},p_w$.  
{Since}~$r>0$, we {will} use conjugation by~$a\in H_w$
to produce again an element which we may use in the previous lemma.
Indeed, let~$j\geq 1$ be minimal for which $ a^j p(t_0)a^{-j}=\exp( \varpi_w^r \Ad_a^j z)\notin K_w[1]$
and note that~$j\leq N^2$.
If~$z'= \varpi_w^r \Ad_a^j z\in\rfrak_w[0]$ (and~$z'\notin\rfrak_w[1]$ by choice of~$j$), we set {$z''=z'$ and will use this element below}. 
However, if~$z'\notin\rfrak_w[0]$
then~$\|z'\|={q}_w^i$ for some $i\in\mathbb N$ with $i\leq {2N}$. In this case we find that the
 element~$z''=p_w^iz'\in\rfrak_w[0]\setminus\rfrak_w[1]$ satisfies
\[
 a^j p(t_0)^{p_w^i} a^{-j}=a^j\exp(p_w^{i}\varpi_w^rz)a^{-j}=\exp(p_w^i z')=\exp(z''),
\]
and can be used in the previous lemma. For this
also note that we may assume that~$F$ is unramified at~$w$,  since there are only finitely many ramified places for $F$ and the implicit
constants in the definition of ``good place'' are permitted to depend on $F$; 
this gives that~$p_w$ is a uniformizer for~$F$ at~$w$. 

Increasing~$\ell$ { to accommodate the change in the formulation
of the statements, the proposition follows from the previous lemma}.
\qed

\section{The dynamical argument}\label{proof}

Throughout this section we let~$w\in\Sigma_f$ denote a good place
for the MASH~$Y=Y_\data$ with~$\data=(\H,\iota,g_\data)$. 
Moreover, we let~$\theta_w(\SL_2)$ be the principal~$\SL_2$ as in~\S\ref{l;good-sl2}
satisfying that~$\theta_w(\SL_2(F_w))$ is contained in the acting subgroup~$H_w$
at the place~$w$ and~$\theta_w(\SL_2(\order_w))<K_w$.

\subsection{Noncompactness}\label{s;nondivergence}
As usual, when $X$ is not compact some extra care is required to control the 
behavior near the ``cusp''; using the well studied non-divergence properties of unipotent flows we need to show that
``most" of the interesting dynamics takes place in a ``compact part" of $X.$
We will also introduce in this subsection the height function~$\height:X\to\R_{>0}$,
which is used in our definition of the Sobolev norms.

For the discussion in this subsection we make the following reduction: put
$\G'={\rm Res}_{F/\Q}(\G)$ and $\H'={\rm Res}_{F/\Q}(\H),$ then $\G'$ and $\H'$ are semisimple $\Q$-groups 
and we have the $\Q$-homomorphism ${\rm Res}_{F/\Q}(\iota):\H'\to\G'.$
Moreover
\[
 \mbox{$\Lbf'(\Q)\backslash \Lbf'(\adele_\Q)=\Lbf(F)\backslash\Lbf(\adele_F),$ 
for $\Lbf=\G,\H$}
\]
and we also get a natural isomorphism between 
$\Lbf'(\Z_p)$ and $\prod_{v\mid p}\Lbf(\order_v),$ see~\cite{PR} for a discussion of these facts.
Similarly we write $\H'_j={\rm Res}_{F/\Q}(\H_j)$ for any $F$-simple factor~$\H_j$ of $\H.$

As is well known (e.g.\ see~\cite{Bo-adele}) 
there exists a finite set $\Xi\subset\G'(\adele_\Q)$ 
so that
\[
\G'(\adele_\Q)=\bigsqcup_{\xi\in\Xi}\G'( \Q)\xi G_\infty' K_f',
\]
where~$G_\infty'=\G'(\R)$ and~$K_f'$ is the compact open subgroup of~$\G'(\adele_{\Q,f})$ corresponding to~$K_f<\G(\adele_{F,f})$. 
We define~$G_q'$ and~$K_q'$ similarly for every rational prime~$q$.
We let~$S_0= S_0(\G)$ be the union of~$\{\infty\}$
and a finite set of primes so that~$\Xi\subset \prod_{v\in S_0} G_v' K_f'$. 

We now recall the standard terminology for~$S$-arithmetic quotients {($S$ is a finite set)}. 
Set $K'(\plw)=\prod_{q\notin\plw}K_q'$ and put $\Z_{\plw}=\Z[\tfrac1{q}:q\in \plw],$
$\Q_\plw=\R\times\prod_{q\in \plw}\Q_q$ and
$G'_\plw=\G'(\Q_\plw)$.
We let~$\mathfrak g'$ be the Lie algebra of~$\G'$. {We choose an integral lattice $\mathfrak g'_{\Z}$
in the $\Q$-vector space $\mathfrak g'$, 
with the property that  $[\mathfrak{g}'_{\Z}, \mathfrak{g}'_{\Z}] \subset \mathfrak{g}'_{\Z}$.}
Let~$\mathfrak g'_{\Z_\plw}=\Z_\plw\mathfrak g'_\Z$ 
be the corresponding~$\Z_\plw$-module. We also define $\|u\|_\plw=\prod_{v\in \plw}\|u\|_v$
for elements~$u$ of the Lie algebra $\mathfrak g'_\plw=\Q_S\otimes_\Q\mathfrak g'$
over~$\Q_S$.

Our choice of~$S_0$ now implies $\G'(\adele_\Q)=\G'(\Q) G'_\plw K'(\plw)$ whenever~$S\supseteq S_0$, which also gives
\[
 \G'(\Q)\backslash\G'(\adele_\Q)/K'(\plw)\cong X_S=\Gamma_S\backslash G_S'=
\Gamma_{S_0}\backslash \Bigl(G_{S_0}'\times \prod_{q\in S\setminus S_0}K_q'\Bigr)
\]
where~$\Gamma_S=\G'(\Q)\cap K'(\plw)$. In that sense we have a projection map~$\pi_S(x)=x K'(S)$ 
from~$X$ to~$X_S$.

Similar to~\cite{EMV}, for every $x\in X$ we put
\begin{multline*}
\height (x):=\height(\pi_S(x))=\sup\Bigl\{\mbox{$\|\Ad(g^{-1})u\|_\plw^{-1}$}: u\in\gfrak'_{\Z_\plw}\setminus\{0\}\mbox{ and }\\
g\in G'_\plw \mbox{ with }\pi_S(x)=\Gamma_\plw g\Bigr\}.
\end{multline*}
We note that in the definition of~$\height(\pi_S(x))$ we may also fix
the choice~$g$ of the representative for a given~$\pi_S(x)$, the 
supremum over all~$u\in\gfrak'_{\Z_\plw}$
will be independent of the choice. If~$S\supsetneq S_0$, we may choose~$g$ such that~$g_q\in K_q'$
for~$q\in S\setminus S_0$. This in turn implies that the definition of~$\height(x)$ is also
independent of~$S\supseteq S_0$.
Define
\[
\Siegel( R):=\{x\in X :\height (x)\leq R\} .
\]

Note that
\be\label{HtCont}
{\height(xg)\ll\|g\|^2\height(x)}\mbox{ for any }g\in G_v, v\in\Sigma. 
\ee
If~$v\in\Sigma_f$ the implicit constant is~$1$ 
and moreover
\be\label{Htconst}
{\height(xg)=\height(x)}\mbox{ for any }g\in K_v.
\ee
Finally, we need:
\begin{lemma}
There exists constants $\consta\label{inj-radius}>1$ and $\constc>0$ such that
for all $x\in X$ the map\index{kaa8@$\ref{inj-radius}$, exponent of height in injectivity radius} 
\be\label{e;injective-Xc}
\mbox{$g\mapsto xg$  is 
injective on $\bigl\{g=(g_\infty,g_f):d(g_\infty,1)\leq \theconstc\height(x)^{-\ref{inj-radius}}, g_f\in K_f'\bigr\}.$}   
\ee
\end{lemma}
\proof Suppose that $x g_1= x g_2$ for $g_1, g_2$ belonging to the set above. 
In what follows take $S=S_0$.   Let $g_{1,S}$ and $g_{2,S}$ be the $S$ component of $g_1$ and $g_2$. 

 Fix $g \in G_S'$ such that $\pi_S(x) = \Gamma_S g$.  Then $\Gamma_S  g g_{1,S} = \Gamma_S g g_{2,S}$, and so  $g_{1,S} g_{2,S}^{-1}$ fixes $g^{-1} \mathfrak{g}'_{\Z_S}$.
In particular, $g_{1,\infty} g_{2,\infty}^{-1}$ fixes 
$$ L_x := g^{-1} \mathfrak{g}'_{\Z_S} \cap \mathfrak{g}'_{\Z},$$
{ the intersection being taken inside of $\mathfrak{g}'$}; 
this can  also be described as those elements $u \in g^{-1} \mathfrak{g}'_{\Z_S}$ that satisfy $\|u\|_v \leq 1$ for all nonarchimedean $v \in S$. 

We consider $L_x$ as a $\Z$-lattice inside  the real vector space $\mathfrak{g}' \otimes \R$. For every $\lambda \in L_x$ we have
$\|\lambda\| \geq \mathrm{ht}(x)^{-1}$. The  covolume of $L_x$ inside
$\mathfrak{g}' \otimes \R$ is the same as the covolume of $g^{-1} \mathfrak{g}'_{\Z_{S}}$
inside $\mathfrak{g}' \otimes \Q_{S}$, and this latter covolume is independent of $x$. 
By lattice reduction theory, then, $L_x$ admits a basis 
  $\lambda_1, \dots, \lambda_d$
such that $\| \lambda_i\| \ll \mathrm{ht}(x)^{(d-1)}$.  

Thus, if we choose the constant $\kappa_{10}$ sufficiently large and $c_1$ suitably small, we have
$$ \| (g_{1,\infty} g_{2,\infty}^{-1}) \lambda_i  - \lambda_i\|  < \mathrm{ht}(x)^{-1} \mbox{ for all $1 \leq i \leq d$,}$$
and thus the fact that $g_{1,\infty} g_{2,\infty}^{-1}$ fixes the lattice $L_x$ setwise
implies that it in fact fixes $L_x$ pointwise. This forces $g_{1,\infty} g_{2,\infty}^{-1}$ to belong to the center
of $G_{\infty}'$, and this will be impossible if we choose $c_1$ small enough. 
\qed

Let $w\in\places_{f}$ be the good place as above, which gives that $\H_j(F_w)$ is not compact for all $j,$ 
and let $p_w$ be the prime so that $w\mid p_w.$ Then $\H_j'(\Q_{p_w})$ is not compact for all~$j$.

We have the following analogue\footnote{We note that due to the dependence on~$p$ we do not obtain
at this stage a fixed compact subset that contains $90\%$
of the measure for all MASH, see Corollary~\ref{non-divergence-adelic}.} of~\cite[Lemma 3.2]{EMV}.

\begin{lemma}[Non-divergence estimate]\label{l;non-div}
There are positive constants $\consta\label{non-div-power-p}$ and $\consta\label{non-div-cons},$ 
depending on $[F:\Q]$ and $\dim\G,$ so that for any {\rm MASH}
set $Y$ we have\index{kaa8@$\ref{non-div-power-p}$, exponent of~$p_w$ in non-divergence estimate}
\[
\mu_\data\left(X\setminus\Siegel( R)\right)\ll p_w^{\ref{non-div-power-p}}R^{-\ref{non-div-cons}},
\]
where $p_w$ is a rational prime with~$w\mid p_w$ for a good place~$w\in\Sigma_F$ for~$Y$.\index{kaa9@$\ref{non-div-cons}$,
exponent of height in non-divergence estimate} 
\end{lemma}

\proof
The proof is similar to the proof of~\cite[Lemma 3.2]{EMV},
using the $\plw$-arithmetic version of the quantitative non-divergence of unipotent flows which is proved in~\cite{KT}, for
which we set $\plw=S_0\cup\{p_w\}$. We recall parts of the proof. 

Recall that $\H'_j(\Q_{p_w})$ is not compact 
for any $F$-almost simple factor  $\H_j$ of $\H$
and that $\H(F_w)$ is naturally identified 
with the group of $\Q_{p_w}$-points of ${\rm Res}_{F_w/\Q_{p_w}}(\H),$ see e.g.~\cite{PR}. 
Let $H_{w}=g_{\data,w}^{-1}\iota(\H(F_w))g_{\data,w}$ be the component
of the acting group at the place~$w$, where~$g_\data\in\G(\adele)$
is the group element from the data~$\data=(\H,\iota,g_\data)$
determining the MASH set~$Y=Y_\data$. 

Let us note that discrete $\Z_\plw$-submodules of $\Q_\plw^k$ are free,~\cite[Prop.\ 8.1]{KT}.
Furthermore, by~\cite[Lemma 8.2]{KT} if $\Delta=\oplus_{i=1}^\ell\Z_\plw{\bf v}_i$ 
is a discrete $\Z_\plw$-module, then the covolume of $\Delta$ in
$V=\oplus_{i=1}^\ell\Q_\plw{\bf v}_i$  is defined by
${\rm cov}(\Delta)=\prod_{v\in \plw}\|{\bf v}_1\wedge\cdots\wedge{\bf v}_\ell\|_v$
and we will refer to~$\Delta$ as an~$S$-arithmetic lattice in~$V$.
 
Let~$h\in G_S'$.
A subspace $V\subset \gfrak'_\plw$ is called $\Gamma_S h$-rational if
$V\cap \Ad^{-1}_h\gfrak'_{\Z_\plw}$ is an $\plw$-arithmetic lattice in $V;$
the covolume of $V$ with respect to~$\Gamma_S h$ is defined to be $\operatorname{cov}(V\cap \Ad^{-1}_h\gfrak'_{\Z_\plw})$
(and is independent of the representative).
One argues as in the proof of~\cite[Lemma 3.2]{EMV}  (given in Appendix B of \cite{EMV})  and gets: there exist positive constants 
$\constc$ and $\consta\label{non-div-proof-1}$ such\index{kaa91@$\ref{non-div-proof-1},\ref{non-div-proof-2}$, exponents
of~$p_w$ in the non-divergence proof} 
that
\begin{multline} 
\label{e;no-thin-subspace}\mbox{there is no $x$-rational, $H_w$-invariant }\\
 \mbox{proper subspace of covolume $\leq \theconstc p_w^{-\ref{non-div-proof-1}}$}
\end{multline}
where we fix some~$x\in\pi_S(Y)$. We define~$\rho=\theconstc p_w^{-\ref{non-div-proof-1}}$.

Let now $U=\{u(t)\}$ be a one parameter $\Q_{p_w}$-unipotent subgroup of 
$H_w$ which projects
nontrivially into all $\Q_{p_w}$-simple factors of ${\rm Res}_{F_w/\Q_{p_w}}(\H)$. 
Then, since the number of 
$x$-rational proper subspaces
of covolume $\leq \rho=\theconstc p_w^{-\ref{non-div-proof-1}}$ is finite
and by the choice of $U$ above, 
a.e.~$h\in H_w$ has the property that $hUh^{-1}$ does not leave invariant any 
proper~$x$-rational subspace
of covolume $\leq \rho.$ 
Alternatively, we may also conclude for a.e.~$h\in H_w\cap K_w$
that~$U$ does not leave invariant any proper~$xh$-rational subspace
of covolume~$\leq\rho.$

Since $\H$ is simply connected, it 
follows from the strong approximation theorem and the
Mautner phenomenon that $\mu_\data$ is ergodic for the action  of $\{u(t)\}.$
This also implies that the~$U$-orbit of~$xh$ equidistributes with
respect to~$\mu_\data$ for a.e.~$h$. We choose~$h\in H_w\cap K_w$
so that both of the above properties hold true for~$x'=xh$.

{Let $x'=\Gamma_Sh'$. Hence, for any $\Gamma_S h'$-rational subspace $V$,} 
if we let 
\[
\psi_V( t)={\rm cov}(\Ad_{u(t)}(V\cap \Ad^{-1}_{h'}\gfrak'_{\Z_\plw})),
\]
then either $\psi_V$ is unbounded or equals a constant $\geq \rho.$  
Thus, by~\cite[Thm.\ 7.3]{KT} there exists a positive constant $\consta\label{non-div-proof-2}$ 
so that\index{kaa91@$\ref{non-div-proof-1},\ref{non-div-proof-2}$, exponents
of~$p_w$ in the non-divergence proof}  
\be\label{e;non-div}
\mbox{$|\{t:|t|_w\leq r,x' u(t)\notin\Siegel(\epsilon^{-1})\}|
\ll p_w^{\ref{non-div-proof-2}}(\tfrac{\epsilon}{\rho})^\alpha |\{t:|t|_{w}\leq r\}|,$}
\ee
for all large enough $r$ and {$\epsilon>0$}, 
where~$\alpha=\ref{non-div-cons}$ only depends on the degree of the polynomials
appearing in the matrix entries for the elements of the one-parameter unipotent subgroup~$U$
(see \cite[Lemma~3.4]{KT}).
The lemma now follows as the~$U$-orbit equidistributes with respect to~$\mu_\data$.

We note that the proof of~\eqref{e;no-thin-subspace} also uses non-divergence estimates and induction on the dimension, 
which is the reason why the right hand side contains a power of $p$.
\qed

\subsection{Spectral input}\label{ss:spectral-input}
As in \S\ref{l;good-sl2} we let~$\theta_w(\SL_2)<g_{\data,w}^{-1}\iota(\H)g_{\data,w}$ be the principal~$\SL_2$ and also recall the one-parameter
unipotent subgroup
 \[
 u(t) :=  \theta_{w}\left(\left(\begin{array}{cc} 1 & t \\0 & 1 \end{array}\right)\right).
 \]

In the following we will assume that the representations of $\SL_2(F_w)$, via~$\theta_w$, both on 
\[
L_0^2(\mu_\data):=\bigl\{f\in L^2(X,\mu_\data):\textstyle\int f\operatorname{d}\!\mu_\data=0\bigr\},
\]
  and on $L_0^2(X,\operatorname{vol}_G)$  are $1/\temp$-tempered (i.e.\ the matrix coefficients of the $\temp$-fold tensor product
are in~$L^{2+\epsilon}(\SL_2(F_w))$ for all~$\epsilon>0$). 
(Recall again here that $\H$ is simply connected.)
 As was discussed in \S\ref{sec:tau} this follows directly  in the case when $\H(F_w)$ has property (T), see \cite[Thm.~1.1--1.2]{O},
and in the general case we apply property $(\tau)$ in the strong form, see~\cite{LC}, \cite{GMO} and \cite[\S6]{EMV}.\footnote{Note 
that the $F_w$-rank of the almost simple factors of~$\H$ are never zero, since $\H$ is $F_w$ quasi split.}

\subsection{Adelic Sobolev norms}\label{sec;sob-intro}
Let $C^{\infty}(X)$ denote the space of functions 
which are invariant by a compact open
subgroup of $\G(\adele_f)$ and are smooth at all infinite places. There exist a system of norms $\Sob_d$ on $C_c^\infty( X)$ 
with the following properties, {see Appendix A, in particular, see~\eqref{Sobnormdef} and~\eqref{eq:Sob-inc}.}

\begin{enumerate}
\item [S0.] ({\it Norm on~$C_c(X)$}) Each $\Sob_d$ is a pre-Hilbert norm 
on $C_c^\infty(X)=C^\infty(X)\cap C_c(X)$ (and so in particular finite there).

\item[S1.] ({\it Sobolev embedding}) There exists some~$d_0$ depending
on~$\dim\G$ and~$[F:\Q]$ such that for all $d\geq d_0$ we have
$\|f\|_{L^{\infty}} \ll_d \Sob_d(f)$.

\item[S2.] ({\it Trace estimates}) Given $d_0,$ 
there are 
$d>d'>d_0$
and an orthonormal basis $\{e_k\}$ of the completion of $C_c^\infty(X)$
with respect to $\Sob_{d}$ which is orthogonal with respect to $\Sob_{d'}$
so that 
\[
\mbox{$\sum_k\Sob_{d'}(e_k)^2<\infty\quad$ and $\quad\sum_k\tfrac{\Sob_{d_0}(e_k)^2}{\Sob_{d'}(e_k)^2}<\infty.$} 
\]

\item[S3.] ({\it Continuity of representation}) Let us write~$g\cdot f$
for the  action of~$g\in \G(\adele)$ on~$f\in C_c^\infty(X)$. For all $d\geq 0$ we have 
\[
\Sob_d(g\cdot f)\ll \|g\|^{4d}\Sob_d(f),
\]
for all $f\in C_c^\infty(X)$ and where
\[
 \|g\|=\prod_{v\in\Sigma}\|g_v\|.
\]
Moreover, we have~$\Sob_d(g\cdot f)=\Sob_d(f)$ 
if in addition~$g\in K_f$. 
For the unipotent subgroup~$u(\cdot)$ 
in the principal~$\operatorname{SL}_2$ at the good place~$w$
we note that $\|u(t)\|\leq(1+|t|_w)^N$ for all $t\in F_w.$

\item[S4.] ({\it Lipshitz constant} at $w$) {There exists some~$d_0$ depending
on~$\dim\G$ and~$[F:\Q]$ such that for all $d\geq d_0$ the following holds.} For any $r\geq0$ and any $g\in K_w[r]$ we have
\[
\|g\cdot f-f\|_\infty\leq q_w^{-r}\Sob_d(f)
\]
for all $f\in C_c^\infty(X)$.

\item[S5.]  ({\it Convolution on ambient space}) 
Recall from \S\ref{sec:equi0no}  and \S\ref{proofnotn} that $\Gpp$
is the projection onto the space of $\G(\adele)^+$ invariant functions
and that $L^2_0(X,\vol)$ is the kernel of $\Gpp.$
Let $\Av$ be the operation of averaging
over $K_w[\levelbound],$ where $\levelbound$ is given by Proposition~\ref{effgen}.
For $t\in F_w$ we define $\mathbb{T}_t=\Av \star\, \delta_{u(t)} \star \Av$ by convolution. 
For all $x\in X$, all $f\in C_c^\infty(X)$, and~$d\geq d_0$ we have
\[
|\mathbb{T}_t(f-\Gpp f)(x)|\ll q_w^{(d+2)L} \height(x)^{d} \|\mathbb{T}_t\|_{2,0}\Sob_d( f),
\]
where $\|\mathbb{T}_t\|_{2,0}$ denotes the operator norm of $\mathbb{T}_t$
on $L^2_0(X,\vol_G)$. Once more~$d_0$
depends on~$\dim\G$ and~$[F:\Q]$.

\item[S6.] ({\it Decay of matrix coefficients})
For all $d\geq d_0$ we have 
\begin{multline}\label{sobnormmc}\Big| \langle u(t) f_1, f_2 \rangle_{L^2(\mu_\data)} - 
	\int f_1\operatorname{d}\!\mu_\data \int\bar f_2\operatorname{d}\!\mu_\data \Big|\\ 
 \ll (1+|t|_w)^{-1/2\temp} \Sob_{d}(f_1) \Sob_d(f_2),
\end{multline}
where~$d_0$ depends on~$\dim\G$ and~$[F:\Q]$; {recall that $\H$ is simply connected}.
\end{enumerate}

%

\subsection{Discrepancy along $v$-adic unipotent flows}\label{s;discrepancy}
We let~$M$ be as in \S\ref{ss:spectral-input}
and choose the depending parameter $\flowscale= 100\temp$.

We say a point $x \in X$ is {\em $T_0$-generic w.r.t.\ the Sobolev norm $\Sob$} 
if for any ball of the form $J= \{t \in F_w: |t-t_0|_w \leq |t_0|_w^{1-1/\flowscale}\}$,
with its center satisfying ${\rm n}(J)=|t_0|_w \geq T_0,$ 
we have
\begin{equation} \label{disc}
D_J(f)(x) =\left| \frac{1}{|J|}
\int_{t \in J} f(x u(t)) \operatorname{d}\!t - \int f
\operatorname{d}\!\mu_\data\right| \leq {\rm n}(J)^{-1/\flowscale} \Sob (f), \end{equation}
for all $f\in C_c^\infty( X)$.
Here $|J|$ denotes the Haar measure of $J$ and we note that the definition
of ${\rm n}(J)=|t_0|_w$ is independent of the choice of the center~$t_0\in J$.

\begin{lem}[$T_0$-generic points] \label{lem:genericbound}
For a suitable $d_0$ depending only on~$\dim\G$ and~$[F:\Q]$ and all~$d\geq d_0$
 the measure of points that are not $T_0$-generic
w.r.t.\ $\Sob_{d}$ is decaying polynomially with~$T_0$. More precisely,
\[
\mu_\data\bigl(\bigl\{y\in Y: y\mbox{ is not~$T_0$-generic}\bigr\}\bigr) \ll T_0^{-1/4\temp}
\]
for all $T_0>q_w^\star$.
\end{lem}

\proof
We let~$\Sob=\Sob_{d_0}$ and will make requirements on~$d_0\geq 1$ during the proof. 
We first consider a fixed $f$ in $L^2(X)$ which is in the closure of $C_c^{\infty}(X)$ with respect to $\Sob.$ 
Since $\H$ is simply connected, by~\eqref{sobnormmc} we have
\be\label{e;matrix-coef-generc}
\left| \langle u(t) f, f \rangle_{L^2(\mu_\data)} - \Bigl| \int f \operatorname{d}\!\mu_\data \Bigr|^2 \right| \ll (1+|t|_w)^{-1/2\temp}  \Sob (f)^2,
\ee
where we assume~$d_0$ is sufficiently big for~S6.\ to hold.

For a fixed $J$ let $D_{J}(f)(x)$ be defined in~\eqref{disc}. Then we have
\[
\int_{X} \bigl|D_{J}(f)(x)\bigr|^2\operatorname{d}\!\mu_\data=\frac{1}{|J|^2}\int_{J\times J} 
\langle u(t) f, u(s)f \rangle \operatorname{d}\!s\operatorname{d}\!t- 
\left( \int f \operatorname{d}\!\mu_\data \right)^2.
\]
Split $J\times J$ into $|t-s|_w\leq {\rm n}(J)^{\frac{1}{2} (1-1/\flowscale)}$ and 
$|t-s|_w> {\rm n}(J)^{\frac{1}{2} (1-1/\flowscale)};$
in view of~\eqref{e;matrix-coef-generc} we thus get 
\[
\int_{X} \bigl|D_{J}(f)(x)\bigr|^2\operatorname{d}\!\mu_\data\ll {\rm n}(J)^{-\tfrac{1}{4\temp}(1-\tfrac{1}{\flowscale})} \Sob(f)^2.
\]

Still working with a fixed function~$f$, this implies in particular that 
\[
\mu_\data\Bigl(\Bigl\{x\in X: D_{J}f(x) \geq {\rm n}(J)^{-1/\flowscale} \lambda\Bigr\}\Bigr)
\ll\lambda^{-2} {\rm n}(J)^{\tfrac{8\temp+1}{4\temp \flowscale} - \tfrac{1}{4M}} \Sob(f)^2
\]
 for any $\lambda>0$. We note that by our choice of~$\flowscale$ the second term~$-\frac1{4M}$
in the exponent is more significant than the first fraction. 

Given $n\in\mathbb{N},$
the number of disjoint balls $J$ as above
with ${\rm n}(J)=q_w^n$ is bounded above by $q_w{\rm n}(J)^{1/\flowscale}$.
Consequently, summing over all possible values of $n\in\mathbb{N}$ with 
$q_w^n \geq T_0$ and all possible subsets $J$ as above, we see
that
\begin{multline}\label{e;gen-one-func} 
	\mu_\data\Bigl(\Bigl\{ x \in X: D_Jf(x) \geq  {\rm n}(J)^{-1/\flowscale} \lambda\;\Sob(f) \, \mbox{ and }
{\rm n}(J) \geq T_0\Bigr\}\Bigr) \\ \ll \lambda^{-2} q_w\sum_{q_w^n \geq T_0}
q_w^{n\left(\tfrac{12\temp+1}{4\temp \flowscale} - \tfrac{1}{4\temp}\right)} \ll \lambda^{-2} T_0^{-1/8\temp},
\end{multline}
where we used $T_0>q_w^\star.$

To conclude, we use property S2.\ of the Sobolev norms. Therefore, there are 
$d>d'>d_0$
and an orthonormal basis $\{e_k\}$ of the completion of $C_c^\infty(X)$
with respect to $\Sob_{d}$ which is orthogonal with respect to $\Sob_{d'}$
so that 
\be\label{eq:reltr-cont}
\mbox{$\sum_k\Sob_{d'}(e_k)^2<\infty\quad$ and $\quad\sum_k\tfrac{\Sob(e_k)^2}{\Sob_{d'}(e_k)^2}<\infty.$} 
\ee
Put $c=(\sum_k\Sob_{d'}(e_k)^2)^{-1/2}$
and let $B$ be the set of points so that for some $k$ and some~$J$ with~${\rm n}(J)\geq T_0$ 
we have\footnote{Thanks to S1.\ 
and assuming $d_0$ is big enough, all expressions considered here are continuous 
w.r.t.\ $\Sob_d$ for all $d\geq d_0$.}
\[
D_Je_k(x) \geq c\,{\rm n}(J)^{-1/\flowscale} \Sob_{d'}(e_k).
\]
In view of~\eqref{e;gen-one-func}, applied for $f=e_k$ with 
$\lambda_k=c\tfrac{\Sob_{d'}(e_k)}{\Sob(e_k)},$
and~\eqref{eq:reltr-cont} the measure of this set
is $\ll T_0^{-1/8\temp}$. 

Let $f \in C_c^\infty(X)$ and write $f=\sum f_k e_k$ and  suppose $x \not\in B.$ Let~$J$
be a ball with~${\rm n}(J)\geq T_0$. 
 Then using the triangle inequality for $D_J$ we obtain
\begin{align*} 
D_J(f)(x) &\leq \sum_k |f_k|D_J(e_k)(x)\leq c{\rm n}(J)^{-1/\flowscale} \sum_k |f_k| \Sob_{d'}(e_k)\\
&\leq c{\rm n}(J)^{-1/\flowscale}  \left(\sum_{k} |f_k|^2\right)^{1/2} \left(\sum_k \Sob_{d'}(e_k)^2 \right)^{1/2}
\\ &= {\rm n}(J)^{-1/\flowscale} \Sob_{d}(f).
\end{align*}
\qed

\subsection{Pigeonhole principle}\label{sec:pigeonhole}
We now use a version of the pigeonhole principle to show that if $\vol(Y)$
is large, then in some part of the space and on certain ``small but not too small" scales $Y$ is not aligned along 
$\stab(\mu)$.
This gives the first step to producing nearby generic points to which
we may apply the effective ergodic theorem, discussed above. 


With the notation as in \S\ref{s;nondivergence} put 
\[
\Xc=\Siegel\left(p_w^{(\ref{non-div-power-p}+20)/\ref{non-div-cons}}\right),
\]
then by Lemma~\ref{l;non-div} we have $\mu_\data(\Xc)\geq 1-2^{-20}.$

Let us also assume that the analogue of~\eqref{omegav2} holds for 
$\wedge^\ell\Ad$ for $1\leq\ell\leq\dim\G$ where as usual $\Ad$ denotes the adjoint representation.
Therefore, we have that the analogue of~\eqref{eq:chevalley-claim} holds for $\wedge^\ell\Ad.$ 
More precisely, for any infinite place $v,$ any~$u\in\Xi_v,$ and all~$z\in\bigwedge^\ell \gfrak_v$
we have
\be\label{eq:chevalley-claim-Ad}
 \wedge^{\ell}\Ad(\exp u)z=z\mbox{ implies that }\Ad(\exp (tu))z=z\mbox{ for all }t\in F_v
\ee
for all $1\leq\ell\leq\dim\G.$

We now fix\index{theta@$\Theta^*$, open neighbourhood of the identity
in pigeon hole argument} $\Theta^*=\prod_{v\in\Sigma_\infty}\Theta_v^*\times K_f\subset \G(\adele)$ 
with $\Theta_v^*\subset\exp(\Xi_v)$ { open} for all infinite places $v$ 
so that the map $g'\in\Theta^*\mapsto xg'\in X$ is injective for all~$x\in\Xc$. 
Note that in view of our choice of $\Xc$ and~\eqref{e;injective-Xc} we may and will 
choose $\Theta^*$ with $\vol_G(\Theta^*)\gg p_w^{-\ref{vol-theta}}$
for some\index{kaa93@$\ref{vol-theta}$, exponent of~$p_w$ in lower bound for~$\Theta^*$} $\consta>0\label{vol-theta}.$
We will also use the notation
\[
\Theta^*[w^m] = \{ g \in \Theta^*: g_w \in K_w[m]\},
\]
for all $m\geq0.$

Recall from the Stabilizer lemma (Lemma~\ref{s:stabilizer-lem}) that the stabilizer
of our MASH set is given by
${\rm Stab}(\mu_\data)=g_{\data}^{-1}\iota(\H(\adele))\mathbf N(F)g_\data$ where
$\mathbf N$ denotes the normalizer of $\iota(\H)$ in $\G.$ In the following we will use \S\ref{ss:Borel-Prasad}
and in particular the notation $S=\places_\infty\cup\{w\},$ $\widetilde{H_\data},$ and $N_S$ introduced there.

We claim that 
\be\label{eq:stab-theta}
\stab(\mu_{\data})\cap\Theta^*[w^1]\subset \widetilde{H_\data}=N_SH_\data.
\ee
To see this let $g'=\gamma h\in \stab(\mu_\data)\cap\Theta^*[w^1]$ with $\gamma\in g_\data^{-1}\N(F)g_\data$ and $h\in H_\data.$ 
At all $v\in\places_\infty$ apply~\eqref{eq:chevalley-claim-Ad} with~$\ell=\dim \H(F_v)$, with the vector~$z$
belonging to~$\wedge^{\ell}\Ad(g_{v}^{-1})\Lie(\iota(\H)(F_v))$, { and taking $u$ such that $\exp(u)=g'_v$. The quoted statement
shows that a one-parameter subgroup containing $g'_v$ normalizes $H_v$,} and since the connected component of the normalizer of the Lie group $H_v$
equals $H_v$ this implies $\gamma_v\in H_v.$ In particular we get $\gamma\in g_\data^{-1}\iota(\H)(F)g_\data.$
At the place $w$ we use the fact that 
\[
K_w[1]\cap\iota(\H)(F_w)\subset\iota(\H(F_w))
\] 
to establish the claim.

For a subset $\nbhd \subset \G(\adele)$  denote the ``doubled sets''
by $\nbhd_2 =\nbhd \cdot \nbhd^{-1}$ and $\nbhd_4 = \nbhd_2 \cdot \nbhd_2$.

\begin{lem}\label{l;pigeonhole}
 Suppose a measurable subset $E \subset Y$ satisfies $\mu_\data(E) > 3/4$.
Let $\nbhd \subset G$ be open with $\nbhd_4 \subset \Theta^*[w^1]$
and $\vol_G(\nbhd) > 2 \widetilde{\vol}(Y)^{-1}$.  
Then there exist $x,y \in E$ so that $x = y g_0$ with $g_0 \in \nbhd_4 \setminus \stab(\mu_\data)$.
\end{lem}

\proof
Let $\{x_i: 1 \leq i \leq I\}$ be a maximal set of points in $\Xc$ such that $x_i \nbhd$
are disjoint. By our choice of~$\Theta^*$ (as function of~$\Xc$ and so of~$q_w$)
we have $I \leq \vol_G(\nbhd)^{-1}$. By maximality of~$I$ we also have
that $\{x_i \nbhd_2: 1\leq i\leq I\}$ covers $\Xc$.
This observation implies in particular that there exists some $i_0$ so that 
\[
\mu_\data(x_{i_0} \nbhd_2\cap E)\geq \frac{1}{2 I}.
\]  
Fix some $y_1\in x_{i_0} \nbhd_2\cap E,$ 
then any $y_2\in x_{i_0} \nbhd_2\cap E$ is of the form $y_1 g$, where $g \in \nbhd_4$.

Suppose, contrary to our claim, that every $y_2\in x_{i_0} \nbhd_2\cap E$ 
were actually of the form $y_1 h$ with $h \in \stab(\mu_\data) \cap \nbhd_4$. Recall that $\stab(\mu_\data)\cap \nbhd_4\subset\widetilde{H_\data}.$
The orbit map $h \mapsto y_1 h$, upon restriction to $\nbhd_4$, is injective
by assumption (on $\Theta^*$) and $y_1\in\supp(\mu_\data),$ we thus get $\mu_\data(x_{i_0} \nbhd_2\cap E)\leq\widetilde{m_\data}(\nbhd_4\cap \widetilde{H_\data}).$ The definition of the volume of a homogeneous set together with the above discussion now gives
\[
\vol_G(\nbhd)\leq\frac{1}{I}\leq 2\mu_\data(x_{i_0} \nbhd_2\cap E)\leq2\widetilde{m_\data}(\nbhd_4\cap \widetilde{H_\data})\leq 2\widetilde{\vol}(Y)^{-1}
\]
which contradicts our assumption.
\qed





\subsection{Combining pigeon hole and adjustment lemmas}\label{sec:GP} \label{GP}
For any $v\in\places_\infty$ let $\Theta_v\subset\Theta_v^*$ be so that $(\Theta_{v})_4\subset\Theta_v^*$
and put $\Theta=\prod_{v\in\places_\infty}\Theta_v\times K_f.$
We may assume that  
$\vol_G(\Theta)\geq\constc\label{theta-thetas}\vol(Y)^{-1}$ where $\ref{theta-thetas}$
depends only on $\G(F_v)$ for $v\in\places_\infty.$ We define $\Theta[w^m]=\Theta\cap\Theta^*[w^m].$
We will use the notation $\nu^g(f):=\nu(g\cdot f)$ (with $f\in C_c(X)$) for the action of $g\in G$ on a probability 
measure $\nu$ on $X.$ 

Put $\Sob=\Sob_{d}$ for some $d> d_0$
so that the conclusion of the {\it generic points lemma} 
of \S\ref{lem:genericbound} holds true.

\begin{lem}[Nearby generic points] 
Let $r \geq 0$ be so that 
\[
2\vol_G(\Theta[w^r])^{-1} \leq \vol(Y)^{\ref{vol-tvol-cons}}.
\]
There exists  $x_1,x_2 \in\Xc\cap Y$ and $g\in G$ so that $x_2 = x_1g$ and
\begin{enumerate}
\item $x_1,x_2$ are both $T$-generic for $\mu_\data$ for some $T > q_w^\star$;
\item $g \in {\Theta^*}[w^r],$ 
\item we may write\footnote{As before $g_w$ denotes simply the $w$-component of $g \in \G(\adele)$.} $g_w = \exp(z)
$, 
where
$z \in \mathfrak{r}_w$ is not fixed by $\Ad(u(t))$ 
and in particular $z \neq 0.$ Moreover $\|z\| \leq q_w^{-r}$.
\end{enumerate}
\end{lem}

\proof 
Let us call $x\in\Xc$ a {\em $T$-good} point if the fraction of $h \in K_w[1]\cap H_w$ 
for which $xh$ is $T$-generic exceeds $3/4,$ with respect to the Haar measure on $H_w.$
Note that by the defintion $x h\in\Xc$ for all $h\in K_w$ and $x\in\Xc.$
We apply the {\it generic points lemma} in \S\ref{s;discrepancy} and obtain that 
for $T\geq q_w^\star$ the $\mu_\data$-measure of the set of $T$-generic points exceeds $0.99.$  
Using Fubini's theorem, and our choice of $\Xc$ we conclude that
the measure of the set $E=\{y$ is a $T$-good point$\}$ exceeds $3/4$.

By our assumption on $r$ and~\eqref{eq:vol-tvol-final} we have $\vol_G(\Theta[w^r])\geq 2\widetilde{\vol}(Y)^{-1}.$
Let $\nbhd=\Theta[w^r]$. Applying the lemma in \S\ref{sec:pigeonhole}  
there are $T$-good points $y_1, y_2 \in X$
such that 
\[
\mbox{$y_1 = y_2 g_0\;\;$ where $g_0 \in \Theta^*[w^r] \setminus \stab(\mu)$.}
\]
By the {\it adjustment lemma} in \S\ref{lem:adjustment} and definition of $T$-good points,
there exists $g_1, g_2 \in K_w[1] \cap H_w$ so that $x_i :=y_i g_i $ are $T$-generic, and
so that $g := g_1^{-1} g_0 g_2$
satisfies $g_w= \exp(z)$, where $z \in \mathfrak{r}_w$ {and $\|z\|\leq q_w^{-r}$}.

%

Now let us show that $z$ is not centralized by $u(t)$.
 Suppose to the contrary.
Because $x_2 = x_1 g$ and $x_1, x_2$ are $T$-generic; for any $f\in C_0^\infty(X)$ 
{and any $t_0\in F_w$ with $|t_0|>T$,} we have
\[
|\mu_\data(f)-\mu_\data^{g}(f)|\leq D_J(f)(x_2)+D_J(g\cdot f)(x_1)\leq |t_0|^{-1/m}(\Sob (f) +\Sob (g\cdot f) )
\]
which implies $\mu_\data$ is invariant under $g$. 
But we assumed $g_0 \notin \stab(\mu_\data)$ which also implies $g \notin \stab(\mu_\data)$. \qed

\subsection{Combining generic point and admissible polynomial lemmas}\label{H-principle}
We refer to \S\ref{div} for the definition of admissible polynomials.

\begin{lem}[Polynomial divergence]
There exists 
an admissible polynomial $p: F_w \rightarrow \mathbf{G}(F_w)$ so that:
\begin{equation} \label{au} 
	\left| \mu_\data^{p(t)}(f) - \mu_\data(f) \right|  \leq \vol(Y)^{-\star} \Sob(f), \quad \text{for all }\; t\in\order_w.
\end{equation}	
\end{lem}

\begin{proof}
	We maximize $r$ in the {\it nearby generic points lemma} of \S\ref{GP}.
	This gives~$p_w^{\ref{vol-theta}} q_w^{(r+1)\dim\G }  \gg \vol(Y)^{\ref{vol-tvol-cons}}$.
	Using~$q_w\ll(\log\vol(Y))^2$ we may simply write~$q_w^r\gg\vol(Y)^\star$.
	
Let $x_1,x_2$ be two $T_0$-generic points given by this lemma, 
in particular, there is $g \in \Theta^*[w^r]$ so that $x_2= x_1g$ where $g_w = \exp(z_0)$,
and $z_0 \in \mathfrak{r}_w $ is not fixed by $\Ad(u(t))$. In the notation of~\S\ref{div}
we have~$z_0^{\rm mov}\neq 0$. 
Then by the {\it admissible polynomials lemma} in \S\ref{div} there exists $T\in F_w$ with 
\[
|T| \gg \|z_0^{\rm mov}\|^{-\star}\geq \|z_0\|^{-\star}\gg \vol(Y)^\star,
\] 
and an admissible  polynomial
$p$ so that:
\be\label{eq:g-t}
 \exp(\mathrm{Ad}({u(t)}) z_0) = p(t/T) g_t
\ee
where $d(g_t, 1) \leq \|z_0\|^\star$ when $|t| \leq |T|$.

Suppose $t_0\in F_w$ with $|t_0|\leq |T|$ and as in \S\ref{s;discrepancy} put 
$J=\{t\in F_w: |t-t_0|_w\leq |t_0|_w^{1-1/m}\}.$ 
Fix some arbitrary  $f\in C_c^\infty(X).$
Then by the {\it generic point lemma} 
in \S\ref{s;discrepancy} and assuming $|t_0|_w\geq T_0$ we have
\be\label{eq:Birkhoff}
\left| \frac{1}{|J|}
\int_{t \in J} f(x_i u(-t)) \operatorname{d}\!t - \int f
\operatorname{d}\!\mu_\data\right| \leq |t_0|^{-1/\flowscale} \Sob (f),\quad i=1,2.
\ee

Let $\tilde p: F_w\to {\bf G}(\adele),$ be a polynomial given by 
$\tilde p(t/T)_w=p(t/T)$ with $p$ as above
and\footnote{The element $g$ above need not have ``small'' $v$ components for $v\neq w.$}  
$\tilde p(t)_v=g_v$ for all $v\neq w.$
Using property S4. and~\eqref{eq:g-t} this polynomial satisfies 
\begin{align}
\label{eq:poly}f(x_2u(-t))&=f(x_1u(-t)\tilde p(t/T))+O(\|z_0\|^\star \Sob(f))\\
\notag&=f(x_1u(-t)\tilde p(t_0/T))+O({|T|}^{-\star}\Sob(f))+O(\|z_0\|^\star \Sob(f))
\end{align}
for~$|t_0|\geq |T|^{1/2}$ and~$t\in J$ (defined by~$t_0$),
where we used S4.\ and the definition of $J$ in the last step.

All together we get
\be\label{eq:tp}
\left| \mu_\data^{\tilde p(t/T)}(f) - \mu_\data(f) \right|  \leq \vol(Y)^{-\star} \Sob(f), \quad \text{for }\; |T|^{1/2}\leq |t|_w\leq |T|.
\ee
Indeed this follows from~\eqref{eq:Birkhoff} and~\eqref{eq:poly}.

Now choose~$t_1\in\order_w$ with~$|t_1|\in [|T|^{-1/2},q_w|T|^{-1/2}],$
this implies that~\eqref{eq:tp} holds for~$t=t_1T$. Also note that with
this choice~$p(t_1)\in K_w[\kappa\log_{q_w}(\vol(Y))]$
for some constant~$\kappa>0$ (that only depends on the parameters
appearing in the definition of an admissible polynomial). The latter implies that~\eqref{eq:tp}
holds for~$p(t_1)$ instead of~$\tilde p(t/T)$ trivially as a consequence of S4. 
Since $p(t/T)=\tilde p(t/T)\tilde p(0)^{-1}=\tilde p(t/T)\tilde p(t_1)^{-1}p(t_1),$ we get~\eqref{au} from~\eqref{eq:tp} in view of property S3 --
{ note that, if $|t|\leq |T|^{-1/2}$,~\eqref{eq:tp}
holds for~$p(t/T)$ instead of~$\tilde p(t/T)$ trivially as a consequence of S4.}  
\end{proof}

\subsection{Proof of Theorem~\ref{adelic}}\label{ai}
{For simplicity in notation we write $\mu$ for $\mu_\data$.}
Let $p:F_w\to\G(F_w)$ be the admissible polynomial given by the polynomial divergence lemma in \S\ref{H-principle}.
Let $\Av$ be the operation of averaging
over $K_w[\levelbound],$ where $\levelbound$ is given by Proposition \ref{effgen}.  
Then, it follows from that proposition and property S3.\ that
\[
\left| \mu(f) - \mu(\Av*f) \right| \ll q_w^{\star}(\vol Y)^{-\star} \Sob(f)\ll \vol(Y)^{-\star}\Sob(f),
\]
for all $f \in C^{\infty}_c(X)$. Note that in Proposition~\ref{effgen} any $g\in K_w[L]$ is written 
as a bounded product of two types of elements. The first type of elements belong to $\{h\in H_w: \|h\|\leq q_w^L\},$
preserve $\mu,$ and distort the Sobolev norm by a power of $q_w.$
The second type of elements are powers of the values of the admissible polynomial
at $\order_w,$ preserve the Sobolev norm, and almost preserve the measure.

Let $t \in F_w$. Denote by $\delta_{u(t)}$
the delta-mass at $u(t)$, and let $\star$ denote convolution of measures.
Using the fact that $\mu$ is $u(t)$-invariant the above gives
\begin{equation}\label{AI}
|\mu(f)-\mu(\Av \star \delta_{u(t)} \star \Av* f)|\leq
\vol(Y)^{-\star} (\Sob(\delta_{u(t)} \star \Av* f) +
\Sob(f)).
\ee
Recall that we are assuming $\H$ is simply connected, 
thus $\iota(\H(\adele))\subset\G(\adele)^+$ and in particular 
$\Gpp f$ is $H_w$-invariant. Also since $K_w[L]\subset \G(\adele)^+$, 
the support of $\Av \star \delta_{u(t)} \star \Av$ is contained in $\G(\adele)^+.$ 
These observations together with~\eqref{AI} imply 
\begin{multline*}
	|\mu(f-\Gpp f)| \ll \\
	|\mu(\Av \star \delta_{u(t)} \star \Av* (f-\Gpp f))| +  \vol(Y)^{-\star} (\Sob(\delta_{u(t)} \star \Av* f) +
\Sob(f)). 
\end{multline*}

But $\Av$ reduces Sobolev norms, and by property S3.\ the application of $u(t)$ multiplies them by at most $(1+|t|_w)^{4Nd}$. Therefore
$$|\mu(f-\Gpp f)| \ll \int_X|\mathbb{T}_t (f-\Gpp f) (x)|\operatorname{d}\!\mu(x) + \vol(Y)^{-\star} (1+|t|_w)^{4dN } \Sob(f),$$
where we write $\mathbb{T}_t$ for the ``Hecke operator" $\Av \star \delta_{u(t)} \star \Av$.

By property S5.\ we have for any $x\in X$
\[
|\mathbb{T}_t(f-\Gpp f) (x)| \ll q_w^{(d+2)\levelbound}\height(x)^{d}\|\mathbb{T}_t\|_{2,0}\Sob( f);
\] 
moreover, by ~\eqref{eq:T-L2-bound} we have 
$\|\mathbb{T}_t\|_{2,0}\ll |t|_w^{-1/2\temp} q_w^{2 d \levelbound}$. 

Let $R>0$ be a (large) parameter; writing $\int_X|\mathbb{T}_t f (x)|\operatorname{d}\!\mu(x)$ as
integrals over $\Siegel( R)$ and $X\setminus\Siegel( R),$
in view of the lemma in \S\ref{s;nondivergence}, we get
\begin{align*}
&|\mu(f-\Gpp f)| \ll\\  
&\left( |t|_w^{-1/2\temp} q_w^{(3d+2)\levelbound}R^{d} + p_w^{\ref{non-div-power-p}}R^{-\ref{non-div-cons}}+ 
\vol(Y)^{-\star}(1+ |t|_w)^{4dN } \right) \Sob(f).
\end{align*}
Optimizing $|t|_w$ and $R$, using the fact that $q_w\ll(\log\vol(Y))^2$, 
we get the theorem. We note that the power of $\vol(Y)$ depends only on 
the parameter $\temp$ from S6., $\dim_F\G$ and $[F:\Q].$ 
\qed

\subsection{Beyond the simply connected case}\label{s;not-sc}
The proof of the main theorem above assumed that $\H$
is simply connected.
In this section, using the discussion in the simply connected case, 
we will relax this assumption.
It is worth mentioning that for most applications the theorem
in the simply connected case already suffices.

Let $\tH$ denote the simply connected covering of $\H,$ and 
let $\pi:\tH\to\H$ denote the covering map.  
We define~$H'={\H(F)\pi(\tH(\adele))}$, which is closed since it corresponds
to a finite volume orbit of~$\pi(\tH(\adele))$ in~$\H(F)\backslash \H(\adele)$. 
By the properties of the simply connected cover $H'$ is a normal
subgroup of $\H(\adele)$ and $\H(\adele)/H'$ is abelian, see e.g.~\cite[p. 451]{PR}.
As $\H(F)\backslash\H(\adele)$ has finite volume the same applies to $\H(\adele)/H'$
which implies this quotient is compact. 
Let $\nu$ be the probability Haar measure on this compact abelian group. 

Suppose the data $\data$ is fixed as in the introduction, dropping the assumption that $\H$ is simply connected
and let $Y=Y_{\data}$ be as before. We also define the MASH set and measure
\[
\bigl(\widetilde{Y},\tilde{\mu}\bigr)=
\left(\iota\bigl(\H(F)\backslash\H(F)\pi(\tH(\adele))\bigr)g,\tilde{\mu}\right),
\]
which is defined by the simply connected group~$\tH$, the homomorphism~$\iota\circ\pi$,
and the same element~$g\in\G(\adele)$ as for~$Y$. 

Then, we have $\mu=\int_{\H(\adele)/H'}{\tilde\mu}^h d\nu(h)$ 
where ${\tilde\mu}^h$ is the probability Haar measure on
$\iota\left(\H(F)\backslash\H(F)\pi(\tH(\adele))h\right)g.$
 
Moreover, in view of our 
definition of volume and the fact that $\H(\adele)/H'$ is abelian
we have
$\vol(\widetilde{Y})=\vol(\widetilde{Y}g^{-1}\iota(h)g)$
(as those orbits have the same stabilizer group)
Applying Theorem~\ref{adelic} we get 
\[
|{\tilde{\mu}}^h(f)-{\tilde{\mu}}^h(\pi^+(f))|\leq 
\vol(\widetilde{Y})^{-\kappa_0}\Sob(f)\quad
\text{for all $h\in\H(\adele)/H'.$}
\]
All together we thus have
$|\int_X (f-\pi^+(f)) d\mu|\leq \vol(\widetilde{Y})^{-\kappa_0}\Sob(f).$

It seems likely that the argument in \S\ref{ss:Borel-Prasad} could be used to show that
$\vol(\widetilde{Y})\asymp\vol(Y)^\star.$ We will not pursue  this here.

\subsection{Proof of Corollary~\ref{non-divergence-adelic}}\label{sec:proof-adel-nondiv}

We will first consider MASH measures for which the algebraic group~$\H$ is simply connected.

Let~$\epsilon>0$ be arbitrary. Choose some compact~$Z\subset X$
 with~$\mu_{x\G(\adele)^+}(Z)>1-\frac\epsilon2$
for every~$x\in X$. Now choose some~$f_\epsilon\in C_c^\infty (X)$
with~$1_{Z}\leq f_\epsilon\leq 1$. Applying Theorem~\ref{adelic} to~$f_\epsilon$ and any
MASH measure~$\mu_\data$ with~$\data=(\H,\iota,g)$ and~$\H$ simply connected we 
find some~$\constc\label{constcfornondivcor}$ with
\[
  \int f_\epsilon\operatorname{d}\!\mu_\data>
\int_X f_\epsilon\operatorname{d}\!\operatorname{vol}_G
-\ref{constcfornondivcor}\Sob(f_\epsilon)\vol(Y)^{-\finalexponent}.
\]
In particular, there exists some~$\constc\label{const2fornondivcor}=\ref{const2fornondivcor}(\epsilon)$
such that if~$\vol(Y)>\ref{const2fornondivcor}$ 
then
\[
\mu_\data(\supp(f_\epsilon))\geq \int f_\epsilon\operatorname{d}\!\mu_\data>1-\epsilon.
\]

In the case where~$\vol(Y)\leq \ref{const2fornondivcor}$
we first find a good place~$w$ as in \S~\ref{ss;good-place}
with~$q_w\ll_\epsilon 1$ and then apply Lemma~\ref{l;non-div} to find another compact set~$Z'$
with~$\mu_\data(Z')>1-\epsilon$.  The set~$X_\epsilon=\supp(f_\epsilon)\cup Z'$
now satisfies the corollary for all MASH measures with~$\H$ simply connected.

If~$\mu_\data$ is a MASH measure and~$\H$ is not simply connected, then we can 
repeat the argument from the previous subsection to obtain~$\mu_\data(X_\epsilon)>1-\epsilon$ also.

\appendix
\section{Adelic Sobolev norms}\label{sec:appendixA}
We begin by defining, for each finite place $v$, a certain system of projections $\operatorname{pr}_v[m]$ of any unitary $\G(F_v)$-representation; 
these have the property that $\sum_{m\geq 0}\operatorname{pr}_v[m]=1$.
The definitions in the archimedean place likely 
can be handled in a similar fashion using spectral theory applied 
to a certain unbounded self adjoint differential operator (e.g.\ by splitting the spectrum into intervals $[e^m,e^{m+1})$). However, we will work instead more directly with differential operators in the definition of the norm.
 
\subsection{Finite places}\label{s-finite-place}

Let $v$ be a finite place.
Let $\operatorname{Av}_v[m]$ be the averaging projection on $K_v[m]$-invariant vectors, 
put $\operatorname{pr}_v[0]=\operatorname{Av}_v[0]$ and $\operatorname{pr}_v[m] = \operatorname{Av}_v[m] - \operatorname{Av}_v[m-1]$ for $m\geq1$.

We note that, if $\mu$ is any spherical (=$K_v[0]$-bi-invariant)
probability measure on $\G(F_v)$, then convolution with $\mu$
commutes with $\operatorname{pr}_v[m]$ for all $m$.  Indeed, the composition (in either direction)
of $\mu$ with $\operatorname{pr}_v[m]$ is {\em zero} for $m \geq 1$, and equals $\mu$ for $m=0$.

\subsection{Adelization}
We denote by $\underline{m}$ any function on the set of finite places of $F$
to non-negative integers,
which is zero for almost all $v$.
Write $\|\underline{m}\| = \prod_{v} q_v^{m_v}$. Note that
\begin{equation} \label{poly}
\|\underline{m}\| \geq 1\quad\text{and}\quad
 \# \{\underline{m}: \|\underline{m}\| \leq N \}= O_\epsilon(N^{1+\epsilon}),
 \end{equation}
which follows since $\ell^\epsilon$ bounds the number 
of ways in which $\ell$ can be written as a product of $[F:\Q]$ factors.

For such $\underline{m}$, we set
$K[\underline{m}] := \prod_{v\in\places_f} K_v[m_v]$, and
$\operatorname{pr}[\underline{m}] := \prod_{v} \operatorname{pr}_v[m_v]$.
Then $\operatorname{pr}[\underline{m}]$ acts on any unitary $\G(\adele)$-representation, 
and $\sum_{\underline{m}} \operatorname{pr}[\underline{m}] = 1.$
We remark that if $f\in C^\infty(X),$ then $\sum_{\m}\operatorname{pr}[\m]f=f$ and the left
hand side is actually a finite sum. We may refer to this as the decomposition
of~$f$ into pure level components.

If we fix a Haar measure on $\G(\adele_f)$, then
\be\label{eq;volume-Km}
\vol(K[\underline{m}])^{-1} \ll \|\underline{m}\|^{1+\dim(\G)},
\ee
where the implicit constant depends on $\G, \rho$ (cf. \S\ref{notations})
and the choice of Haar measure.
Here one uses a local calculation in order to control
$[K_v[0]:K_v[m_v]]$ for a finite place $v$, see e.g.~\cite[Lemma 3.5]{Nori}.


\subsection{Definition of the Sobolev norms}
For any archimedean place $v$ 
we fix a basis $\{X_{v,i}\}$
for $\mathfrak g_v=\mathfrak g\otimes_F F_v.$
Let $V = L^2(X)$, where, as in the text, $X = \G(F) \backslash \G(\adele)$.
Given an integer $d\geq0$ we define a degree~$d$ Sobolev norm by
\begin{equation} \label{Sobnormdef} 
\Sob_d(f)^2 := \sum_{\m} \left(
 \|\m\|^d\sum_{\mathcal{D}}\|\operatorname{pr}[\m](1+\height(x))^{d}{\mathcal D} f(x)\|_2^2\right).
\end{equation}
where the inner sum is over all monomials 
$\mathcal{D}=\prod_{v\in\places_\infty}\mathcal D_v$
with $\mathcal D_v\in U(\gfrak_v)$  
of degree at most $d_v$ in the given basis $\{X_{v,i}\}$
and $\deg\mathcal D=\sum d_v\leq d.$
For a compactly supported smooth function on $X$ any of these Sobolev norms is finite.
It is easy to see that
\be\label{eq:Sob-inc}
\text{$\Sob_d(f)\leq\Sob_{d'}(f)$ if $d<d'.$}
\ee

Note that since $\height(\cdot)$ is $K_f=K[0]$-invariant, we see that
$\operatorname{pr}[\underline m]$ commutes with multiplication by~$(1+\height(x))$ and with 
the differential operators $\mathcal D_v.$

We note that the contribution of the finite places to the above is related 
to the ``level" of $f,$ since for a finite place $v$ a function of the form
$\operatorname{pr}_v[m_v]f$ should be thought of as having pure level $m_v$ at $v.$
  
Also note that, if $X$ is compact, 
then $\height(\cdot)$ is uniformly bounded and may be dropped from the definition.

\subsection{Property S1.\ -- { Upper bound for} $L^{\infty}$-norms}\label{s;S1}
We shall now verify property S1.\ of the Sobolev norms.
Let us recall from~\eqref{e;injective-Xc} that the map $g\mapsto xg$ is an injection for all $g=(g_\infty,g_f)$
with $g_\infty\in G_\infty=\G(F_{\places_\infty})$ with {$d(g_\infty,1)\leq c_1\height(x)^{-\ref{inj-radius}}$ and $g_f\in K_f$.} 

Let $f$ belong to the completion of $C_c^\infty(X)$ with 
respect to $\Sob_d.$ Suppose first that $f$ is invariant under $K[\underline{m}]$ for some $\m.$
For any $x\in X$ define the function $g\mapsto f(xg)$ on 
\[
\Omega_\infty(x)=\{g\in G_\infty:d(g,1)\leq c_1\height(x)^{-\ref{inj-radius}}\}.
\]
Then by the usual Sobolev inequality, see e.g.~\cite[Lemma 5.1.1]{EMV}, 
there is some integer $d_0>[F:\Q]\dim\G$ so that we have 
\[
|f(x)|^2\ll \sum_{\mathcal D}\tfrac{1}{\vol (\Omega_{\infty}(x))}\int_{\Omega_{\infty}(x)} | \mathcal D f|^2. 
\]
where the sum is taken over all $\mathcal D$ of degree at most $d_0.$

Let {$d\geq 1+{\ref{inj-radius}}d_0,$} if we integrate the above over $K[\underline{m}],$ then in view of the fact that $f$ is invariant under $K[\underline{m}]$ we get 
from~\eqref{eq;volume-Km} and the estimate $\vol(\Omega_\infty(x))^{-1}\ll\height(x)^{{\ref{inj-radius}}[F:\Q]\dim\G}$ that
\begin{align} \label{linftyest} 
|f(x)|^2&\ll \vol(K[\underline{m}])^{-1} \vol(\Omega_\infty(x))^{-1}
\sum_{\mathcal D}\int_{\Omega_\infty(x)\times K[\underline{m}]}| \mathcal D  f|^2\\
\notag&\ll\|\underline{m}\|^{d}\sum_{\mathcal D}\int_{\Omega_\infty(x)\times K[\underline{m}]}
|(1+\height(x))^{d}\mathcal{D}  f|^2\\
\notag&\ll\|\underline{m}\|^{d}\sum_{\mathcal D}
\|(1+\height(x))^{d}\mathcal{D}  f\|_2^2
\end{align}
where again the sum is over all $\mathcal D$ of degree at most $d.$


Let us now drop the assumption
that~$f$ is invariant under a fixed compact subgroup of~$K_f$. 
In this case we may decompose~$f$ into a converging sum $f = \sum_{\m} \operatorname{pr}[\m] f,$ and obtain
\begin{multline}
\label{eq:Cauchy-Schwarz} |f(x)|^2 = |\sum_{\m} \operatorname{pr}[\m] f(x)  |^2
\leq  \sum_{\m} \|\m\|^{-2} \sum_{\m} \|\m\|^{2} |\operatorname{pr}[\m] f(x)|^2
\\ \ll\sum_{\m} \|\m\|^{-2} \sum_{\m, \mathcal D} \|\m\|^{d+2}
\|\operatorname{pr}[\m](1+\height(x))^{d}\mathcal{D}  f\|_2^2
\ll \Sob_{d+2}(f)^2,
\end{multline}
where we used Cauchy-Schwarz, the above, the definition in~\eqref{Sobnormdef} and the estimate
\[
\sum_{\m}\|\m\|^{-2}=\sum_{k\geq 1}\sum_{\m:\|\m\|=k}k^{-2}\ll_\epsilon \sum_k{k^{-2+\epsilon}}<\infty.
\]

\subsection{Property S2.\ -- Trace estimates}
Let $r\geq 0,$ let $\mathcal D_0$ be a monomial of degree at most $r$, and
let $\underline{m}$ be arbitrary.  Furthermore, let $f\in C_c^\infty(X),$ 
and apply \eqref{linftyest} to the function~$\mathcal D_0 \operatorname{pr}[\underline{m}]f,$
multiplying the inequality by $\|\m\|^r(1+\height(x))^r$ we get
\begin{multline*}
\|\underline{m}\|^r|(1+\height(x))^{r}\mathcal D_0\operatorname{pr}[\underline{m}]f(x)|^2\leq \\
c\|\underline{m}\|^{d+r}\sum_{\mathcal D}\int_{\Omega_\infty(x)\times K[\underline{m}]}
|(1+\height(x))^{d+r}\mathcal{D}\operatorname{pr}[\m]  f|^2
\end{multline*}
where the sum is over $\mathcal D$ of degree at most $d+r$ and $d\geq {\ref{inj-radius}}d_0$
is as above.
Moreover, this also gives
\[
\|\underline{m}\|^r|(1+\height(x))^{r}\mathcal D_0\operatorname{pr}[\underline{m}]f(x)|^2\leq
c\|\m\|^{-s}\Sob_{d+r+s}(f),
\]
for all $d$ as above and $s\geq0.$

For $x\in X$ put 
$L_{x,\underline{m}}(f)=\|\m\|^r(1+\height(x))^{r}\mathcal D_0\operatorname{pr}[\underline{m}]f(x).$  
Then the above implies
\[
{\rm Tr}(|L_{x,\m}|^2|\Sob_{d'}^2)\leq c\|\m\|^{-s}\;\;\mbox{ for all $d'\geq d+r+s$ and any $s\geq0$},
\]
see~\cite{BerRez} and~\cite{EMV} for a discussion of relative traces.
 
Integrating over $x\in X,$ using~\eqref{poly} to sum over $\m,$ and summing over $\mathcal D_0$ with $\deg \mathcal D_0\leq r$ we get ${\rm Tr}(\Sob^2_r|\Sob_{d'}^2)\ll1,$ again for all $d'\geq d+r+s$ and $s\geq 2.$

Let us now use the notation of S2: given $d_0$ the above shows that there exists
$d'>d_0$ and $d>d'$ with ${\rm Tr}(\Sob^2_{d_0}|\Sob_{d'}^2)<\infty$ and ${\rm Tr}(\Sob^2_{d'}|\Sob_{d}^2)<\infty.$
To find an orthonormal basis with respect to $\Sob_{d'}$ which is 
orthogonal with respect to $\Sob_d$ as in S2. one may argue as follows.
Recall that $\Sob_{d'}( f )\leq\Sob_d( f)$ and therefore, by Riesz representation theorem, there exists some positive definite operator ${\rm Op}_{d',d}$ so that
\[
\langle f_1,f_2\rangle_{\Sob_{d'}}=\langle{\rm Op}_{d',d}f_1,f_2\rangle_{\Sob_d}\quad\text{for $f_1,f_2\in C_c^\infty(X).$}
\]
This operator satisfies ${\rm Tr}({\rm Op}_{d',d})={\rm Tr}(\Sob_{d'}^2|\Sob_{d}^2)$
and so it is compact.
Now choose an orthonormal basis with respect to $\Sob_d$
consisting of eigenvectors for ${\rm Op}_{d',d}.$
Therefore, this basis is still orthogonal with respect to $\Sob_{d'},$ and S2. follows
from the definition of the relative trace.
 


\subsection{Property S3.\ -- Bounding the distortion by $g\in \G(F_v)$}
Let~$v\in \G(F_v)$ for some~$v\in\places_f$. 
Note that~$g$ commutes with any differential operator~$\mathcal D$ used above
as well as with the averaging and projection operators~$\operatorname{Av}_{v'}[\cdot]$ 
and~$\operatorname{pr}_{v'}[\cdot]$ for~$v'\in\Sigma_f\setminus\{v\}$.

So if~$g\in K_v$ (or more generally~$g\in K_f$) then~$gK_v[m_v]g^{-1}=K_v[m_v]$ for all~$m_v\geq 0$.
This implies that the action of~$g$ commutes also with the decomposition
of~$f\in C^\infty_c(X)$ into pure level components at~$v$, and so~$\Sob_d(g\cdot f)=\Sob(f)$ 
by~\eqref{Htconst} and~\eqref{Sobnormdef}.
 
Let now $g\not\in K_v$ and~$f\in C_c^\infty(X)$. 
This also implies
\[
gK_v[2\log_{q_v}\|g\|+m]g^{-1}\subseteq K_v[m].
\]
Using this, that~$\operatorname{pr}_v[\ell]f$ is invariant under~$K_v[\ell]$ for~$\ell\geq 0$, and
that~$\operatorname{Av}_v[\ell-1](\operatorname{pr}_v[\ell]f)=0$ for~$\ell\geq 1$, we get for all~$m,\ell\geq 0$
 that
\[
\operatorname{pr}_v[m](g\cdot(\operatorname{pr}_v[{\ell}]f))=0
\mbox{ unless }|m-\ell|\leq 2\log_{q_v}\|g\|.
\]  
Applying this and defining~$R=2\log_{q_v}\|g\|$ we get
\begin{align*}
 \|(1+\height)^d\operatorname{pr}_v[m](g\cdot f)\|_2\hspace{-2cm}&\\
&=\Bigl\|\operatorname{pr}_v[m](1+\height)^d\Bigl(g\cdot \sum_{|\ell-m|\leq R}\operatorname{pr}_v[\ell] f\Bigr)\Bigr\|_2\\
&\leq \sum_{|\ell-m|\leq R}\|(1+g^{-1}\cdot\height)^d  \operatorname{pr}_v[\ell] f \|_2\\
&\ll (2R+1)\|g\|^{2d}\max_{|\ell-m|\leq R}\|(1+\height)^d  \operatorname{pr}_v[\ell] f \|_2,
\end{align*}
where we also used~\eqref{HtCont}. Fixing~$f\in C_c^\infty(X)$ 
we now apply this
 for the functions~$\prod_{v'\in\Sigma_f\setminus\{v\}}\operatorname{pr}_{v'}[m_{v'}]\mathcal D f$
and sum over all~$\underline{m}$ and~$\mathcal D$ to get
\begin{multline*}
\Sob_d(g\cdot f)^2
\ll\\ (2R+1)^2\|g\|^{4d}\sum_{\underline{m},\mathcal D}\|\underline{m}\|^d
\sum_{|\ell-m_v|
\leq R}\Bigl\|(1+\height)^d  \operatorname{pr}_v[\ell] 
{\prod_{v'\in\Sigma_f\setminus\{v\}}}\operatorname{pr}_{v'}[m_{v'}]\mathcal Df \Bigr\|_2^2\\
\leq (2R+1)^3\|g\|^{6d}
\Sob_d(f)^2\ll \|g\|^{8d}\Sob_d(f)^2,
\end{multline*}
which gives S3.

For~$v\in\places_\infty$ the argument consists of expressing the element~$\operatorname{Ad}_g(\mathcal D)$
in terms of the basis elements considered in the definition of~$\Sob_d(\cdot)$, and bounding the change of
the height as above.

Let now~$u(\cdot)$ be the unipotent subgroup as in Property S3. In view of the 
definition of $K_w$ we get $\rho\circ\theta_w(\SL_2(\order_w))\subset \SL_N(\order_w).$  
Therefore, $\|u(t)\|\leq |t|_w^{N}$ as was claimed.

\subsection{Property S4. -- Estimating the Lipshitz constant at~$w$}\label{s:new-prop}
Let $f$ belong to the completion of $C_c^\infty(X)$ with 
respect to $\Sob_d.$ First note that if $f$ is invariant under $K[\underline{m}]$ for some $\m,$ then $g\cdot f$ is also invariant under $K[\underline{m}]$ for all $g\in K_w.$
Therefore, $\operatorname{pr}[\m] g\cdot f=g\cdot \operatorname{pr}[\m] f$ for all $g\in K_w.$
Also note that if $g\in K_w[r]$ and $f$ is $K[\m]$ invariant with $m_w\leq r,$ then $g\cdot f=f.$
 

Let now $g\in K_w[r]$ and let $f$ be in the completion of $C_c^\infty(X)$ with 
respect to $\Sob_d.$ Therefore, as in~\eqref{eq:Cauchy-Schwarz} we can use~\eqref{linftyest}  and get
\begin{align*}
|(g\cdot f-f)(x)|^2 &= \left|\sum_{\m} \operatorname{pr}[\m] (g\cdot f-f) (x) \right|^2=\left|\sum_{\m} (g\cdot\operatorname{pr}[\m]f-\operatorname{pr}[\m]f) (x)\right|^2\\
&=\left|\sum_{\m:m_w>r} (g\cdot\operatorname{pr}[\m]f-\operatorname{pr}[\m]f) (x)\right|^2\\
&\leq  \sum_{\m:m_w>r} \|\m\|^{-2} \sum_{\m:m_w>r} \|\m\|^{2} |(g\cdot\operatorname{pr}[\m]f-\operatorname{pr}[\m]f) (x)|^2\\
& \leq  \sum_{\m:m_w>r} \|\m\|^{-2} \sum_{\m} \|\m\|^{2} |\operatorname{pr}[\m](g\cdot f-f) (x)|^2\\
& \ll\sum_{m_w>r} \|\m\|^{-2} \sum_{\m, \mathcal D} \|\m\|^{d+2}
\|\operatorname{pr}[\m](1+\height(x))^{d}\mathcal{D} (g\cdot f-f)\|_2^2\\
&\ll q_w^{-2r} \Sob_{d+2}(g\cdot f-f)^2\ll q_w^{-2r}\Sob_{d+2}(f)^2
\end{align*}
where in the last inequality we used property S3.

\subsection{Property S6. -- Bounds for matrix coefficients}\label{s;S5}
Recall that at the good place~$w$ 
there exists a non-trivial homomorphism $\theta: \SL_2(F_w) \rightarrow H_w\subset \G(F_w)$
such that  $K_{\SL_2}=\theta(\SL_2(\order_w))\subset K_f.$
We also write
$u(t) =\theta\left(\left(\begin{array}{cc} 1 & t \\ 0 & 1 \end{array}\right)\right).$ 

Let $\nu$ be a MASH measure on $X$ which is invariant and ergodic by 
$\theta(\SL_2(F_w)).$ Recall from~\S\ref{ss:spectral-input}  that the
$\SL_2(F_w)$-representation
on
\[
L^2_0(X,\nu)=\Bigl\{f\in L^2(X,\nu):\int fd\nu=0\Bigr\}
\]
is ${1}/{\temp}$-tempered.

Let $f_1, f_2 \in C_c^{\infty}(X)$. Consider 
$I :=  \langle  u(t) f_1, f_2 \rangle_{L^2(\nu)} -\int f_1d\nu \int \bar f_2 d\nu$.
By  \cite{CHH} (see also~\S\ref{eff-decay-sec}) $|I|$ can be bounded above by
\begin{multline} \label{mcunipotent} 
(1+|t|_w)^{-1/2\temp}
\dim (K_{\SL_2}\!\!\cdot\!  f_1)^{1/2} \dim (K_{\SL_2}\!\!\cdot\!  f_2)^{1/2} \|f_1\|_{L^2(\nu)}
\|f_2\|_{L^2(\nu)}.
\end{multline}
Suppose that $f_1$ is fixed by $K[\underline{m}],$
the dimension of $K_{\SL_2}\!\!\cdot\! f_1$ is bounded above
by the number of $K[\underline{m}] \cap K_{\SL_2}$-cosets
in $K_{\SL_2}$, which in turn is bounded by 
\be\label{eq;levelbound-matcoef}
\bigl[K_w[0]: K_w[m_w]\bigr] \ll q_w^{m_w\dim\G}.
\ee

Decomposing $f_1 := \sum \operatorname{pr}_v[\underline{m}] f_1$ and similarly for $f_2$, we see that
 in general:
\begin{align*}
	|I| &\ll (1+|t|_w)^{-1/2\temp} \prod_{i\in\{1,2\}} \left(  \sum_{\underline{m}} \|\m\|^{\frac12\dim\G}  \|\operatorname{pr}[\underline{m}]  f_i \|_{L^2(\nu)}\right)\\
	\end{align*}
We now apply Cauchy-Schwarz inequality and~\eqref{linftyest} to the expression in the parenthesis to get
\begin{align*}
	&\left(\sum_{\underline{m}} \|\m\|^{\frac12\dim\G}  \|\operatorname{pr}[\underline{m}]  f_i \|_{L^2(\nu)}\right)^2\\
	&\ll   \sum_{\underline{m}} \|\m\|^{\dim\G+2}  \|\operatorname{pr}[\underline{m}]  f_i \|^2_{L^2(\nu)}\\
	&\ll\sum_{\m}\|\underline{m}\|^{\dim\G+2+d}\sum_{\mathcal D}
\|(1+\height(x))^{d}\mathcal{D} (\operatorname{pr}[\underline{m}]  f_i)\|_2^2\ll \Sob_{\dim\G+2+d}(f_i).
\end{align*}
This gives property S6.

\subsection{Property S5.\ -- The operator $\mathbb{T}_t$ and Sobolev norms}\label{ss;S3}
We will use the same notation as above.
Recall that we defined the operator $\mathbb{T}_t$ to be $\Av \star \delta_{u(t)} \star \Av$
where $\Av$ is the operation of averaging
over $K_w[\levelbound],$ with $\levelbound>0$ as in Proposition \ref{effgen}.

Here we will modify the argument in the proof of property S1.\ to get the desired property.
Let us note again that $\height(\cdot)$ is invariant under $K[0].$

Let $d\geq\ref{inj-radius}d_0$
and let $f$ be an arbitrary smooth compactly supported
function. Then
\be\label{e;const-der0}
\mathcal D \operatorname{pr}[\m]\mathbb{T}_t\Gpp f=\operatorname{pr}[\m]\mathcal D\mathbb{T}_t\Gpp f=0,\quad
\text{whenever $\deg\mathcal D\geq 1;$}
\ee
Indeed $\Gpp(f)$ is invariant under $\G(\adele)^+$
and the latter contains $\exp(\mathfrak g_v)$ for all $v\in\places_\infty.$
We also note that $\mathbb{T}_t\pi^+=\pi^+$ because $\G(\adele)^+$
contains $\{u(t)\}$ and $\levelbound\geq1$ satisfies $K_w[\levelbound]\subset\G(\adele)^+.$
Given $\m$ let us put $\operatorname{pr}^{(w)}[\m]=\prod_{v\neq w}\operatorname{pr}[m_v].$ Note that
\be\label{eq;Tt-commutes}
\mbox{$\mathcal D\mathbb{T}_t=\mathbb{T}_t\mathcal D\quad$ and 
$\quad \operatorname{pr}^{(w)}[\m]\mathbb{T}_t=\mathbb{T}_t\operatorname{pr}^{(w)}[\m].$}
\ee

For any $\m$ put $\Phi_{\m}(x)=\operatorname{pr}[\m]\mathbb{T}_t(f-\Gpp f )(x),$
note that $\Phi_{\m}$ is $K[\m]$ invariant.
Arguing as in~\eqref{linftyest} for the function $\Phi_{\m}(\cdot)$ 
and using~\eqref{eq;Tt-commutes} we get
\begin{align}
\label{e;Tt-bound1}|\Phi_{\m}(x)|^2
&\ll\vol(K[\underline{m}])^{-1} \vol(\Omega_\infty(x))^{-1}
\sum_{\mathcal D}\int_{\Omega_\infty(x)\times K[\underline{m}]}| \mathcal D  \Phi_{\m}|^2\\
\notag&\ll\|\m\|^d\height(x)^{d}
\sum_{\mathcal D}\|\operatorname{pr}[\m]\mathbb{T}_t(\mathcal{D}(f-\Gpp f))\|_2^2\\
\notag&\ll\|\m\|^d\height(x)^{d}
\|\mathbb{T}_t\|_{2,0}^2\sum_{\mathcal D}\|\operatorname{pr}^{(w)}[\m]\mathcal D f\|^2_2.
\end{align}
where in the last step we used the fact that both $\operatorname{pr}_w[m_w]$ and $\pi^+$
are projections, together with~\eqref{e;const-der0} and~\eqref{eq;Tt-commutes}.

Since $\mathbb{T}_t=\Av\star\delta_{u(t)}\star \Av,$ 
we have: $\operatorname{pr}[\m]\mathbb{T}_t=0$ for all $m_w>\levelbound.$
Also recall that $\sum_{\m}\operatorname{pr}[\m]=1.$ Therefore, 
\[
\mathbb{T}_t(f-\Gpp f )(x)=\sum_{\m: m_w\leq\levelbound}\underbrace{\operatorname{pr}[\m]\mathbb{T}_t(f-\Gpp f )(x)}_{\Phi_{\m}(x)}.
\]

Arguing as in the last paragraph in \S\ref{s;S1}, using~\eqref{e;Tt-bound1} and the above
identity we get
\begin{align*}
|\mathbb{T}_t(f&-\Gpp f )(x)|^2\ll\height(x)^{d}\|\mathbb{T}_t\|_{2,0}^2
\sum_{\substack{\m:m_w\leq\levelbound\\\mathcal D}}\|\m\|^{d+2}\|\operatorname{pr}^{(w)}[\m]\mathcal D f\|^2_2\\
&\ll\height(x)^{d}\|\mathbb{T}_t\|_{2,0}^2(\levelbound+1)q_w^{(d+2)\levelbound}\sum_{\substack{\m:m_w=0\\\mathcal D}}\|\m\|^{d+2}\|\operatorname{pr}^{(w)}[\m]\mathcal D f\|^2_2\\
&\ll\height(x)^{d}\|\mathbb{T}_t\|_{2,0}^2q_w^{(d+2)\levelbound}\sum_{\substack{\m:m_w=0\\\mathcal D}}\|\m\|^{d+2}\bigl\|\mbox{$\sum_{m_w}$}\operatorname{pr}[m_w](\operatorname{pr}^{(w)}[\m]\mathcal D f)\bigr\|^2_2\\
&\ll\height(x)^{d}\|\mathbb{T}_t\|_{2,0}^2q_w^{(d+2)\levelbound}\sum_{\m,\mathcal D}\|\m\|^{d+2}\|\operatorname{pr}[\m]\mathcal D f\|^2_2
\end{align*}
which implies S5.

Let us note that the argument in \S\ref{s;S5} applies to the 
representation of $\SL_2(F_w)$ on $L^2_0(X,\vol_G),$ i.e.\ the orthogonal complement of $\G(\adele)^+$-invariant functions. Suppose this representation is $1/\temp$-tempered,  
then similar to~\eqref{mcunipotent} 
we get
\begin{multline*}
\bigl|\langle u(t)\cdot f_1,f_2\rangle-\langle\Gpp f_1,\Gpp f_2\rangle\bigr|\ll\\
(1+|t|_w)^{-1/2\temp}
\dim (K_{\SL_2}\!\!\cdot\!f_1)^{1/2} \dim (K_{\SL_2}\!\!\cdot\! f_2)^{1/2} \|f_1\|_2
\|f_2\|_2.
\end{multline*}
This estimate and~\eqref{eq;levelbound-matcoef}, in view of the definition of $\mathbb T_t,$ imply
\be\label{eq:T-L2-bound}
\|\mathbb T_t\|_{2,0}\ll (1+|t|_w)^{-1/2\temp} q_w^{2 d \levelbound}.
\ee

\section{The discriminant of a homogeneous set}\label{sec;discrim}
The paper \cite{ELMV} defined the {\em discriminant} of a homogeneous set 
in the case when the stabilizer is a torus. 
Here we shall adapt this definition to the case at hand, see also~\cite[Sec.\ 17]{EMV}.

Let $\H$ be a semisimple, simply connected group defined over $F.$ 
As in \S\ref{sec:intro} we fix a nontrivial  $F$-homomorphism $\iota:\H\to\G$
with central kernel and let $(g_v)_v\in\G(\adele).$ Put $Y=\iota(\H(F)\backslash\H(\adele))(g_v)_v.$

We choose a differential form $\omega$ of top degree on $\mathbf{H}$ 
and an $F$-basis $\{f_1,\dots,f_r\}$ for $\Lie(\H)$ such that $\omega(z)=1$
for 
\[
z=f_1\wedge\dots\wedge f_r \in \wedge^{r} \Lie(\mathbf{H}).
\] 

Using $\rho:\mathbf{G}\rightarrow \SL_N$ and $\iota$ we have
$\rho\circ\iota(z) \in \wedge^r \mathfrak{sl}_N$. 
We put $z_v := \rho\circ\Ad(g_{v})^{-1}\circ\iota(z)$. 
For each $v\in\places,$ we denote by $\omega_v$ the form of top degree on $\H(F_v)$ induced by $\omega.$

Let $\|\;\|_v$ be a compatible system of norms
on the vector spaces $\wedge^r \mathfrak{sl}_N \otimes F_v$. 
In particular, we require at all the finite places $v$ that the norm $\|\;\|_v$ is the max norm. 
Denote by $B$ the bilinear form on $\wedge^r \mathfrak h$ induced by the Killing form. 
We define the {\em discriminant} of the homogeneous set~$Y$ by
\be\label{e;disc}
\disc(Y) = \tfrac{\mathcal{D(\H)}}{\Ecal{(\H)}}\prod_{v} \disc_v(Y) =\tfrac{\mathcal D(\H)}{\Ecal(\H)} \prod_{v} \|z_v\|_v,
\ee
where 
\begin{enumerate}
\item $\disc_v(Y) = |B(\omega_v, \omega_v)|_v^{1/2}\|z_v\|_v$
is the {\em local discriminant} at $v,$ and is independent of the choice of the $F$-basis,
\item $\mathcal{D}(\H)\geq1$ is defined in~\eqref{eq:field-disc}, and
\item $0<\Ecal(\H)\leq1$ is defined in~\eqref{eq:zeta-L}.
\end{enumerate}  
The second equality in~\eqref{e;disc} 
uses the fact that $\prod_{v} |B(\omega_v, \omega_v)|_v = 1$ which is a consequence of the product formula and the equality 
$B(\omega_v, \omega_v) = B(\omega, \omega)$.

One key feature of this definition is that it is closely related to the volume in the sense that
\be\label{eq:disc-vol-statement}
\vol(Y)^\star\ll\disc(Y)\ll\vol(Y)^\star.
\ee
We outline a proof of this in this section. 

Before doing that let us use~\eqref{eq:disc-vol-statement} to complete the discussion from
\S\ref{sec:quadratic}.

\subsection{Proof of Lemma~\ref{lem:genus-spingenus}} \label{finishingproof}
We use the notation from~\S\ref{sec:quadratic}. 
In particular,~$F=\Q$, $Q$ is a positive definite integral quadratic form in~$n$ variables, 
$\H'=\SO(Q)$ and $\H={\rm Spin}(Q).$ 

Let $g_Q\in\PGL(n,\R)$ be so that~$g_Q^{-1}\H'(\R)g_Q=\SO(n,\R)$
and put
\[
Y=Y_Q=\pi(\H(\Q)\backslash\H(\adele))(g_Q,e,\ldots).
\]

We recall that~$K'$ and~$K'(\infty)$ are compact open subgroups of~$\H'(\adele_f)$ and~$\H'(\adele)$ respectively.
Also recall the notation
$\tK=K'\cap\pi(\H(\adele_f))$, and 
$
\tK(\infty)=\pi(\H(\R))\tK.
$
Finally put
\[
\tK_{Q}(\infty)=g_Q^{-1}K^*(\infty)g_Q=g_Q^{-1}K'(\infty)g_Q\cap H_\data.
\]

\begin{lem}\label{lem:spin-genus-K*} We have the following
\be\label{eq:spingenus-volume-main}
\vol(Y)^\star\ll\operatorname{spin\ genus}(Q)\ll\vol(Y).
\ee
\end{lem}

\begin{proof}
Using the definition of the volume as in~\eqref{volume}, we have
up to a multiplicative constant depending on $\Omega_0,$ 
that $\vol(Y)\asymp m_{Y}(\tK_Q(\infty))^{-1}.$
On the other hand we have
\[
\mbox{$1=\mu_\data(Y)=\sum_h \tfrac{m_{Y}\bigl(\tK_Q(\infty)\bigr)}{\ell_h}$}
\] 
where $1\leq \ell_h\leq \#\bigl(\H'(\Q)\cap hK'(\infty)h^{-1}\bigr)$ for every double coset representative 
\[
\H'(\Q)hK^*(\infty)\in \H'(\Q)\backslash \H'(\Q)\pi(\H(\adele))/K^*(\infty).
\]
Since $\ell_h$ is bounded by the maximum of orders of finite subgroups of $\PGL_n(\bbq),$ see e.g.~\cite[LG, Ch.\ IV, App.\ 3, Thm.\ 1]{Serre-LG} we get
that, up to a constant depending on $\Omega_0,$ we have 
\be\label{eq:vol-spingen}
\vol(Y)\asymp\#\bigl(\H'(\Q)\backslash \H'(\Q)\pi(\H(\adele))/K^*(\infty)\bigr).
\ee
Recall that
\[
\operatorname{spin\ genus}(Q)=\#\bigl(\H'(\Q)\backslash \H'(\Q)\pi(\H(\adele))K'(\infty)/K'(\infty)\bigr).
\]
Hence~\eqref{eq:vol-spingen} implies the claimed upper bound in the lemma.

We now turn to the proof of the lower bound. The idea is to use 
strong approximation and discussion in \S\ref{ss:Borel-Prasad} to
relate the orbit space appearing on {the right side of~\eqref{eq:vol-spingen}}
to the spin genus. 

Let $p$ be a good prime for $Y$ given by the proposition in \S\ref{ss:existence-good}.
In view of the strong approximation theorem applied to the simply connected group $\H,$ the choice of $p,$ and the definition of $K^*(\infty)$ we have
\be\label{eq:strong-app-genus}
\pi(\H(\adele))=\pi(\H(\Q))\pi(\H(\Q_p))K^*(\infty).
\ee
where we have identified $\H(\Q_p)$ as a subgroup of $\H(\adele).$

Therefore,
every double coset $\H'(\Q)\pi( h)K'(\infty)$ has a representative in $\pi(\H(\Q_p)).$
That is: we use~\eqref{eq:strong-app-genus} and write $\pi(h)=\pi(\delta,\delta,\ldots)\pi\bigl((h'_p,e,\ldots)\bigr)k^*,$
where $\delta\in\H(\Q),$ $h'_p\in\H(\Q_p)$ and $k^*\in K^*(\infty)$.

Let now $h^{(1)}_p,h_p^{(2)}\in\H(\Q_p)$ be so that 
\be\label{eq:double-coset}
\pi\bigl((h^{(2)}_p,e,\ldots)\bigr)=(\gamma,\gamma,\ldots)\pi\bigl((h^{(1)}_p,e,\ldots)\bigr)k
\ee 
where $\gamma\in\H'(\Q),$ $k\in K'(\infty)$. 
Let us write $k=\bigl(k_p,(k_q)_{q\neq p}\bigr);$ we note that $q=\infty$ is allowed. Then we have 
\[
\gamma k_q=1\quad\text{for all $q\neq p.$}
\] 
Hence we get $k=(k_p,\gamma^{-1},\gamma^{-1},\ldots).$
This, in particular, implies $\gamma\in K'_q$ for all $q\neq p,$ that is
\[
(\gamma,\gamma,\ldots)\in\H'(\Q)\cap\H'(\Q_p)K'(\infty).
\]

Put $\Lambda':=\H'(\Q)\cap\H'(\Q_p)K'(\infty)$ and
$\Lambda:=\pi(\H(\Q))\cap\H'(\Q_p)K'(\infty).$ 
{Taking their projections into $\H'(\Q_p)$, we identify $\Lambda$ and $\Lambda'$ as two lattices in $\H'(\Q_p)$.}
Note that $\Lambda$ is a normal subgroup of $\Lambda'.$ 
We write
\[
\Lambda'/\Lambda=\cup_{i=1}^r\Lambda\gamma_i.
\] 
Also write $K'_p/K^*_p=\cup_{j=1}^{s} k_j K^*_p.$
The above discussion thus implies
\be\label{eq:class-numb-genus}
\pi(h_p^{(2)})\in \bigsqcup_{i,j} \Lambda\gamma_i\pi(h^{(1)}_p)k_j K^*_p.
\ee

Define the natural surjective map from $\pi(\H(\Q))\backslash \pi(\H(\adele))/K^*(\infty)$ to 
$\H'(\Q)\backslash \H'(\Q)\pi(\H(\adele))K'(\infty)/K'(\infty)$ by
\[
\pi(\H(\Q))\bigl(\pi(h_p),e,\ldots\bigr)K^*(\infty)\mapsto \H'(\Q)\bigl(\pi(h_p),e,\ldots\bigr)K'(\infty).
\]
Then since $\Lambda\subset\pi(\H(\Q))$ and $K_p^*\subset K^*(\infty)$ 
we get from~\eqref{eq:class-numb-genus} that the pre-image of 
$\H'(\Q)\bigl(\pi(h_p),e,\ldots\bigr)K'(\infty)$ is contained in $\cup_{i,j}D_{i,j}$
where
\[
D_{i,j}:=\bigl\{\pi(\H(\Q))\bigl(\pi(h'_p),e,\ldots\bigr)K^*(\infty): 
\gamma_i\pi(h_p)k_j\in \Lambda\pi(h'_p)K_p^*\bigr\}.
\]

Note also that $D_{i,j}$ has at most one element.

Putting this all together, we get the following.
\begin{align}
\notag\operatorname{spin\ genus}(Q)&=\#\bigl(\H'(\Q)\backslash \H'(\Q)\pi(\H(\adele))K'(\infty)/K'(\infty)\bigr)\\
\notag{}^{\text{the above discussion}}\leadsto&\gg \frac{\#\bigl(\H'(\Q)\backslash \H'(\Q)\pi(\H(\adele))/K^*(\infty)\bigr)}{[\Lambda':\Lambda][K_p':K_p^*]}\\
\label{eq:class-number-genus-final}{}^{\text{~\eqref{eq:vol-spingen}}}\leadsto&\gg\frac{\vol(Y)}{[\Lambda':\Lambda][K_p':K_p^*]}.
\end{align}

We now bound the denominator in~\eqref{eq:class-number-genus-final}.
First note that 
\[
[K'_p:K_p^*]\leq [\H'(\Q_p):\pi(\H(\Q_p))]\leq\newletter
\]
where $\newletter\ll1$ is an absolute constant.

Bounding the term $[\Lambda':\Lambda]$ is far less trivial and relies on results in \S\ref{ss:Borel-Prasad}.
Put $\widetilde{\Lambda}=\H'(\Q)\cap\pi(\H(\Q_p))K'(\infty).$ Then
\[
[\Lambda':\Lambda]=[\Lambda':\widetilde{\Lambda}][\widetilde{\Lambda}:\Lambda]\leq 
[\H'(\Q_p):\pi(\H(\Q_p))][\widetilde{\Lambda}:\Lambda]\leq \newletter[\widetilde{\Lambda}:\Lambda].
\]
Finally the index $[\widetilde{\Lambda}:\Lambda]$ is controlled as in~\eqref{eq:index-cont}. We note that the quantities 
appearing on the right side of~\eqref{eq:index-cont}, in particular $\places^\flat$, are the same for the group $\pi(\H(\adele)),$
which we used to obtain the bound in~\eqref{eq:class-number-genus-final}, as well  
as for $(g_Q,e,\ldots)^{-1}\pi(\H(\adele))(g_Q,e,\ldots),$ which is used to define $Y.$
Therefore, in view of the {\it equivalence of volume definitions} proposition in \S\ref{ss:Borel-Prasad} we get
\[
\operatorname{spin\ genus}(Q)\gg \vol(Y)^\star
\] 
which is the claimed lower bound.
\end{proof}

By~\eqref{eq:disc-vol-statement} we also know~$\vol(Y)\asymp\disc(Y)^\star$. So it remains to discuss
the genus of~$Q$. For this we are making the following  

\begin{claim}
For any given~$T$
the number of equivalence classes of quadratic forms~$Q$ with~$|\operatorname{spin\ genus}(Q)|<T$
is~$\ll T^\star$. 
\end{claim}

\begin{proof}
Let $X'\subset X$ be a compact set so that $\mu(X')>0.9$ for any MASH {measure $\mu$}; 
this exists by Corollary~\ref{non-divergence-adelic}. 
Suppose $Q$ is a quadratic form with $|\operatorname{spin\ genus}(Q)|<T$. We see 
that 
\[
\text{$\G(\Q)(g_Q,e,\ldots)H_\data\cap X'$ is non-empty.} 
\]

We may assume that~$X'$
is invariant under~$\SO(n,\R)$, then the above also gives some
$\iota(h)\in \iota(\H(\adele_f))$ with~$\G(\Q)(g_Q,\iota(h))\in X'$.
By the correspondence between the spin genus of~$Q$ and the~$g_Q^{-1}K'(\infty)g_Q$-orbits in~$Y$ we have found a quadratic form~$Q_h$
in the spin genus of~$Q$ for which the conjugating matrix~$g_{Q_h}$ 
can be chosen from a fixed compact subset~$L\subset\PGL(n,\R)$. 

Let us now note that the spin genus of~$Q_h$ equals the spin genus of~$Q.$  
So in view of~\eqref{eq:spingenus-volume-main} 
the MASH set~$Y_{Q_h}$
associated to the quadratic form~$Q_h$ has volume which 
is bounded above and below by powers of the volume of the MASH set~$Y_Q$.

We will use the assumption~$g_{Q_h}\in L$ in order to relate the size of the spin genus with the volume of the rational MASH set 
$Y_h'=\pi(\H_{Q_h}(\Q)\backslash\H_{Q_h}(\adele))$, where~$\H_{Q_h}$ is the spin cover of the orthogonal
group of~$Q_h$ (and we intentionally did not include~$g_{Q_h}$ in the definition). 
Indeed, since $g_{Q_h}\in L,$ 
we get from \eqref{eq:disc-vol-statement} and~\eqref{eq:spingenus-volume-main} that 
\be\label{eq:spingenus-disc}
|\operatorname{spin\ genus}(Q)|\asymp\vol(Y_{Q_h})^\star\asymp\vol(Y_h')^\star\asymp\disc(Y_h')^\star.
\ee
The definition of $\disc(Y_h'),$ see \eqref{e;disc},
gives 
\[
\disc(Y_h')=\tfrac{\mathcal D(\H_{Q_h})}{\Ecal(\H_{Q_h})}\prod_{v\in\places}\|z_v\|_v=
\tfrac{\mathcal D(\H_{Q_h})}{\Ecal(\H_{Q_h})}\|z(h)\|_\infty
\] 
where $z(h)\in \wedge^r \mathfrak{sl}_N$ is the primitive integral vector 
which is a rational multiple of $(z_v)_v;$ and in the second equality we used the product formula. 
Note that $z(h)=f_1\wedge\cdots\wedge f_r$ determines the group $\H_{Q_h}$ and hence the form $Q_h$ (up to homotheties).  
In particular, since ${\mathcal D(\H_{Q_h})}/{\Ecal(\H_{Q_h})}\geq1$ we get 
$\|z(h)\|_{\infty}\ll T^\star$. 
As there are only~$\ll T^\star$ many integral vectors 
of norm~$\ll T^\star$ in~$\wedge^r\mathfrak{sl}_N$
we obtain the claimed estimate.
\end{proof} 

We can now finish the proof of Lemma~\ref{lem:genus-spingenus}.
Let us recall from the definitions that
\[
\operatorname{genus}(Q)=\bigsqcup_i\operatorname{spin\ genus}(Q_i)
\]
where $Q_i\in\operatorname{genus}(Q).$ 
Let $Y_{Q_i}$ denote the MASH's which correspond to $Q_i$ as above.
Then by~\eqref{eq:spingenus-volume-main} we have
$\operatorname{spin\ genus}(Q_i)\asymp\vol(Y_i)^\star.$ Note however that
$\vol(Y_i)=\vol(Y_j)$ for all $i,j$ since the corresponding algebraic stabilizers are the same;
indeed $Y_i=Y_jh$ for some $h\in\SO(Q)$ and $\SO(Q)$ normalizes the algebraic stabilizer
of $Y_i.$  
Suppose now~$\operatorname{genus}(Q)=S$. 
Then the above MASH sets all have the same volume $\vol(Y_i)=V$
which, in view of the above claim, gives~$S\ll V^\star$ and finishes the proof of 
Lemma~\ref{lem:genus-spingenus}.

\subsection{Expressing the discriminant in terms of the volume}\label{ss:infinite-place}
Recall the notation from the proof of part (3) in the proposition\footnote{We note that the standing assumption in \S\ref{ss;k-w-star} was that $\H$ is $F$-simple. However, the proof of part (3) in the proposition in \S\ref{ss;k-w-star} works for the case of semisimple groups.} in \S\ref{ss;k-w-star}. In particular, we fixed finitely many standard homomorphisms of $\H(F_v)$ into $\G(F_v)$ for any archimedean place $v.$ 
Recall also that corresponding to {the standard homomorphism $\jmap_0=\Ad(g_0)\circ\jmap$} 
we have a Euclidean structure $\mathfrak p_0$ and that $\|\;\|_{\mathfrak p_0}$ denotes the corresponding 
Euclidean norm.

In this section we have fixed a compatible family of norms $\|\;\|_v.$
Note that for any archimedean place $v$ we have $\|\;\|_v\asymp\|\;\|_{\mathfrak p_0}$ 
with constants depending only on the dimension.
Therefore, without loss of generality we may and will assume that $\{f_1,\ldots,f_r\}$ are chosen so that
\be\label{eq:basis-arch-appB}
1/c\leq\|\wedge^r\jmap_0(f_1\wedge\cdots\wedge f_r)\|_{\mathfrak p_0}\leq c
\ee
for any archimedean place $v$ where $c$ is a universal constant.

Fix an archimedean place $v,$ as in the proof of part (3) in the proposition in \S\ref{ss;k-w-star} we have
\begin{multline}\label{eq:vol-arch-appB-1}
u_1\wedge\cdots\wedge u_r=\\\tfrac{\|u_1\wedge\cdots\wedge u_r\|_{\mathfrak p_0}}{\|\wedge^r\Ad(g_0^{-1}) (\wedge^r\jmap_0(f_1\wedge\cdots\wedge f_r))\|_{\mathfrak p_0}}\wedge^r\Ad(g_0^{-1}) (\wedge^r\jmap_0(f_1\wedge\cdots\wedge f_r)).
\end{multline}
where $\{u_1,\ldots,u_r\}$ is chosen as in there.

Note that $1/c\leq\|\wedge^r\Ad(g_0^{-1}) (\wedge^r\jmap_0(f_1\wedge\cdots\wedge f_r))\|_{\mathfrak p_{\jmap}}\leq c,$ by~\eqref{eq:basis-arch-appB}. Hence we get
\be\label{eq:vol-arch-appB-2}
J_v\asymp\|u_1\wedge\cdots\wedge u_r\|_{\mathfrak p_{\jmap}}\asymp\tfrac{1}{\|\wedge^r\Ad(g_0^{-1}) (\wedge^r\jmap_0(f_1\wedge\cdots\wedge f_r))\|_{\mathfrak p_0}}\asymp\tfrac{1}{\|z_v\|_v}
\ee
where the implied constants are absolute.

Therefore, {in view of~\eqref{eq:def-Jv-Arch} and~\eqref{e;vol-est}}, in order to prove~\eqref{eq:disc-vol-statement} we need to control the contribution from finite places. 

We will use the notation from \S\ref{s;good-place}. Since $\H$ is simply connected we can write $\H$ as the direct product $\H=\H_1\cdots\H_k$ of its $F$-almost simple factors. Let $F_j'/F$ be a finite extension so that $\H_j={\rm Res}_{F_j'/F}(\H_j')$ where $\H_j'$ is absolutely almost simple $F_j'$-group for all $1\leq j\leq k.$
Then $[F_j':F]$ is bounded by $\dim\H.$ 
Let $\mathcal H'_j$ and $L_j/F'_j$ be defined as in \S\ref{volumes}.  Put
\be\label{eq:field-disc}
\mathcal D(\H)=\Bigl(\textstyle\prod_j D_{L_j/F_j}^{\mathfrak s(\Hcal'_j)} D_{F'_j}^{\dim\H'_j}\Bigr)^{{1}/{2}}.
\ee

For each $j$ let $\omega'_j$ denote a differential form of top degree on $\mathbf{H}'_j$ 
and choose an $F'_j$-basis $\{f^{(j)}_1,\dots,f^{(j)}_{r_j}\}$ for $\Lie(\H'_j)$ so that 
$\omega'_j(z^{(j)})=1$
where
\[
z^{(j)}=f^{(j)}_1\wedge\dots\wedge f^{(j)}_{r_j} \in \wedge^{r_j} \Lie(\mathbf{H}'_j).
\] 

We may and will work with the $F$-basis $\{f_1,\ldots,f_r\}$ for $\Lie(\H)$ obtained from
$\{f^{(j)}_i\}$ using the restriction of scalars, i.e.\ we assume fixed a basis $\{e^{(j)}_l\}$ for $F'_j/F$ and write $f^{(j)}_i$
in this basis for each $1\leq i\leq r_j$. 
As before put $z=f_1\wedge\dots\wedge f_r$
and let $\omega$ be a form of top degree on $\H$ so that $\omega(z)=1.$

For each $v'\in\places_{F_j'},$ let $\omega'_{j,v'}$ denote the 
form of top degree on $\H'_j(F'_{j,v'})$ induced by $\omega_j'.$ 
Similarly, for any $v\in\places$ let $\omega_v$ denote the form of top degree 
on $\H(F_v)$ induced by $\omega.$

Given $v\in\places$ we define a form of top degree on $\H(F_v)$
by $\widetilde\omega_v:=\bigl((\omega'_{j,v'})_{v'|v}\bigr)_j.$
Since $\H(F_v)$ is naturally isomorphic to $\prod_j\prod_{v'|v}\H'_j(F'_{j,v'}),$ 
it follows from the definitions that $\widetilde\omega_v(z)=1.$
Therefore, for every $v\in\places$ we have $\widetilde\omega_v=\omega_v.$

Let $v\in\places_{F,f},$
following our notation in \S\ref{s;volume-form}, we denote by
$|\omega_{v}|$ the measure induced on $\Lie(\mathbf{H}) \otimes F_{v}$ and abusing the notation  
the measure on $\H(F_{v})$.
Since $\omega_v(z)=1,$ 
the $\order_v$-span of the $\{f_i\}$ has volume $1$ with respect to $|\omega_v|$.  
Applying a suitable change of basis {to the $\{f_i\}$}, 
we may assume that there is an integral basis, $\{e_1, \dots, e_{N^2-1}\},$ for $\mathfrak{sl}_N(\order_v)$ 
with the property that each $\rho\circ\Ad(g_{v}^{-1})\circ\iota(f_i) = c_i e_i$, for $1 \leq i \leq r$. Then, $\|z_v\|_v = \prod_i |c_i|_v$ where $\|\;\|_v$ denotes the compatible family of norms which we fixed before
and $z_v=\rho\circ\Ad(g_{v}^{-1})\circ\iota(z).$

For every $v\in\places_f$ let  
$\mathsf H_v'$ be the scheme theoretic closure of $\rho(g_v^{-1}\iota(\H)g_v)$ 
in $\SL_N/\mathcal \order_v.$ Then for each $v\in\places_f$
we have 
\[
\mathsf H_v'(\order_v)=\rho(g_v^{-1}\iota(\H)g_v)\cap \SL_N(\order_v).
\]
Put $\mathsf H_v:=\iota^{-1}(g_v\rho^{-1}(\mathsf H_v')g_v^{-1})$ for any $v\in\places_f.$ 

Recall that $K_v^*=\iota^{-1}(g_v\rho^{-1}(\SL_N(\order_v))g_v^{-1});$
using the above notation we have $K_v^*=\mathsf H_v(\order_v).$ 

We write $\Lie(K_v^*)=D\rho\circ\Ad(g_{v}^{-1})\circ\iota (\mathrm{Lie}(\H) \otimes F_v) \cap   \mathfrak{sl}_N(\order_v)$.

Let $\red_v$ denote reduction mod $\varpi_v$ with respect to the scheme structure 
induced by $\mathsf H_v.$ In particular $\red_v K_v^*=K_v^*/(K_v^*)^{(1)}$
where {$(K_v^*)^{(1)}$ is the first congruence subgroup of $\mathsf H_v(\order_v).$}
For $p_v\gg1$, $(K_v^*)^{(1)}$ is the image under the exponential map of the first congruence subalgebra of $\Lie(K_v^*)$.


Let us recall that $k_v$ is the residue field
of $F_v$ with ${\rm char}(k_v)=p_v$ and $\# k_v=q_v=p_v^l$ for some $l\leq [F:\Q].$ Similarly $k'_{j,v'},$ 
is the residue field of $F'_{j,v'}$ and $\# k'_{j,v'}=q'_{j,v'}=p_v^{l_j}.$ 

With this notation, the above discussion implies that for $p_v\gg1$ we have
\begin{align*}
\|z_v\|_v &= \textstyle \prod_i|c_i|_v=\bigl|\omega_v(c_1^{-1}f_1\wedge\dots\wedge c_r^{-1}f_r)\bigr|^{-1}\\
&=\bigl(|\omega_v|(\{u \in  \mathrm{Lie}(\H) \otimes F_v: D\rho\circ\Ad(g_{v}^{-1})\circ\iota (u) \in   \mathfrak{sl}_N(\order_v) \})\bigr)^{-1}\\
& =\bigl(|\omega_v|(\Lie(K_v^*))\bigr)^{-1}\\
&=\bigl(|\omega_v|(K_v^*)\bigr)^{-1}(\# (\red_v K_v^*)\cdot q_v^{-\dim\H}).
\end{align*}

In the last equality, we used the fact that $\exp$ is measure preserving 
diffeomorphism on $\mathfrak{sl}^{(1)}_N(\order_v)$ and $\mathfrak h[1]$ for all $p_v\gg1.$

For small $p_v$, not covered above,
$\exp$ is a measure preserving diffeomorphism on $\mathfrak{h}[m]$ for large enough $m;$
see \S\ref{properties}, in particular~\eqref{dderiv} and discussion in that paragraph.
Hence the contribution of these small primes is $\ll1$.

Recall that $\H=\H_1\cdots\H_k$ is a direct product and $|\omega_v|=\prod_j\prod_{v'|v}|\omega'_{j,v'}|.$ 
Therefore, the above, 
in view of~\eqref{e;vol-est},~\eqref{e;disc} and~\eqref{eq:vol-arch-appB-2}, implies 
\be\label{eq:disc-vol}
\prod_v\|z_v\|\asymp \tfrac{1}{D(\{F'_j\})}\vol(Y) \prod_{v\in\places_f} {\# (\red_v K_v^*)}\cdot q_v^{-\dim\H}.
\ee
where $D(\{F'_j\})=\prod_j D_{F'_j}^{\dim\H'_j/2}.$ 

\subsection{The upper bound}
Let the notation be as in~\S\ref{ss:integral-norm}, in particular, for all $j$ and all $v'\in\places_{F'_j,f}$ the parahoric subgroup $\mathcal{P}'_{j,v'}$ of maximum volume in $\Hcal'_j(F'_{j,v'})$ is fixed as in that section. 
Abusing the notation, we denote the corresponding smooth $\order'_{j,v'}$ 
group scheme by $\mathcal{P}'_{j,v'}.$ 
Given a $v\in\places_{F,f}$ put $\mathcal P_v:=\prod_j\prod_{v'|v}\mathcal P_{j,v'}.$
Define
\begin{align}
\label{eq:zeta-L}
\mathcal E(\H):&=\textstyle\prod_{v\in\places_{F,f}}\#\underline{\mathcal P_v}(k_v)q_v^{-\dim\H}\\
\notag&=\textstyle{\prod_j\prod_{v'\in\places_{F'_j,f}}} \# \underline{\mathcal P'_{j,v'}}(k'_{j,v'})\cdot (q'_{j,v'})^{-\dim \H'_j}.
\end{align}

Using Prasad's volume formula and the order
of almost simple finite groups of Lie type, see~\cite[Rmk. 3.11]{Pr}, we have the following.
The quantity $\mathcal E(\H)$ is a product of the values of the Dedekind zeta functions of $F'_j$ and certain Dirichlet $L$-functions attached to $L_j/F'_j$ at some integer points. 
In particular, $\mathcal E(\H)$ is a positive constant depending on {$F'_j$, $L_j,$ and $\mathcal H'_j$.} Moreover,~\cite[\S2.5, \S2.9]{Pr} imply 
\[
\# \underline{\mathcal P'_{j,v'}}(k'_{j,v})\cdot (q'_{j,v})^{-\dim \H'_j}<1.
\]
All together we have shown
\be\label{eq:EH-estimate}
0<\Ecal(\H)\leq 1.
\ee

Let $\places^\flat_{\rm ur}$ denote the set of places $v\in\places_{F,f}$
so that
\begin{itemize}
\item for all $j$ and all $v'\in\places_{F'_j,f}$ with $v'|v$ we have $v'$ is unramified 
in $L_j,$ and
\item one of the following holds
\begin{itemize}
\item there is some $j$ and some $v'|v$ so that the group $\H'_j$ is not quasisplit over $F'_j,$ or
\item all $\H'_j$'s are quasisplit over $F'_{j,v'},$ hence $\H_j'$ is isomorphic to $\Hcal_j'$ over $F_{j,v'}'$ for all $j$ and all $v'|v,$ but $K_v^*$ is not hyperspecial.
\end{itemize}
\end{itemize}
It was shown in~\eqref{eq:local-ind-improve} that we have
\be\label{eq:part(ii)-push}
{\lambda_v|\omega_v|(K_v^*)\leq\tfrac{p_v}{p_v^2+1}\quad \text{for all $v\in\places^\flat_{\rm ur}$}}
\ee
if $q_v>13.$

Let $\places_{\rm r}^\flat$ be the set of places $v\in\places_{F,f}$ so that 
there exists some $j$ and some $v'|v$ in $F'_j$ which ramified in $L_j.$ Put
$\places^\flat:=\places_{\rm ur}^\flat\cup\places_{\rm r}^\flat.$

Put $D(\{L_j\},\{F'_j\}):=\prod_j D_{L_j/F'_j}^{\mathfrak s(\Hcal'_j)/2}.$
Then, as we got~\eqref{eq;volume-upperbd} from~\eqref{e;vol-est}, 
in view of~\eqref{eq:part(ii)-push} we get
\be\label{eq:sigma-flat-disc}
\vol(Y)\gg D(\{L_j\},\{F'_j\}) D(\{F'_j\})\prod_{v\in\places_{\rm ur}^\flat} p_v.
\ee  

Combining~\eqref{eq:disc-vol} and~\eqref{eq:sigma-flat-disc} we get the upper bound
as follows.
\begin{align*}
\disc(Y)&=\tfrac{\mathcal D(\H)}{\mathcal E(\H)}\prod_v\|z_v\|_v\\
&\asymp D(\{L_j\},\{F'_j\})\vol(Y) \tfrac{\prod_{v\in\places_f} \# (\red_vK_v^*)\cdot q_v^{-\dim\H}}{\mathcal E(\H)}\\
&\ll\vol(Y)^\star \tfrac{\textstyle\prod_{v\in\places_f} \# (\red_v\mathsf H_v({k_v}))\cdot q_v^{-\dim\H}}{\mathcal E(\H)}\\
&\ll\vol(Y)^\star\prod_{v\in\places^{\flat}}\tfrac{\# (\red_v\mathsf H_v({k_v}))}{\# \underline{\mathcal P_v}(k_v) }\ \ \ll\ \  \vol(Y)^\star
\end{align*}
where in the last inequality we used $\# (\red_v\mathsf H_v({k_v}))\leq q_v^\star,$
and also the fact that for any $v\in\places_{\rm r}^\flat$ we have $p_v|D(\{L_j\},\{F_j\}).$ 

\subsection{The lower bound}
We now turn to the lower bound. 
For this part we work with normalized volume forms. Fix the notation
\[
\lambda_v|\omega_v|:=\textstyle{\prod_j\prod_{v'|v}}\lambda'_{j,v'}\omega'_{j,v'},
\] 
see \S\ref{ss:integral-norm} for the notation on the right side of the above.

We will need the following. If $\mathsf M$ is a connected linear algebraic group over $k_v,$ then
\be\label{eq:order-conn-finite}
(q_v-1)^{\dim\mathsf M}\leq \#\mathsf M(k_v)\leq (q_v+1)^{\dim\mathsf M};
\ee
see e.g.~\cite[Lemma 3.5]{Nori}.

Given a parahoric subgroup $P_v$ of $\H(F_v),$
let $\mathfrak{P}_v$ denote the smooth $\order_v$ group scheme
associated to it by Bruhat-Tits theory. Recall from \S\ref{ggo} that
$P_v$ maps onto $\underline{\mathfrak P_v}(k_v)$. 
We also remark that since $\H$ is simply connected the $k_v$-group scheme 
$\underline{\mathfrak{P}_v}$ is connected, see~\cite[3.5.3]{Ti}.

Let $v\in \places_{f},$ for any $j$ and any $v'|v$
we choose a parahoric subgroup $P'_{j,v'}$ which is minimal among those parahoric subgroups containing $\pi_{j,v'} K_v^*;$ here $\pi_{j,v'}$ denotes the natural projection. 
Put 
\[
P_v:=\prod_j\prod_{v'|v} P'_{j,v'}.
\] 
Then $K_v^*\subset P_v.$   

We will prove the lower bound in a few steps.
First we prove some local estimates, i.e.\ we bound  
 terms appearing in the product on the right side of~\eqref{eq:disc-vol} for all $v\in\places_f$. 
Taking the product of these estimates then we will get the lower bound. 

\medskip
\noindent
{\bf Step 1.}
We have
\begin{align}
\label{eq:step1}&\bigl(\lambda_v|\omega_v|(K_v^*)\bigr)^{-1}\bigl(\# (\red_v K_v^*)\cdot q_v^{-\dim\H}\bigr)=\\
\notag&\bigl(\lambda_v|\omega_v|(P_v)\bigr)^{-1}[P_v:K_v^*]\bigl(\# (\red_vK_v^*)\cdot q_v^{-\dim\H}\bigr)=\\
\notag&\bigl(\lambda_v|\omega_v|(P_v)\bigr)^{-1}\tfrac{[P_v:P_v^{(1)}][P^{(1)}_v:(K_v^*)^{(1)}]}{[K_v^*:(K_v^*)^{(1)}]}(\# \bigl(\red_v K_v^*)\cdot q_v^{-\dim\H}\bigr)
={}^{\text{by (3) in~\S\ref{ggo}}}\\
\notag&\bigl(\lambda_v|\omega_v|(P_v)\bigr)^{-1}[P^{(1)}_v:(K_v^*)^{(1)}](\#\underline{\mathfrak{P}_v}(k_w))\cdot q_v^{-\dim\H})
\end{align}
where $P_v^{(1)}$ and $(K_v^*)^{(1)}$ denote the first congruence subgroups
defined using the $\order_v$-scheme structures $\mathfrak P_v$ and $\mathsf H_v$ respectively.

\medskip
\noindent
{\bf Step 2.}
In this step we will estimate the contribution coming from the product
$\prod_{v\in\places_f}\bigl(\#\underline{\mathfrak{P}_v}(k_v)\bigr)\cdot q_v^{-\dim\H}$.

The fact that for any $v\in\places_{\rm r}^\flat$ we have 
$p_v|D(\{L_j\},\{F_j\})$ together with~\eqref{eq:sigma-flat-disc} implies $\#\Sigma^\flat\ll\log(\vol(Y)).$

Now since $\underline{\mathfrak{P}_v}$ is connected, we can use~\eqref{eq:order-conn-finite} together with the definition of~$\mathcal E(\H)$ and get
\begin{align}
\notag
\prod_{v\in\Sigma_f}\Bigl(\#\underline{\mathfrak{P}_v}(k_v)\cdot q_v^{-\dim\H}\Bigr)&=
{\prod_{v\in\Sigma_f}\bigl(\#\underline{\mathcal{P}_v}(k_v)\cdot q_v^{-\dim\H}\bigr)}
\prod_{v\in\Sigma^\flat}\tfrac{\#\underline{\mathfrak{P}_v}(k_v)\cdot q_v^{-\dim\H}}
{\#\underline{\mathcal{P}_v}(k_v)\cdot q_v^{-\dim\H}}\\
\label{eq:lower-bd-1}&\gg \mathcal{E}(\H)\bigl(\log\vol(Y)\bigr)^{-\ref{log-est}}
\end{align}
for some $\consta>0\label{log-est}$ depending only on $F$ and $\G.$ 

\medskip
\noindent
{\bf Step 3.}
We will now get a control over $\bigl(\lambda_v|\omega_v|(P_v)\bigr)^{-1}[P^{(1)}_v:(K_v^*)^{(1)}]$. 

\medskip
We claim that there exists some 
$0<\consta<1\label{dlexp},$ depending only on $\dim_F\G,$ so that for all $v\in\places_f$ at least one of the following holds: either
\begin{align}
\notag
\bigl(\lambda_v|\omega_v|(P_v)\bigr)^{-1}[P^{(1)}_v:(K_v^*)^{(1)}]&\geq\bigl(\lambda_v|\omega_v|(K_v^*)\bigr)^{-\ref{dlexp}}\\\label{eq:local-lower} &=\bigl(\lambda_v|\omega_v|(P_v)\bigr)^{-\ref{dlexp}}\left(\tfrac{(\#\underline{\mathfrak{P}_v}(k_v))[P^{(1)}_v:(K_v^*)^{(1)}]}{\# (\red_vK_v^*)}\right)^{\ref{dlexp}},
\end{align}
or $v\in \places^\flat_{\rm r}$ and
\begin{equation}\label{eq:local-lower-ramified}
\bigl(\lambda_v|\omega_v|(P_v)\bigr)^{-1}[P^{(1)}_v:(K_v^*)^{(1)}]\geq1\geq\bigl(\lambda_v|\omega_v|(K_v^*)\bigr)^{-\ref{dlexp}} p_v^{-1/2}.
\end{equation}

Let us first recall that $\lambda_v|\omega_v|(P_v)\leq1,$ see~\cite[Prop.\ 2.10]{Pr}.
Therefore, if $P_v=K_v^*,$ then~\eqref{eq:local-lower} holds for any $0<\ref{dlexp}<1.$ 
In particular, if $v\not\in\places_f\setminus\places^\flat,$ then~\eqref{eq:local-lower} holds for any $0<\ref{dlexp}<1.$

Recall that $K_v^*\subset P_v$.
Assume first that $(K_v^*)^{(1)}\subsetneq P_v^{(1)}.$

Then, since $P_v^{(1)}$ is a pro-$p_v$ group we have
$
[P_v^{(1)}:(K_v^*)^{(1)}]\geq p_v.
$
We again note that by~\cite[Prop.\ 2.10]{Pr} we have 
\be\label{eq:prasad-overandover}
q_v^{-\dim\H}\leq\lambda_v|\omega_v|(P_v)\leq 1.
\ee
Therefore,~\eqref{eq:local-lower} follows if we show
\begin{align*}
\#(\red_vK_v^*)^{\ref{dlexp}}[P_v^{(1)}:(K_v^*)^{(1)}]^{1-\ref{dlexp}}&\geq 
\bigl(\#\underline{\mathfrak P_v}(k_v)\bigr)^{\ref{dlexp}}\\
&\geq \bigl(\lambda_v|\omega_v|(P_v)\bigr)^{1-\ref{dlexp}}\bigl(\#\underline{\mathfrak P_v}(k_v)\bigr)^{\ref{dlexp}}.
\end{align*}

The second inequality holds for any
$0<\ref{dlexp}<1$ in view of the upper bound in~\eqref{eq:prasad-overandover}. 
The first inequality follows from the upper bound in~\eqref{eq:order-conn-finite} 
and our assumption $[P_v^{(1)}:(K_v^*)^{(1)}]\geq p_v$ if we take $0<\ref{dlexp}<1$ to be small enough.

Similarly, if $\lambda_v|\omega_v|(P_v)\leq 2/p_v$, then~\eqref{eq:local-lower} becomes
\begin{align*}
\bigl(\lambda_v|\omega_v|(P_v)\bigr)^{-1+\ref{dlexp}}[P^{(1)}_v:(K_v^*)^{(1)}]^{1-\ref{dlexp}}&\geq (p_v/2)^{-1+\ref{dlexp}}\\
&\geq \left(\tfrac{\#\underline{\mathfrak{P}_v}(k_v)}{\# (\red_vK_v^*)}\right)^{\ref{dlexp}}.
\end{align*}

Since $\tfrac{\#\underline{\mathfrak{P}_v}(k_v)}{\# (\red_vK_v^*)}=p_v^\star$, the above estimate, and hence~\eqref{eq:local-lower}, hold for all small enough $\ref{dlexp}$ provided that $\lambda_v|\omega_v|(P_v)\leq 2/p_v$.

In view of these observation and~\cite[Prop.~2.10]{Pr} we get that~\eqref{eq:local-lower} holds 
unless $P_v^{(1)}=(K_v^*)^{(1)}$ and we are in one of the following cases. 
\begin{itemize}
\item $v\in\places^\flat_{\rm ur}$, $\H$ is $F_v$-quasi-split, 
and $P_v$ is a hyperspecial parahoric subgroup, or  
\item $v\in\places^\flat_{\rm r}$ and $P_v$ is a special parahoric subgroup. 
\end{itemize}

First note that under the assumption $P_v^{(1)}=(K_v^*)^{(1)}$, the estimate in~\eqref{eq:local-lower-ramified} follows from~\eqref{eq:prasad-overandover} so long as we choose $\ref{dlexp}$ small enough. This establishes the claim for $v\in\places^\flat_{\rm r}$. 

Therefore, we may now assume that $P_v^{(1)}=(K_v^*)^{(1)}$, $v\in\places^\flat_{\rm ur}$, and $P_v$ is hyperspecial.
We claim that these imply $P_v=K_v^*$ if $p_v\gg 1$ which then implies that~\eqref{eq:local-lower} holds for any $0<\ref{dlexp}<1.$

Assume $p_v\gg1$ is large enough, so that the exponential map 
is a diffeomorphism from ${\mathfrak sl}_N^{(1)}(\order_v)$ onto $\SL_N^{(1)}(\order_v).$
Our assumption $P_v^{(1)}=(K_v^*)^{(1)}$ implies that $\Lie(K_v^*)=\Lie(\mathfrak P_v)$.  
Since $P_v$ is hyperspecial, we have $P_v/P_v^{(1)}$ is the $k_v$-points of a 
semisimple group which is generated by unipotent subgroups. These unipotent 
subgroups are reduction mod $\varpi_v$ of unipotent subgroups of $P_v$, \cite[\S3.5.1]{Ti}.
In view of our assumption $p_v\gg1$, unipotent subgroups of $P_v$ are obtained using the exponential map from
$\Lie(\mathfrak P_v)=\Lie(K_v^*)$.
Hence $K_v^*$ surjects onto $P_v/P_v^{(1)}$. Since $P_v^{(1)}=(K_v^*)^{(1)},$ this implies $P_v=K_v^*$ 
as we claimed.

\medskip
\noindent
{\bf Step 4.} We now conclude the proof of the lower bound.
We use the notation 
\[
J_\infty=\Bigl(\textstyle\prod_{v\in\places_\infty}J_v\Bigr)^{-1}.
\]

Recall that in view of part (iii) of the proposition in \S\ref{ss;k-w-star} 
we have $J_v\ll1$ for all the archimedean places $v$.  
Taking the product of~\eqref{eq:local-lower} over all $v\in\places_f,$ using the fact that $p_v|D(\{L_j\},\{F'_j\})$
for all $v\in\places^\flat_{\rm r}$ and~\eqref{eq:local-lower-ramified},
and arguing as in~\S\ref{s;vol-hom} we get the lower bound as follows.

\begin{align*}
&\disc(Y)=\tfrac{\mathcal D(\H)}{\mathcal E(\H)} \prod_v\|z_v\|_v\\
&\asymp D(\{L_j\},\{F'_j\})\vol(Y)\tfrac{\prod_{v\in\places_f} \# (\red_vK_v^*)\cdot q_v^{-\dim\H}}{\mathcal E(\H)}&&\text{\eqref{eq:disc-vol}}\\
&\gg \mathcal D(\H)\prod_{v\in\places}\bigl(\lambda_v|\omega_v|(K_v^*)\bigr)^{-1}\tfrac{\prod_{v\in\places_f} \# (\red_vK_v^*)\cdot q_v^{-\dim\H}}{\mathcal E(\H)}&&\text{\eqref{eq;volume-upperbd}}\\
&\gg \mathcal D(\H)J_\infty\tfrac{\prod_{v\in\places_f}\bigl(\lambda_v|\omega_v|(K_v^*)\bigr)^{-1}\bigl(\# (\red_v K_v^*)\cdot q_v^{-\dim\H}\bigr)}{\mathcal E(\H)}\\
&\gg \tfrac{\mathcal D(\H)}{(\log\vol(Y))^{\ref{log-est}}}J_\infty\prod_{v\in\places_f}{\bigl(\lambda_v|\omega_v|(P_v)\bigr)^{-1}[P^{(1)}_v:(K_v^*)^{(1)}]}&&\text{~\eqref{eq:step1}, \eqref{eq:lower-bd-1}}\\
&\gg\tfrac{\mathcal D(\H)^{1/2}}{(\log\vol(Y))^{\ref{log-est}}}J_\infty{\prod_{v\in\places_f}\bigl(\lambda_v|\omega_v|(K_v^*)\bigr)^{-\ref{dlexp}}}&&\text{\eqref{eq:local-lower}, \eqref{eq:local-lower-ramified} }\\
&\gg \tfrac{\mathcal D(\H)^{\ref{dlexp}}}{(\log\vol(Y))^{\ref{log-est}}}\prod_{v\in\places}\bigl(\lambda_v|\omega_v|(K_v^*)\bigr)^{-\ref{dlexp}}&&\text{$J_\infty, \mathcal D(\H)\gg1$}\\
&\gg \vol(Y)^{\consta}&&\text{\eqref{eq;volume-upperbd}}.
\end{align*}
The proof of the lower bound in~\eqref{eq:disc-vol-statement} is now complete.

\printindex

\bibliographystyle{plain}

\end{document}